\documentclass{amsart}
\usepackage{a4wide}

\usepackage{pgf,tikz,pgfplots}
\pgfplotsset{compat=1.15}
\usepackage{mathrsfs}
\usetikzlibrary{arrows}

\usepackage[T1]{fontenc}
\usepackage{microtype}
\usepackage{lmodern}
\usepackage[colorlinks=true,urlcolor=blue, citecolor=red,linkcolor=blue,linktocpage,pdfpagelabels, bookmarksnumbered,bookmarksopen]{hyperref}
\usepackage[hyperpageref]{backref}
\usepackage{amsthm} 
\usepackage{latexsym,amsmath,amssymb}

\usepackage{bbm}

\usepackage{accents}
\usepackage{esint}

\usepackage{soul}
\usepackage{mathtools} 
\usepackage{xparse} 
\usepackage[capitalize]{cleveref}

\usepackage[shortlabels]{enumitem}

\title{On well-posedness of the $s$-Schr\"odinger maps in the subcritical regime}
\author{Ahmed Dughayshim}
\address{Department of Mathematics,
University of Pittsburgh,
417 Thackeray Hall,
Pittsburgh, PA 15260, USA.}
\email{aha80@pitt.edu}

\newcommand{\N}{{\mathbb N}}

\renewcommand{\S}{{\mathbb S}}

\newtheorem{theorem}{Theorem}
\newtheorem{lemma}[theorem]{Lemma}
\newtheorem{corollary}[theorem]{Corollary}
\newtheorem{proposition}[theorem]{Proposition}

\theoremstyle{definition}

\theoremstyle{remark}
\newtheorem{remark}[theorem]{Remark}

\newcommand\supp{{\rm supp\,}}

\newcommand{\ind}{{\mathbbm{1}} }

\newcommand{\R}{\mathbb{R}}
\newcommand{\C}{\mathbb{C}}

\newcommand{\brac}[1]{\left (#1 \right )}

\newcommand{\abs}[1]{\left\lvert #1 \right \rvert}

\renewcommand{\i}{{\rm \bf i}}

\newcommand{\barint}{
\rule[.036in]{.12in}{.009in}\kern-.16in \displaystyle\int }

\newcommand{\barcal}{\text{$ \rule[.036in]{.11in}{.007in}\kern-.128in\int $}}

\def\mvint_#1{\mathchoice
          {\mathop{\vrule width 6pt height 3 pt depth -2.5pt
                  \kern -8pt \intop}\nolimits_{\kern -3pt #1}}%
          {\mathop{\vrule width 5pt height 3 pt depth -2.6pt
                  \kern -6pt \intop}\nolimits_{#1}}%
          {\mathop{\vrule width 5pt height 3 pt depth -2.6pt
                  \kern -6pt \intop}\nolimits_{#1}}%
          {\mathop{\vrule width 5pt height 3 pt depth -2.6pt
                  \kern -6pt \intop}\nolimits_{#1}}}

\numberwithin{theorem}{section} \numberwithin{equation}{section}

\newcommand{\aleq}{\lesssim}
\newcommand{\ageq}{\gtrsim}
\newcommand{\aeq}{\approx}

\newcommand{\Ds}[1]{|\nabla|^{#1}}

\usepackage{scalerel}[2014/03/10]
\usepackage[usestackEOL]{stackengine}
\def\avint{\,\ThisStyle{\ensurestackMath{%
			\stackinset{c}{.2\LMpt}{c}{.5\LMpt}{\SavedStyle-}{\SavedStyle\phantom{\int}}}%
		\setbox0=\hbox{$\SavedStyle\int\,$}\kern-\wd0}\int}

\let\latexchi\chi
\makeatletter
\renewcommand\chi{\@ifnextchar_\sub@chi\latexchi}
\newcommand{\sub@chi}[2]{
  \@ifnextchar^{\subsup@chi{#2}}{\latexchi^{}_{#2}}%
}
\newcommand{\subsup@chi}[3]{
  \latexchi_{#1}^{#3}%
}
\makeatother

\begin{document}

\begin{abstract}
We study well-posedness of the $s$-Schr\"odinger map equation 
\[\begin{cases}
 \partial_{t} u =- u \wedge (-\Delta)^s u \mbox{    } \mbox{ on } \R^{n} \times [-1,1] \\
 u(x,0) = u_{0}
\end{cases}\]
in dimension $n \geq 3$ in the subcritical regime, more precisely we establish a local well-posedness result when the initial data is $u_{0} \in B^{\sigma}_{Q}$ with $ \sigma \geq \frac{n+1}{2}$ and $ \Vert u_{0} \Vert_{B^{\sigma}_{Q}} \ll 1.$
\end{abstract}
\maketitle
\tableofcontents
\section{Introduction}
We are interested in studying the following initial value problem
\begin{equation}\label{SSchrodingereq}
\begin{cases}
 \partial_{t} u = - u \wedge (-\Delta)^s u \mbox{    } \mbox{ on } \R^{n} \times [-1,1] \\
 u(x,0) = u_{0}
\end{cases}
\end{equation}
where $n \geq 3$, $u : \R^{n+1} \to \S^{2}$ , $s \in (1/2,1)$, $ \mathcal{F} ( (-\Delta)^s u)(\xi,\tau) :=  \vert \xi \vert^{2s} \hat{u}(\xi,\tau)$, and the initial data $u_{0}$ is small in an appropriate Besov/Sobolev space. 

This equation ``interpolates'' between two equations studied in the literature. For $s=1$ this is the Schr\"odinger map equation. A global well-posedness result for $n \geq 2$ and small initial data in the critical Sobolev space $\dot{H}^{n/2}$ was proven by Bejenaru, Ionescu, Kenig, and Tataru in \cite{AnnalesKI}. 

For $s=1/2$ it is the so-called halfwave map equation whose study was initiated by Krieger--Sire \cite{halfwav} and Lenzmann--Schikorra \cite{LenzSchik}. 
For $ s =\frac{1}{2}$ a global well posedness result for $n \geq 5$ was proven by Krieger and Sire in \cite{halfwav}, and later improved to $n \geq 4$ by Krieger and Kiesenhofer in \cite{halfwav2}. For the case $ n=3$ , and $ s=1/2$ uniqueness of solutions was proven by Eyeson, Farina-Reyes, and Schikorra in \cite{Uniq}, and weakly well-posed solutions were shown to exist by Marsden in \cite{marsden2024}. A general well-posedness result in the case $ n =3$, $ s =\frac{1}{2}$ seems to be still open. Lastly, for $ n =1 $ and $ s= 1/2$, a global well-posedness result was proven recently by Gérard and Lenzmann in \cite{halfn1}.  For further references on the borderline cases $s \in \{\frac{1}{2},1 \}$ see \cite{lenzsurvey}, \cite{Liu_2023}, \cite{SurveySch}.

In \cite{model} the study of \eqref{SSchrodingereq} was proposed, and a model equation in the critical case was analyzed.

Here we are interested in the geometric equation \eqref{SSchrodingereq}, but in the subcritical regime, and for all $ s \in (1/2,1)$. 

To state our result we fix the following notation; for $ Q \in \S^{2}$ and $ \sigma \geq \frac{n+1}{2}$ we define the Besov space $ B^{\sigma}_{Q}$ as follows
$$
B^{\sigma}_{Q} : = \{ f : \R^{n} \to \S^{2} : f_{l} - Q_{l} \in B^{\sigma}_{2,1} \text{ for } l =1,2,3 \}
$$
where $B^{\sigma}_{2,1}$ is the inhomogeneous Besov space. We give $B^{\sigma}_{Q}$ the metric
$$
d^{\sigma}_{Q} (f,g) : = \sum_{l=1}^{3} \Vert f_{l} -g_{l} \Vert_{B^{\sigma}_{2,1}}.
$$
We will also use the following notation; for $ f \in B^{\sigma}_{Q}$ set
$$
\Vert f \Vert_{B^{\sigma}_{Q}}:= d^{\sigma}_{Q} (f, Q).
$$
The main result is the following
\begin{theorem}\label{MainResult}
Let $ n \geq 3$, $ s\in (1/2,1)$ then for any $ \sigma_{0} \geq \frac{n+1}{2}$ there exists $ \epsilon_{0} := \epsilon_{0}(n,s, \sigma_{0}) > 0$ so that the following holds 
\begin{enumerate}
    \item For any $ Q \in \S^{2}$ and $u_{0} \in B^{\sigma_{0} +2s}_{Q}$ with $ \Vert u_{0} \Vert_{B^{\sigma_{0}}_{Q}} \leq \epsilon_{0} $ there exists $u \in C( [-1,1] :B^{\sigma_{0}+2s}_{Q})$ that solves \eqref{SSchrodingereq}. And satisfy the estimate 
    $$
    \sup_{t \in [-1,1]} \Vert u \Vert_{B^{\sigma}_{Q}} \lesssim_{ \sigma_{0}} \Vert u_{0} \Vert_{B^{\sigma}_{Q}}
    $$
    for any $ \sigma \in [ \frac{n+1}{2} , \sigma_{0} +2s]$.
\item  If $ v : \R^{n+1} \to \S^{2}$ and $ u : \R^{n+1} \to \S^{2}$ are the solutions obtained in (1) with initial data $ v_{0},u_{0}$ respectively, then 
\begin{equation*}
\sup_{t \in [-1,1]} d^{\sigma_{0}}_{Q} (u,v) \lesssim_{\sigma_{0}} d_{Q}^{\sigma_{0}}(u_{0},v_{0}).
\end{equation*}
And hence the solution map extends uniquely to a Lipschitz map on the set $\{ u_{0} \in B^{\sigma_{0}}_{Q} : \Vert u_{0} \Vert_{B^{\sigma_{0}}_{Q}}  \leq \epsilon_{0} \}$. 
\item Let $ \sigma \in [ \sigma_{0}, \infty)$ then there exists $ \epsilon(\sigma) \in (0, \epsilon_{0}]$ so that for any $R > 0$, $ \sigma' \in [ \sigma_{0},\sigma]$ and for any $ u_{0} , v_{0} \in B^{\sigma_{0}+2s}_{Q}$  with 
$$
\Vert u_{0}  \Vert_{B^{\sigma_{0}}_{Q}}, \Vert v_{0} \Vert_{B^{\sigma_{0}}_{Q}} \leq \epsilon(\sigma) \text{ and } \Vert u_{0}  \Vert_{B^{\sigma'}_{Q}}, \Vert v_{0} \Vert_{B^{\sigma'}_{Q}} \leq R,
$$ the solutions $ u ,v $   obtained in (1) satisfy the following estimate
$$
\sup_{t \in [-1,1]} d^{\sigma'}_{Q}(u,v) \lesssim_{R, \sigma'} d^{\sigma'}_{Q}( u_{0} , v_{0} ).
$$
\end{enumerate}

\end{theorem}

The proof of Theorem \ref{MainResult} takes great inspiration from the work of Ionescu and Kenig in \cite{DIE}, \cite{cmp}. Namely, using the stereographic projection, we reduce Theorem \ref{MainResult} to a Schr\"{o}dinger equation with a nonlinearity that contains derivatives. Then we use smoothing estimates to prove the nonlinear estimates. The reason why we are unable to use this method to push the result to the range $ \sigma \geq \frac{n}{2}$ is due to the absence of a null structure in the nonlinearity. In the case $s=1$, this null structure was first observed by Bejenaru in \cite{Null}. Finding a null structure is quite a challenge even for $s=\frac{1}{2}$ let alone for $s \in (1/2,1)$. However, a compensation structure, reminiscent of the effect of a null structure, will be used for the case $ \sigma = \frac{n}{2}$ to obtain a priori estimates in dimension $n \geq 4$, \cite{AhmedArminFuture}, which combined with \Cref{MainResult} leads to well-posedness results.

\subsection*{Acknowledgment}
Discussions with A. Schikorra are gratefully acknowledged. The author was partially funded by the Northern Borders university's scholarship.

\subsection*{Plan of paper} In section 2 we reduce \eqref{SSchrodingereq} to a Schr\"{o}dinger equation using the stereographic projection. In section 3 we introduce the resolution spaces and prove technical lemmata that will be used in section 4 to prove the smoothing and maximal estimates. In section 5 we prove commutator estimates that will be used extensively to prove the nonlinear estimates. In section 6 we prove that our main resolution space is an algebra. In section 7 we prove the nonlinear estimates. Finally, in section 9 we prove Theorem \ref{Redeq} which will imply Theorem \ref{MainResult}.

\section{Reduction to a Schr\"{o}dinger equation}
By rotation invariance, it suffices to prove Theorem \ref{MainResult} for $Q = (0,0,1)$. Therefore, 
throughout this section let $ Q=(0,0,1)$ and let $ u : \R^{n+1} \to \mathbb{S}^{2} $ be in the space $ C ([-1,1]: B_{Q}^{\frac{n+1}{2} +2s})$ such that $ u(\cdot,t)$ is small in Besov norm $ B^{\frac{n+1}{2}}_{Q}$ for every $ t\in [-1,1]$. Let $L : \S \to \C$ be the stereographic projection, i.e
$$
L(v) = \frac{ v_{1} + \i v_{2}}{1 + v_{3}}
$$
and set $ f : = L(u)$. The goal is as follows: if $u $ solves the $ s- $Schr\"{o}dinger map, then $ f$ solves a corresponding scalar PDE. It turns out $f$ solves an equation of the form
$$
(\i \partial_{t} - (-\Delta)^{s} )f = \mathcal{N}(f)
$$
where the nonlinearity contains distributed derivatives of $f$, then if we prove well-posedness of this PDE we can transfer the result back to $u$ using the stereographic projection again. The first step is to calculate the corresponding PDE concerning $f$. This will be a sequence of lemmata. The computations of the following lemma are elementary and can be found in the appendix.
\begin{lemma}\label{stereographic stuff}
Let $ L^{-1} : \C \to \S^{2}
$ be the inverse of the stereographic projection, i.e for $ z = z_{1} + \i z_{2} $ 
$$
L^{-1}(z) = \frac{1}{1+|z|^{2}} [ 2z_{1} , 2z_{2} , 1- |z|^{2} ].
$$
 Then we have
\begin{enumerate}
    \item 
    $$
    \partial_{z_{1}} L^{-1}(z) = \frac{2}{(1+|z|^{2})^{2}} [ (1+ \vert z \vert^{2} - 2 z^{2}_{1}, -2 z_{1} z_{2} , -2z_{1}]
    $$
    \item 
    $$
    \partial_{z_{2}} L^{-1}(z) = \frac{2}{(1+ |z|^{2})^{2}} [ -2 z_{1} z_{2} , (1 + |z|^{2} ) - 2 z_{2}^{2} , -2z_{2}]
    $$
    \item 
    $$
    | \partial_{z_{1}} L^{-1}(z) | = \frac{2}{1+|z|^{2}} = | \partial_{z_{2}} L^{-1}(z) |
    $$
    \item 
    $$
    \langle \partial_{z_{1}}L^{-1}(z), \partial_{z_{2}} L^{-1}(z) \rangle = 0
    $$
    \item 
    $$
    L^{-1}(z) \wedge \partial_{z_1} L^{-1} (z) = \partial_{z_{2}} L^{-1}(z) 
    $$
    \item 
    $$
    L^{-1}(z) \wedge \partial_{z_{2}} L^{-1}(z) = - \partial_{z_{1}} L^{-1}(z)
    $$
\end{enumerate}
\end{lemma}

Next, fix $ u : \R^{n+1} \to \S^{2}$ as above, and define $ f : = L(u)$ and hence $ u = L^{-1}(f)$. Then for $i \in \{1,2 \}$ define $ e_{i} $ as 
$$
e_{i} : =  \frac{\partial_{f_{i}} L^{-1}(f)}{ |\partial_{f_{i}} L^{-1}(f)|}.
$$
Then by the chain rule and Lemma \ref{stereographic stuff} we obtain
\begin{equation}\label{prepde}
\partial_{t} f = \frac{1+|f|^{2}}{2} \left( \langle e_{1} , \partial_{t} u \rangle + \i \langle e_{2} , \partial_{t} u \rangle \right)
\end{equation}
then using \eqref{prepde} and Lemma \ref{stereographic stuff}, it is clear that $u$ solves 
$$
\partial_{t} u =- u \wedge (-\Delta)^su
$$
if and only if $f$ solves the PDE
\begin{equation}\label{pdew1}
\partial_{t} f =  \frac{1 + |f|^{2}}{2} \left(  \langle e_{2} , (-\Delta)^s u \rangle - \i \langle e_{1} , (-\Delta)^s u \rangle \right).
\end{equation}
We need to write the right hand side of \eqref{pdew1} entirely in terms of $f$. This can be done because both $ e$ and $u$ are functions of $f$. To that end, we calculate 
\begin{equation}\label{dt1}
\begin{split}
& \frac{1+|f|^{2}}{2} \langle e_{2}, (-\Delta)^s u \rangle
\\
& = \frac{1}{2} [ -2 f_{1} f_{2}, (1 + |f|^{2}) - 2 f_{2}^{2}, -2 f_{2}] \\
& \cdot [ (-\Delta)^s ( \frac{2 f_{1}}{ 1+ |f|^{2}}) , (-\Delta)^s( \frac{ 2 f_{2}}{ 1+ |f|^{2}}) , (-\Delta)^s ( \frac{ 1 - |f|^{2}}{ 1+ |f|^{2}})]
\\
& = - f_{1} f_{2} (-\Delta)^s ( \frac{2 f_{1}}{1+ | f|^{2}} )
\\
& + (1 + |f|^{2}) (-\Delta)^s ( \frac{ f_{2}}{1+|f|^{2}}) - f_{2}^{2} (-\Delta)^s( \frac{2 f_{2}}{1+|f|^{2}} ) 
\\
& - f_{2} (-\Delta)^s( \frac{1 - |f|^{2}}{1+ |f|^{2}} ).
\end{split}
\end{equation}
Similarly, 
\begin{equation}\label{dt2}
\begin{split}
& \i \frac{1 + |f|^{2}}{2} \langle e_{1}, (-\Delta)^s u \rangle 
\\
& = \frac{\i}{2} \big[ 1 + |f|^{2} - 2 f^{2}_{1} , -2 f_{1} f_{2} , -2 f_{1} ]
\\
& \cdot [ (-\Delta)^s ( \frac{ 2 f_{1}}{1 + |f|^{2}}) , (-\Delta)^s ( \frac{2 f_{2}}{1+|f|^{2}}), (-\Delta)^s ( \frac{ 1 - |f|^{2}}{1+|f|^{2}}) ]
\\
& = \i ( 1+ |f|^{2} ) (-\Delta)^s ( \frac{ f_{1}}{ 1+ |f|^{2}}) - \i f^{2}_{1} (-\Delta)^s ( \frac{ 2 f_{1}}{1+ |f|^{2}} ) 
\\
& - \i f_{1} f_{2} (-\Delta)^s ( \frac{ 2 f_{2}}{1+|f|^{2}}) 
\\
& - \i f_{1} (-\Delta)^s ( \frac{1 - |f|^{2}}{1+|f|^{2}}).
\end{split}
\end{equation}
Now we calculate \eqref{dt1} $-$\eqref{dt2} to get 
\begin{equation}\label{dt3}
\begin{split}
& -\partial_{t} f
\\
& =   f_{1} f_{2} (-\Delta)^s ( \frac{2 f_{1}}{1+ | f|^{2}} )
\\
& -(1 + |f|^{2}) (-\Delta)^s ( \frac{ f_{2}}{1+|f|^{2}}) + f_{2}^{2} (-\Delta)^s( \frac{2 f_{2}}{1+|f|^{2}} ) 
\\
& + f_{2} (-\Delta)^s( \frac{1 - |f|^{2}}{1+ |f|^{2}} )
\\
& + 
\\
&  \i ( 1+ |f|^{2} ) (-\Delta)^s ( \frac{ f_{1}}{ 1+ |f|^{2}}) - \i f^{2}_{1} (-\Delta)^s ( \frac{ 2 f_{1}}{1+ |f|^{2}} ) 
\\
& - \i f_{1} f_{2} (-\Delta)^s ( \frac{ 2 f_{2}}{1+|f|^{2}}) 
\\
& - \i f_{1} (-\Delta)^s ( \frac{1 - |f|^{2}}{1+|f|^{2}})
\end{split}
\end{equation}
collect similar terms to obtain
\begin{equation}\label{dt4}
\begin{split}
& -\partial_{t} f
\\
& = (1+|f|^{2}) (-\Delta)^s ( \frac{ \i f }{1+|f|^{2}})
 - \i f (-\Delta)^{s } ( \frac{ 1 - |f|^{2}}{1+|f|^{2}})
 \\
 &
+ f_{1} f_{2} (-\Delta)^s ( \frac{ 2 \bar{f}}{1+|f|^{2}}) + f_{2}^{2} (-\Delta)^s ( \frac{ 2 f_{2}}{1+|f|^{2}} )- \i f_{1}^{2} (-\Delta)^s ( \frac{2 f_{1}}{1+|f|^{2}} ).
\end{split}
\end{equation}
The terms in the first line are of the form we want ( i.e written in terms of $f$ or $ \bar{f}$). we need to rewrite the terms appearing in the second line. To that end, define 
$$
\beta : = (-\Delta)^s ( \frac{2 f}{ 1+|f|^{2}}),
$$
then we can simplify each term in the second line of \eqref{dt4} as follows
\begin{equation}\label{dt5}
\begin{split}
& f_{1} f_{2} (-\Delta)^s ( \frac{ 2 \bar{f}}{1+|f|^{2}})
\\
& = f_{1} f_{2} \bar{\beta} = \frac{(f + \bar{f})( f - \bar{f})}{4 \i} \bar{\beta}
\\
& = - \frac{ \i}{4} ( f^{2} - \bar{f}^{2}) \bar{\beta}.
\end{split}
\end{equation}
Similarly, we deal with the remaining two terms in the last line of \eqref{dt4}
\begin{equation}\label{dt6}
\begin{split}
& f_{2}^{2} (-\Delta)^s ( \frac{ 2 f_{2}}{1+|f|^{2}} )- \i f_{1}^{2} (-\Delta)^s ( \frac{2 f_{1}}{1+|f|^{2}} )
\\
& = f_{2}^{2} \beta_{2} - \i f_{1}^{2} \beta_{1}
\\
& = \left( \frac{ f  - \bar{f}}{2\i} \right)^{2} \frac{( \beta - \bar{\beta})}{2 \i} - \i \left( \frac{ f + \bar{f}}{2} \right)^{2} \frac{ ( \beta + \bar{\beta})}{2} 
\\
& = \frac{ \i }{8} ( f^{2} + \bar{f}^{2} - 2 f \bar{f}) ( \beta - \bar{\beta}) - \frac{\i}{8} ( f^{2} + \bar{f}^{2} + 2 f \bar{f} )( \beta + \bar{\beta})
\\
& = \frac{ -\i}{2} f \bar{f} \beta - \frac{ \i }{4} ( f^{2} + \bar{f}^{2}) \bar{\beta}.
\end{split}
\end{equation}
Combine \eqref{dt5} and \eqref{dt6} to obtain 
\begin{equation}\label{dt7}
\begin{split}
& f_{1} f_{2} (-\Delta)^s ( \frac{ 2 \bar{f}}{1+|f|^{2}}) + f_{2}^{2} (-\Delta)^s ( \frac{ 2 f_{2}}{1+|f|^{2}} )- \i f_{1}^{2} (-\Delta)^s ( \frac{2 f_{1}}{1+|f|^{2}} )
\\
& = - \frac{ \i}{2} |f|^{2} \beta - \frac{\i}{2} f^{2} \bar{\beta}
\\
& = - \i |f|^{2} (-\Delta)^s( \frac{ f}{1+|f|^{2}}) - \i f^{2} (-\Delta)^s ( \frac{ \bar{f}}{1+ |f|^{2}}).
\end{split}
\end{equation}
Plug \eqref{dt7} back into \eqref{dt4} to obtain
\begin{equation}\label{dt8}
\begin{split}
& -\partial_{t} f 
\\
& = (1+|f|^{2}) (-\Delta)^s ( \frac{ \i f }{1+|f|^{2}})
 - \i f (-\Delta)^{s } ( \frac{ 1 - |f|^{2}}{1+|f|^{2}})
 \\
 &- \i |f|^{2} (-\Delta)^s( \frac{ f}{1+|f|^{2}}) - \i f^{2} (-\Delta)^s ( \frac{ \bar{f}}{1+ |f|^{2}})
 \\
 & =  (-\Delta)^s ( \frac{ \i f }{1+|f|^{2}})
 - \i f (-\Delta)^{s } ( \frac{ 1 - |f|^{2}}{1+|f|^{2}})
- \i f^{2} (-\Delta)^s ( \frac{ \bar{f}}{1+ |f|^{2}}).
 \end{split}
\end{equation}
Now we apply the fractional Leibniz rule to each term in the right hand side of \eqref{dt8}
\begin{equation}\label{dtt1}
\begin{split}
& \i (-\Delta)^s( \frac{ f}{1+|f|^{2}} ) 
\\
& =  \i \frac{1}{1+|f|^{2}} (-\Delta)^s(f) + f \i (-\Delta)^s ( \frac{1}{1+|f|^{2}})  + \i H_{s} ( f, \frac{1}{1+|f|^{2}})
\end{split}
\end{equation}
We do the same for the second term on the right hand side of \eqref{dt8}
\begin{equation}\label{dtt2}
\begin{split}
& - \i f (-\Delta)^s( \frac{ 1 - | f|^{2}}{1+|f|^{2}})
\\
& = \frac{ \i f}{1+|f|^{2}} (-\Delta)^s ( |f|^{2}) - \i f ( 1 - |f|^{2}) (-\Delta)^s ( \frac{ 1}{ 1+ |f|^{2}}) - \i f H_{s} (1- |f|^{2} , \frac{ 1}{1+|f|^{2}}) 
\\
& = \i \frac{f. \bar{f}}{1+|f|^{2}} (-\Delta)^s(f) + \i \frac{ f^{2}}{1+ |f|^{2}} (-\Delta)^s(\bar{f}) + \frac{f}{1+|f|^{2}}\i H_{s} ( f,\bar{f})
\\
& - \i f ( 1 - |f|^{2}) (-\Delta)^s ( \frac{ 1}{ 1+ |f|^{2}}) - \i f H_{s} (1- |f|^{2} , \frac{ 1}{1+|f|^{2}}).
\end{split}
\end{equation}
We do the same for the last term on the right hand side of \eqref{dt8}
\begin{equation}\label{dtt3}
\begin{split}
& - \i f^{2} (-\Delta)^s ( \frac{ \bar{f}}{1+ |f|^{2}})
\\
& = - \i f^{2} \bar{f} (-\Delta)^s ( \frac{1}{1+|f|^{2}}) - \i \frac{f^{2}}{1+|f|^{2}} (-\Delta)^s( \bar{f}) - \i f^{2} H_{s} ( \bar{f}, \frac{1}{1+|f|^{2}}).
\end{split}
\end{equation}
Plug \eqref{dtt1},\eqref{dtt2}, and \eqref{dtt3} back into \eqref{dt8} to obtain
\begin{equation*}
\begin{split}
& -\partial_{t} f
\\
& = \i \frac{1}{1+|f|^{2}} (-\Delta)^s(f) + \i f (-\Delta)^s ( \frac{1}{1+|f|^{2}})  + \i H_{s} ( f, \frac{1}{1+|f|^{2}})
\\
& +
\\
& \i \frac{f. \bar{f}}{1+|f|^{2}} (-\Delta)^s(f) + \i \frac{ f^{2}}{1+ |f|^{2}} (-\Delta)^s(\bar{f}) + \frac{f}{1+|f|^{2}}\i H_{s} ( f,\bar{f})
\\
& - \i f ( 1 - |f|^{2}) (-\Delta)^s ( \frac{ 1}{ 1+ |f|^{2}}) - \i f H_{s} (1- |f|^{2} , \frac{ 1}{1+|f|^{2}})
\\
& +
\\
&  - \i f^{2} \bar{f} (-\Delta)^s ( \frac{1}{1+|f|^{2}}) - \i \frac{f^{2}}{1+|f|^{2}} (-\Delta)^s( \bar{f}) - \i f^{2} H_{s} ( \bar{f}, \frac{1}{1+|f|^{2}}).
\end{split}
\end{equation*}
This yields 
\begin{equation}\label{Sss}
\begin{split}
    & -\partial_{t} f = \i (-\Delta)^s(f) + \i H_{s}(f, \frac{1}{1+|f|^{2}}) 
    \\
    & + \i \frac{f}{1+|f|^{2}} H_{s}(f,\bar{f})- \i f H_{s}(1- |f|^{2},\frac{1}{1+|f|^{2}})
    \\
    & - \i f^{2} H_{s} ( \bar{f},\frac{1}{1+|f|^{2}}).
\end{split}
\end{equation}
Therefore, $u$ solves the $s$-Schr\"{o}dinger map, \eqref{SSchrodingereq}, if and only if $f = L(u)$ solves 
\begin{equation}\label{reductionmainold}
\begin{split}
&\i \partial_{t} f - (-\Delta)^s f
\\
&= H_{s} (f, \frac{1}{1+\vert f \vert^{2}} )+\frac{f}{1+\vert f \vert^{2}} H_{s} (f,\bar{f}) + fH_{s} ( \vert f \vert^{2} , \frac{1}{1+\vert f \vert^{2}} ) - f^{2} H_{s} ( \bar{f} , \frac{1}{1+\vert f\vert^{2}}).
\end{split}
\end{equation}
or equivalently, 
\begin{equation}\label{reductionmain}
\begin{split}
& \i \partial_{t} f - ( - \Delta)^{s} f = H_{ {s}} ( f, \frac{1}{1+|f|^{2}}) +  f H_{{s}} ( f, \frac{ \bar{f}}{1+|f|^{2}} ).
\end{split}
\end{equation}
The equivalence between \eqref{reductionmainold} and \eqref{reductionmain} can be seen by expanding the right hand side of \eqref{reductionmainold}. \\
The following lemma will imply that it is enough to study \eqref{reductionmain} ( under the assumption that $u$ is small in $B^{\sigma}_{Q} )$ for $ \sigma \geq \frac{n+1}{2}$
\begin{lemma}\label{Reduction2}
Let $ \sigma_{0} \geq \frac{n+1}{2}$ then the following holds
\begin{enumerate}
    \item Let $ f,g : \R^{n} \to \C $ and $ f,g \in B^{\sigma_{0}}_{2,1}$ with $ \Vert f \Vert_{B^{\sigma_{0}}_{2,1}} , \Vert g \Vert_{B^{\sigma_{0}}_{2,1}} \leq \epsilon \ll 1$. Then we have for any $ \sigma \geq \sigma_{0}$
\begin{equation}\label{wad}
d^{\sigma}_{Q} ( L^{-1} (f) , L^{-1} (g) ) \lesssim C( \sigma , \Vert f \Vert_{B^{\sigma}_{2,1}} , \Vert g \Vert_{B^{\sigma}_{2,1}} ) \Vert f- g \Vert_{B^{\sigma}_{2,1}}.
\end{equation}

\item If  $ u,v \in B_{Q}^{\sigma_{0}} $ with $ \Vert u \Vert_{B_{Q}^{\sigma_{0}}} , \Vert v \Vert_{B_{Q}^{\sigma_{0}}} \leq \epsilon \ll 1$ then for any $ \sigma \geq \sigma_{0}$ we have
\begin{equation}\label{wad2}
\Vert L (u) - L(v) \Vert_{B^{\sigma}_{2,1}} \lesssim C(\sigma ,\Vert u \Vert_{B^{\sigma}_{Q}} , \Vert v \Vert_{B^{\sigma}_{Q}} ) d^{\sigma}_{Q}(u,v).
\end{equation}
\item If  $ f : \R^{n} \to \C$ satisfy $ \Vert f \Vert_{B^{\sigma_{0}}_{2,1}} \leq \epsilon \ll 1.$ Then there exists $ C: = C(n,\sigma_{0} )$ such that 
\begin{equation}\label{sfg}
\Vert L^{-1}(f) \Vert_{B^{\sigma_{0}}_{Q}} \leq C \epsilon.
\end{equation}

\item If  $u \in B^{\sigma_{0}}_{Q}$ with $ \Vert u \Vert_{B^{\sigma_{0}}_{Q}} \leq \epsilon \ll 1$. Then there exists $ C:= C(n,\sigma_{0})$ so that
\begin{equation}\label{sfg2}
    \Vert L(u) \Vert_{B^{\sigma_{0}}_{2,1}} \leq C \epsilon.
\end{equation}

\end{enumerate}

\end{lemma}
\begin{proof}
This is a consequence of the following three facts regarding Besov spaces:
\begin{enumerate}
    \item There exists $C : = C(n)$ so that 
\begin{equation}\label{Besov1}
\Vert f g \Vert_{B_{2,1}^{\sigma}} \leq \frac{ C^{\sigma +1}}{\sigma } \left( \Vert f \Vert_{B^{\sigma}_{2,1}} \Vert g \Vert_{B^{\frac{n}{2}}_{2,1}} +   \Vert f \Vert_{B^{\frac{n}{2}}_{2,1}} \Vert g \Vert_{B^{\sigma}_{2,1}} \right).
\end{equation}
\item If $ F: \R \to \R $ is smooth function with $ F(0) = 0 =F'(0)$ then for any real valued $ f \in B^{\sigma}_{2,1}$ with $ \Vert f \Vert_{B^{n/2}_{2,1}} \ll 1$ we have 
\begin{equation}\label{Besov2}
\Vert F (f ) \Vert_{B^{\sigma}_{2,1}} \leq C(n,F',\sigma) \Vert f \Vert_{B^{\sigma}_{2,1}}.
\end{equation}
\item If $F$ is as in (2) then for any real valued $ f,g \in B^{\sigma}_{2,1}$ with $ \Vert f \Vert_{B^{n/2}_{2,1}}, \Vert g \Vert_{B^{n/2}_{2,1}} \ll 1$  we have 
\begin{equation}\label{Besov3}
\Vert F(f) - F(g) \Vert_{B^{\sigma}_{2,1}} \leq C ( n, F'', \sigma, \Vert f \Vert_{B^{\sigma}_{2,1}} , \Vert g \Vert_{B^{\sigma}_{2,1}}) \Vert f - g \Vert_{B^{\sigma}_{2,1}}.
\end{equation}
\end{enumerate}
These  inequalities are well-known. See for example \cite[Section 2.8]{Bahouri}. 
Using these inequalities, we start proving \eqref{wad}.
We need to prove the bound 
\begin{equation}\label{tdf}
\Vert ( L^{-1}(f) )_{l} - ( L^{-1}(g) )_{l} \Vert_{B^{\sigma}_{2,1}} \lesssim \Vert f -g \Vert_{B^{\sigma}_{2,1}}
\end{equation}
for $l =1,2,3$. We start with $l=1$, by definition of $L^{-1}$ we need to estimate
$$
\Vert \frac{f_{1}}{1+\vert f \vert^{2}} - \frac{g_{1}}{ 1+ \vert g \vert^{2}} \Vert_{B^{\sigma}_{2,1}}.
$$
To that end, 
\begin{equation}\label{annoyingstuff1}
\begin{split}
& \Vert  \frac{f_{1}}{1+\vert f \vert^{2}} - \frac{g_{1}}{ 1+ \vert g \vert^{2}} \Vert_{B^{\sigma}_{2,1}}
\\
& \lesssim \Vert \frac{f_{1} -g_{1}}{1+\vert f \vert^{2}} \Vert_{B^{\sigma}_{2,1}} + \Vert g_{1} (\frac{1}{1+ \vert f \vert^{2}} - \frac{1}{1+\vert g \vert^{2}}) \Vert_{B^{\sigma}_{2,1}}.
\end{split}
\end{equation}
For the first term on the right hand side of \eqref{annoyingstuff1} 
\begin{equation}\label{annoying2}
\begin{split}
& \Vert \frac{f_{1} -g_{1}}{1+\vert f \vert^{2}} \Vert_{B^{\sigma}_{2,1}}
\\
& \lesssim \Vert f_{1} - g_{1}  \Vert_{B^{\sigma}_{2,1}} + \Vert (g_{1} - f_{1} ) (\frac{1}{1+\vert f \vert^{2}} - 1 ) \Vert_{B^{\sigma}_{2,1}}. 
\end{split}
\end{equation}
For the second term on the right hand side \eqref{annoying2} we define $ F (r) : = \frac{1}{1+r {}} -1$ for $ r \geq 0$, and then use \eqref{Besov1}, \eqref{Besov2} to obtain 
\begin{equation*}
\begin{split}
&  \Vert (g_{1} - f_{1} ) (\frac{1}{1+\vert f \vert^{2}} - 1 ) \Vert_{B^{\sigma}_{2,1}} 
\\
& = \Vert ( g_{1} - f_{1} ) F(\vert f \vert^{2}) \Vert_{B^{\sigma}_{2,1}}
\\
& \lesssim  C( \sigma,  \Vert f \Vert_{B^{\sigma}_{2,1}}) \Vert f -g \Vert_{B^{\sigma}_{2,1}}.  
\end{split}
\end{equation*}
This handles the first term on the right hand side of \eqref{annoyingstuff1}. We handle the second term similarly, namely using \eqref{Besov3} and \eqref{Besov1} we obtain 
\begin{equation*}
\begin{split}
& \Vert g_{1} (\frac{1}{1+ \vert f \vert^{2}} - \frac{1}{1+\vert g \vert^{2}}) \Vert_{B^{\sigma}_{2,1}}
\\
& \lesssim C(\sigma,  \Vert f \Vert_{B^{\sigma}_{2,1}} , \Vert g \Vert_{B^{\sigma}_{2,1}} ) \Vert f - g \Vert_{B^{\sigma}_{2,1}}.
\end{split}
\end{equation*}
This completes the proof of \eqref{tdf} for $l =1$. For $ l \in \{2,3 \}$ the same argument give \eqref{tdf}. Namely we only use \eqref{Besov1},\eqref{Besov2}, and \eqref{Besov3}.

We proceed to proving \eqref{wad2}.  As before, we need to prove
\begin{equation}\label{tdf2}
\Vert \frac{u_{1} +\i u_{2}}{1+u_{3}} - \frac{v_{1} + \i v_{2}}{1+v_{3}} \Vert_{B^{\sigma}_{2,1}} \lesssim \sum_{l=1}^{3} \Vert u_{l} - v_{l} \Vert_{B^{\sigma}_{2,1}}.
\end{equation}
To that end, 
\begin{equation}\label{tdf3}
\begin{split}
& \Vert \frac{u_{1} +\i u_{2}}{1+u_{3}} - \frac{v_{1} + \i v_{2}}{1+v_{3}} \Vert_{B^{\sigma}_{2,1}}
\\
& \lesssim \Vert \frac{u_{1} -v_{1}}{1+u_{3}} \Vert_{B^{\sigma}_{2,1}} + \Vert \frac{u_{2} -v_{2}}{1+u_{3}} \Vert_{B^{\sigma}_{2,1}} + \Vert ( v_{1} +\i v_{2} ) \left( \frac{1}{1+u_{3}} - \frac{1}{1+v_{3}} \right) \Vert_{B^{\sigma}_{2,1}}. 
\end{split}
\end{equation}
We estimate each term on the right hand side of \eqref{tdf3} with the right hand side of \eqref{tdf2}. Starting with the first term, notice that since $\Vert u \Vert_{B^{\sigma_{0}}_{Q}} \ll 1$ then this implies that $u_{3} $ lives in a small neighborhood of $ 1$. Namely, we may assume  $ \vert u_{3} - 1 \vert \leq \frac{1}{4}$. Then take $ \psi : \R \to \R$ to be a smooth function which is equal to $ 1$ in the set $ \{ t \in \R : \vert t \vert \leq \frac{1}{2} \}$ and equal to $0 $ in the set $ \{ t \in \R :  \vert t \vert \geq 1 \}$. Then we can write 
$$
\frac{1}{1+ u_{3}} =\frac{1}{2+ ( u_{3} -1)} = \frac{1}{2+ ( u_{3} -1)} \psi( u_{3} -1) 
$$
define $ F(r) : = \frac{1}{2+r} \psi(r) $.
Then using \eqref{Besov1}, \eqref{Besov2} we obtain 
\begin{equation*}
\begin{split}
& \Vert \frac{u_{1} -v_{1}}{1+u_{3}} \Vert_{B^{\sigma}_{2,1}}
\\
& \lesssim \Vert (u_{1} - v_{1} ) F(u_{3} -1) \Vert_{B^{\sigma}_{2,1}}
\\
& \lesssim C( \sigma ,  \Vert u_{3} -1 \Vert_{B^{\sigma}_{2,1}} ) \Vert u_{1} -v_{1} \Vert_{B^{\sigma}_{2,1}}.
\end{split}
\end{equation*}
The second term on the right hand side of \eqref{tdf3} is estimated in exactly the same way. It remains to prove the estimate for the third term on the right hand side \eqref{tdf3}. Using \eqref{Besov1} and \eqref{Besov3} we have
\begin{equation*}
\begin{split}
& \Vert ( v_{1} +\i v_{2} ) \left( \frac{1}{1+u_{3}} - \frac{1}{1+v_{3}} \right) \Vert_{B^{\sigma}_{2,1}} 
\\
& = \Vert ( v_{1} +\i v_{2} ) \left( F(u_{3}-1) - F(v_{3}-1) \right) \Vert_{B^{\sigma}_{2,1}} 
\\
& \lesssim C( \sigma ,  \Vert v \Vert_{B^{\sigma}_{Q}}, \Vert u \Vert_{B^{\sigma}_{Q} } ) \Vert v_{3} - u_{3} \Vert_{B^{\sigma}_{2,1}}.
\end{split}
\end{equation*}
This proves \eqref{wad2}. Next, we prove \eqref{sfg}, we need to prove 
\begin{equation}\label{sfg3}
\begin{split}
& \Vert ( L^{-1}(f) -Q)_{l} \Vert_{B^{\sigma_{0}}_{2,1}} \leq C \epsilon, 
\end{split}
\end{equation}
for $ l =1,2,3.$ Starting with $ l=1$, put $ F(r) = \frac{1}{1+r^{}} -1$ for $r \geq 0$. Then using \eqref{Besov1}, and \eqref{Besov2} we have
\begin{equation}\label{sfg4}
\begin{split}
& \Vert \frac{f_{1}}{ 1+ \vert f \vert^{2}} \Vert_{B^{\sigma_{0}}_{2,1}}
\\
& \leq \Vert f_{1} \Vert_{B^{\sigma_{0}}_{2,1}} + \Vert f_{1}( \frac{1}{1+\vert f \vert^{2} } -1 ) \Vert_{B^{\sigma_{0}}_{2,1}}
\\
& \leq \Vert f \Vert_{B^{\sigma_{0}}_{2,1}} + \Vert f_{1} F(\vert f \vert^{2}) \Vert_{B^{\sigma_{0}}_{2,1}}
\\
& \lesssim \epsilon + C_{\sigma_{0}} \epsilon \Vert \vert f \vert^{2} \Vert_{B^{\sigma_{0}}_{2,1}}
\\
& \lesssim C_{\sigma_{0}} \epsilon.
\end{split}
\end{equation}
This proves \eqref{sfg3} for $ l =1$. For $l =2$ we just replace $ f_{1} $ with $f_{2}$ in \eqref{sfg4} and this will yield \eqref{sfg3} for $l=2$. As for $l=3$ we estimate as follows 
\begin{equation*}
\begin{split}
& \Vert ( L^{-1}(f) -Q)_{3} \Vert_{B^{\sigma_{0}}_{2,1}}
\\
& = 2 \Vert \frac{ \vert f \vert^{2}}{1+ \vert f \vert^{2}} \Vert_{B^{\sigma_{0}}_{2,1}}
\\
& \lesssim \Vert \vert f \vert^{2} \Vert_{B^{\sigma_{0}}_{2,1}} + \Vert \vert f \vert^{2} F( \vert f \vert^{2}) \Vert_{B^{\sigma_{0}}_{2,1}}
\\
& \lesssim C_{\sigma_{0}} \epsilon.
\end{split}
\end{equation*}
Lastly, we prove \eqref{sfg2}. We need to prove 
\begin{equation}\label{sfg5}
\begin{split}
& \Vert \frac{ u_{1} + \i u_{2}}{1+ u_{3}} \Vert_{B^{\sigma_{0}}_{2,1}} \leq C_{\sigma_{0}} \epsilon.
\end{split}
\end{equation}
To that end, put $ F(r) = \frac{1}{2 + r} \psi(r)$ where  $\psi : \R \to \R$ is a Schwartz function equaling $ 1 $ in the set $ \{ r \in \R : | r | \leq \frac{1}{2} \}$ and equals $0$ in the set $ \{ r \in \R : | r  | \geq 1 \}.$ Since $ u $ is small in $ B^{\sigma_{0}}_{Q}$ we may assume $ |u_{3} - 1 | \leq \frac{1}{4}$ everywhere. And hence we estimate by \eqref{Besov1}, and \eqref{Besov2}
\begin{equation}
\begin{split}
& \Vert \frac{ u_{1} + \i u_{3}}{1+ u_{3}} \Vert_{B^{\sigma_{0}}_{2,1}}
\\
& \leq  \Vert (u_{1}+ \i u_{2} ) F( u_{3}-1 ) \Vert_{B^{\sigma_{0}}_{2,1}}
\\
& \leq C_{\sigma_{0}} \Vert u \Vert_{B^{\sigma_{0}}_{Q}}
\\
& \leq C_{\sigma_{0}} \epsilon.
\end{split}
\end{equation}
This concludes the proof of the lemma.
\end{proof}

In light of Lemma \ref{Reduction2}, for Theorem \ref{MainResult} it suffices to prove the following 
\begin{theorem}\label{Redeq}
Let $n \geq 3$, $ s\in (1/2,1)$, and $ \sigma_{0} \geq \frac{n+1}{2}$. Then there exists $ \epsilon_{0} : = \epsilon_{0}(n,s, \sigma_{0}) > 0$ so that for any $f_{0} \in B^{\sigma_{0}}_{2,1}$ with $ \Vert f_{0} \Vert_{B^{\sigma_{0}}_{2,1}} \leq \epsilon_{0} $ there exists $ f \in C( [-1,1] : B^{\sigma_{0}}_{2,1})$ that solves \eqref{reductionmain} with $ f(x,0) = f_{0}(x)$. Further the solution satisfy 
\begin{enumerate}
    \item for any $ \sigma \in [\frac{n+1}{2}, \sigma_{0} +100]$ we have
    $$
    \sup_{t \in [-1,1]} \Vert f \Vert_{B^{\sigma}_{2,1}} \lesssim_{\sigma_{0}} \Vert f_{0} \Vert_{B^{\sigma}_{2,1}}.
    $$
    In particular, if $ f_{0} \in B^{\sigma}_{2,1}$ then $ f \in C ( [-1,1] : B^{\sigma}_{2,1})$.
    \item If $f,g: \R^{n+1} \to \C$ are the solutions obtained in (1) with initial data $f_{0},g_{0} $ respectively, then we have the bound 
    $$
    \sup_{t \in [-1,1]} \Vert f -g \Vert_{B^{\sigma_{0}}_{2,1}} \lesssim_{\sigma_{0}} \Vert f_{0} - g_{0} \Vert_{B^{\sigma_{0}}_{2,1}}.
    $$
    \item Let  $ \sigma \in [ \sigma_{0},\infty) $ then there exists $ \epsilon(\sigma) \in (0, \epsilon_{0}]$ so that for any $R>0$, $ \sigma' \in [\sigma_{0},\sigma]$ and any $ f_{0} , g_{0} \in B^{\sigma_{0}}_{2,1}$ with 
    $$
    \Vert f_{0} \Vert_{B^{\sigma_{0}}_{2,1}} , \Vert g_{0} \Vert_{B^{\sigma_{0}}_{2,1}} \leq \epsilon(\sigma), \text{ and }  \Vert f_{0} \Vert_{B^{\sigma'}_{2,1}} , \Vert g_{0} \Vert_{B^{\sigma'}_{2,1}} \leq R,
    $$
    if $f,g$ are the solutions with initial data $ f_{0} , g_{0}$ then the following estimate holds  
    $$
    \sup_{t \in [-1,1]} \Vert f -g \Vert_{B^{\sigma'}_{2,1}} \lesssim_{ R,\sigma'} \Vert f_{0} -g_{0} \Vert_{B^{\sigma'}_{2,1}}.
    $$
\end{enumerate}
\end{theorem}
Now we show how Theorem \ref{Redeq} implies Theorem \ref{MainResult}.
\begin{proof}[Proof of Theorem \ref{MainResult}]
By rotation invariance, we may assume $ Q = (0,0,1)$. Then fix $ \sigma_{0} \geq \frac{n+1}{2}$ and choose $ \epsilon_{0} > 0$ so that $ C_{\sigma_{0}} \epsilon_{0} $ satisfy the conclusion of Theorem \ref{Redeq}, here $ C_{\sigma_{0}}$ is the constant coming from \eqref{sfg2}. Take $ u_{0} \in B^{\sigma_{0}+2s}_{Q}$ so that 
$$
\Vert u_{0} \Vert_{B^{\sigma_{0}}_{Q}} \leq \epsilon_{0}.
$$
Then by Lemma \ref{Reduction2} we have
$$
\Vert  L(u_{0}) \Vert_{B^{\sigma_{0}}_{2,1}} \leq C_{\sigma_{0}} \epsilon_{0},
$$
and by our choice of $ \epsilon_{0}$ we can apply Theorem \ref{Redeq} to obtain  $f : \R^{n+1} \to \C$ that solves \eqref{reductionmain} with initial data $L(u_{0})$, and satisfy the estimates in Theorem \ref{Redeq}. In particular, taking $u : = L^{-1}(f)$ we see that $u: \R^{n+1} \to \S^{2}$ solves \eqref{SSchrodingereq} with initial data $u_{0}$. Next, by Lemma \ref{Reduction2} we have 
$$
\sup_{t \in [-1,1]} \Vert u \Vert_{B^{\sigma}_{Q}} \lesssim_{\sigma_{0}} \Vert u_{0} \Vert_{B^{\sigma}_{Q}},
$$
for any $ \sigma \in [\frac{n+1}{2},\sigma_{0}+2s]$. This proves part (1) of Theorem \ref{MainResult}. Next, for part (2) we have by Lemma \ref{Reduction2}
\begin{equation*}
\begin{split}
&\sup_{t \in [-1,1]} d^{\sigma_{0}}_{Q} (u,v) 
\\
& \lesssim_{\sigma_{0}} \sup_{t \in [-1,1]} \Vert L(u) - L(v) \Vert_{B^{\sigma_{0}}_{2,1}}
\\
& \lesssim \Vert L(u_{0})- L(v_{0}) \Vert_{B^{\sigma_{0}}_{2,1}}
\\
& \lesssim_{\sigma_{0}} d^{\sigma_{0}}_{Q} (u_{0},v_{0})
\end{split}
\end{equation*}
and this gives part (2) of Theorem \ref{MainResult}. Lastly, we prove part (3), we use the same argument. Namely, choose $ \epsilon(\sigma)>0$ so that $ C_{\sigma} \epsilon(\sigma)$ satisfy the conclusion of part (3) of Theorem \ref{Redeq}. Then using Lemma \ref{Reduction2} and Theorem \ref{Redeq} we obtain for any $ \sigma' \in [ \sigma_{0}, \sigma]$
\begin{equation*}
\begin{split}
& \sup_{t \in [-1,1]} d^{\sigma'}_{Q}(u,v) 
\\
& \lesssim_{\sigma',R} \sup_{t \in [-1,1]}\Vert L(u) - L(v) \Vert_{B^{\sigma'}_{2,1}}
\\
& \lesssim_{R,\sigma'} \Vert L(u_{0}) - L(v_{0}) \Vert_{B^{\sigma'}_{2,1}}
\\
& \lesssim_{\sigma' ,R} d^{\sigma'}_{Q}(u_{0}, v_{0}),
\end{split}
\end{equation*} 
This concludes the proof of Theorem \ref{MainResult}.
\end{proof}
The rest of the paper is devoted to proving Theorem \ref{Redeq}.

\section{Preliminaries}
\subsection*{Notation}
We use the same notation in \cite{model}. Namely,  we will use Greek letters $\xi,\tau$ for phase space, and Roman letters $x,t$ for physical space.

Throughout this work we assume, unless otherwise stated, that $n \geq 3$ and $s \in (\frac{1}{2},1)$. Constants, like $C$ can change from line to line and generally may depend on the dimension and $s$. We use the notation $A \aleq B$ if there is a multiplicative, nonnegative constant $C$ such that $A \leq C B$, and $A \aeq B$ if $A \aleq B$ and $A \ageq B$.

We will write the Fourier transform as 
\[
 \mathcal{F}_{\R^k} f(\zeta) = \int_{\R^k} e^{-\i \zeta \cdot x} f(x) dx
\]
and it's inverse (up to a multiplicative constant which we will ignore for the sake of readability)
\[
 \mathcal{F}_{\R^k}^{-1} f(x) = \int_{\R^k} e^{+\i \zeta \cdot x} f(\zeta) d\zeta.
\]
We will omit the subscript when the space is clear from context. Another abuse of notation is that we will ``define'' the Fourier symbol of the operator $\i \partial_t - (-\Delta^s)$ as
 \[
   \mathcal{F} \brac{(\i \partial_t - (-\Delta^s))f}(\xi,\tau) =  -(\tau +  |\xi|^{2s})\mathcal{F} f(\xi,\tau)
\]
the multiplicative constant that should be on the right-hand side plays no role in the analysis.

\subsection*{Frequency projections}
Let $\varphi_{0}$ be the typical Littlewood-Paley bump function, $\varphi_{0} \in C_c^\infty((-2,2),[0,1])$, $\varphi_{0} \equiv 1$ in $[-1,1]$, $\varphi_{0}(r) = \varphi_{0}(|r|)$. Also, let $\varphi(r) := \varphi_{0}(r) - \varphi_{0}(2r)$ and observe that we have 
\[
 1 = \varphi_{0}(r) + \sum_{k=1}^\infty \varphi(r/2^{k}) \quad \forall r \in [0,\infty).
\]
For vectors $\xi \in \R^n$ we will tacitly mean 
\[
\varphi_{0}(\xi) \equiv \varphi_{0}(|\xi|), \quad \varphi(\xi) := \varphi(|\xi|).
\]
We will use the notation 
$$
\varphi_{[k_{1},k_{2}]}(r) = \sum_{m=k_{1}}^{k_{2}} \varphi(r/2^{m}).
$$
In particular, 
$$
\varphi_{k} (r) : = \varphi (r /2^{k} ).
$$
We also define $\varphi^+$ to be the restriction of $\varphi$ to $r > 0$, e.g. for $k_1,k_2 \geq 0$

\[ \varphi^{+}_{[k_{1},k_{2}]}(r) = \ind_{[0,\infty)}(r) \varphi_{[k_{1},k_{2}]}(r). \]

For $f: \R^{n+1} \to \C$, we denote the frequency projections by
\[
 \Delta_{k} f(x,t) := \mathcal{F}^{-1}_{\R^{n+1}}(\varphi_{k}(\xi) \mathcal{F}_{\R^{n+1}}f(\xi,\tau))  \quad k \geq 0,
\] 
and for $j \geq 0$ we denote the modulation 
\[
 Q_j f(x,t) :=  \begin{cases}
                   \mathcal{F}_{\R^{n+1}}^{-1} \brac{\varphi_{0}({\tau {+} |\xi|^{2s}})\,  \mathcal{F}_{\R^{n+1}} f(\xi,\tau)}   \quad &j = 0, \\
                   \mathcal{F}_{\R^{n+1}}^{-1} \brac{\varphi_{j}\brac{{\tau {+} |\xi|^{2s}}}\,  \mathcal{F}_{\R^{n+1}} f(\xi,\tau)} \quad & j \geq 1.
                   \end{cases}
\]
We also set  
\[
 \Delta_{\leq k}  := \sum_{\ell =0}^{k} \Delta_{\ell} 
\]
and similarly for $Q_{\leq j}$.

\subsection{Definition of the spaces:}
Throughout this paper we fix
\[
 \mathscr{E} \subset \S^{n-1},
\]
a finite set of vectors in $\S^{n-1}$ 
so that whenever $ e \in \mathscr{E}$ then $ -e \in \mathscr{E}$. We also take the cardinality of $ \mathscr{E}$ to be large enough so that we have the conclusion of the following ``geometric observation''.
\begin{lemma}\label{la:geometricobserv}
There exists a finite subset $\mathscr{E} \subset \S^{n-1}$ such that the following holds:

 \item For any $\xi \in \R^n$ there exists $e \in \mathscr{E}$ such that 
 \[
  \abs{\xi-\abs{\xi}e} \leq \vert \xi \vert,
 \]
and in particular
 \[
  \langle \xi,e\rangle \geq  \frac{\vert \xi \vert}{2}.
 \]

 Moreover there exists a decomposition of unity on the $\S^{n-1}$-sphere denoted by $( \tilde{\vartheta}_e)_{e \in \mathscr{E}}$ with 
 \[
  \supp \tilde{\vartheta}_e \subset \{\xi \in \S^{n-1}: \quad \langle \xi, e\rangle \geq \frac{1}{2} \}
 \]
and 
 \[
  \sum_{e \in \mathscr{E}} \tilde{\vartheta}_e(\xi/|\xi|) = 1 \quad \forall \xi \in \R^n \setminus \{0\}.
 \]

 We will denote 
 \[
  \vartheta_{e}(\xi) := \tilde{\vartheta}_{e} (\xi/|\xi|)
 \]
and with some abuse of notation will also use this notation for $e \in \S^{n-1}$ for $e \not \in \mathscr{E}$ as simply a cutoff function on the cone.

Observe that $\vartheta_{e} \not \in C^\infty$ since it has a singularity at $0$.
\end{lemma}
For $m \in \N $ define $ I_{m} : = [ 2^{m-1},2^{m+1} ]$ and for $m=0$ define $I_{0} : = [0, 1]$. \\
Let $j ,k \geq 0$ then define the following set 
$$
D_{k,j} : = \{ (\xi,\tau) \in \R^{n+1} : \vert \xi \vert \in I_{k}, \vert \vert \xi \vert^{2s} + \tau \vert \in I_{j} \}.
$$
Moreover, 
$$
D_{  k , \leq j } : = \bigcup_{j'=0}^{j} D_{k,j'},
$$
and in particular 
$$
D_{k,\infty} : = \bigcup_{j'=0}^{\infty} D_{k,j'} 
$$
we define $ D_{k, \geq j }$, $D_{ \leq k , j }$ in the same way.

Further, for any vector $ e \in \S^{n-1}$ consider the following set 
$$
D^{e,k'}_{k,j} : = \{ (\xi,\tau) \in D_{k,j} : \langle \xi ,e \rangle \in [2^{k'-1},2^{k'+1}] \}. 
$$
With this notation we define the following resolution spaces; 
for $f: \R^{n+1} \to \C$ and $k \geq 0$ we define the $X_k$-spaces as follows
\begin{equation}\label{eq:Xdef}
 \|f\|_{X_k} := \begin{cases}
                                       \sum_{j=0}^\infty  2^{\frac{j}{2}} \|Q_{j} f\|_{L^{2}_{t,x}}  \quad &\text{if $\supp \hat{f} \subset D_{k,\infty} $ } \\
                                       \infty \quad \text{otherwise}.
                                      \end{cases}
\end{equation}
Also, for $k \geq 100$, $k' \in \N \cap [1,k+1]$ and $e \in \mathscr{E}$ we define the $Y_{k,k'}^{e}$ on maps whose frequency support is in the cone 
\[
 \|f\|_{Y_{k,k'}^{e}} = \begin{cases}
                                      2^{-k'\frac{2s-1}{2}} \gamma_{k,k'} \|(\i \partial_t - (-\Delta)^s+\i) f \|_{L^1_e L^2_{t,e^\perp}} \quad &\text{if }   \supp \hat{f} \subset D^{e,k'}_{k,\leq 2sk+10} 
                                      \\
                                       \infty \quad \text{otherwise}
                                      \end{cases}
\]
where $ \gamma_{k,k'} : =  2^{2n ( k -k')} $.
We define the space
\[
 Z_{k} := X_{k} + \sum_{e \in \mathscr{E}} \sum_{k'=1}^{k+1} Y_{k,k'}^e
\]
equipped with the norm
\[
 \|f\|_{Z_{k}} = \inf_{f = f_1 + \sum_{e,k'} f_{e,k'}} \brac{\|f_1\|_{X_{k}} + \sum_{e \in \mathscr{E},k'=1,...,k+1}\|f^{k'}_e\|_{Y_{k,k'}^e}}.
\]

For $\sigma \geq 0$ the main resolution space $F^\sigma$ is induced by the norm
\begin{equation}\label{eq:Fsigma}
\|u\|_{F^\sigma} := \sum_{k = 0}^\infty 2^{k\sigma} \|\Delta_{k} u\|_{Z_k} ,
\end{equation}
and the space for the right-hand side is induced by the norm
\begin{equation}\label{eq:Nsigma}
\|F\|_{N^\sigma} := \sum_{k=0}^\infty 2^{k\sigma} \|(\i \partial_t -(-\Delta)^s + \i)^{{-1}} \Delta_{k} F\|_{Z_k}.
\end{equation}
\begin{remark}
The reason why we switch from the less technical spaces that we used in \cite{model}, which were adapted from \cite{DIE}, is that when given a function $ f \in Y^{e}_{k}$ and another function $ g \in Z_{\tilde{k}}$ with $ \tilde{k} < k -100$, we would like to have $ fg  \in Y^{e}_{k}$. It is easy to prove that  
$ \Vert fg \Vert_{L^{1}_{e}L^{2}_{e^{\perp},t}} < \infty$. However, the issue is that the support of $ \mathcal{F}(fg) $ might be outside the set $ \{ (\xi ,\tau) : \langle \xi, e \rangle \geq 2^{k-1} \} $. Therefore, by definition $ fg \notin Y^{e}_{k}$. On the other hand, if we work with the $Y^{e}_{k,k'}$ spaces, then this difficulty is avoided. Indeed, it is easy to show that 
$$
\supp \mathcal{F}(fg) \subset \{ (\xi,\tau ) :  \langle \xi, e \rangle \geq 2^{k-3} \}.
$$
Hence, if $ f \in Y^{e}_{k,k'} $ and $ g \in Z_{\tilde{k}}$ with $ 2^{\tilde{k}} \ll 2^{k'}$ then $ fg \in Y^{e}_{k,k'+1} + Y^{e}_{k,k'} + Y^{e}_{k,k'-1}$. The spaces $Y^{e}_{k,k'} $ were first used in \cite{cmp} for the case $s =1$.
\end{remark}

\subsection{Estimates between the resolution spaces}
In this section we restate the results we obtained in \cite{model} in terms of the spaces $Y^{e}_{k,k'}$. However, some of the results are not present in \cite{model} and are instead adapted from \cite{cmp} to the case $ s \in (1/2, 1)$. We begin by defining $T_{k} = [  \frac{4}{5} k, k+1] \cap \N$. This is useful in distinguishing the spaces $Y^{e}_{k,k'}$ when $ k' $ is close to $k$ and when $ k'$ is far from $k$. Also, for any $ k' \in T_{k}$ we write $ q(k',k) : = \frac{k'}{k}$. If $k,k'$ are clear from the context we will simply write $q$. We begin by recalling the following estimate concerning the multiplier $N$. This was proven for $k' = k$ in \cite{model}. We fix the following notation throughout this section: given a direction $ e \in \S^{n-1}$ we write a vector $ \xi \in \R^{n} $ as $ \xi= \xi_{1} e + \xi'$, where $ \xi' \in e^{\perp}$. 
\begin{lemma}\label{MultiplierN}
Let $ k \geq 100$ and $k' \in T_{k}$, $ e \in \S^{n-1}$. Fix $ (\xi,\tau)$ with the following properties
\begin{enumerate}
    \item $ 2^{k-1} \leq \vert \xi \vert \leq 2^{k+1}$
    \item $2^{k'-1} \leq \xi_{1} \leq 2^{k'+1} $
    \item  $ \vert \tau + \vert \xi \vert^{2s} \vert \leq 2^{2k( s + q -1)-80}$. 
\end{enumerate} 
Then $ \tau < 0 $ with $ (- \tau)^{1/s} \geq \vert \xi' \vert^{2} $ and if we
define $ N(\xi', \tau) : = (( -\tau)^{1/s} - \vert \xi' \vert^{2})^{1/2}  $,
then we have the following 
\begin{enumerate}
    \item $ N(\xi',\tau) \in  [ 2^{k'-2},2^{k'+2}] $
    \item $ \vert \tau + \vert \xi \vert^{2s} \vert \approx 2^{k(2s-2+q)} \vert \xi_{1} -N \vert$
\end{enumerate}
\end{lemma}
\begin{proof}
Notice that from $(1)$ we have that $ \vert \xi \vert^{2s} \in  [ 2^{2s(k-1)},2^{2s(k+1)}]$. Therefore, $(3)$ implies that $ \tau < 0$ and $ -\tau \in [ 2^{2s(k-2)},2^{2s(k+2)}]$. Then by mean value theorem we find $ r \approx 2^{2sk} $ so that 
\begin{equation}\label{w1}
\begin{split}
& \vert (-\tau)^{1/s} - \vert \xi \vert^{2} \vert 
\\
& = \frac{1}{s} r^{\frac{1}{s} - 1}  \vert - \tau - \vert \xi \vert^{2s} \vert 
\\
& \leq \frac{16}{s} 2^{2sk( \frac{1-s}{s})} \vert \tau + \vert \xi \vert^{2s} \vert 
\\
& \leq \frac{16}{s}  2^{2k (1-s)} 2^{2k(s+q -1) -80}
\\
& \leq 2^{2k q -70 } = 2^{2k'-70}.
\end{split}
\end{equation}
On the other hand, we have
\begin{equation*}
\begin{split}
&  (- \tau )^{1/s} - \vert \xi' \vert^{2}  
\\
& =  (-\tau)^{1/s} - \vert \xi \vert^{2}  + \xi^{2}_{1},
\end{split}
\end{equation*}
by the hypothesis we have $ \xi^{2}_{1} \in [ 2^{2(k'-1)},2^{2(k'+1)}]$. Therefore, by \eqref{w1} we have 
$$
 (- \tau )^{1/s} - \vert \xi' \vert^{2} \in [2^{2(k'-2)},2^{2(k'+2)}]
 $$
this yields $(1)$. Next, we prove $ (2) $. By applying the mean value theorem we find $ r \approx 2^{k'}$ so that 
\begin{equation*}
\begin{split}
&|  \tau + \vert \xi \vert^{2s} |
\\
& = | ( \vert \xi' \vert^{2} + \xi_{1}^{2} )^{s} - ( N(\xi',\tau)^{2} + \vert \xi'\vert^{2})^{s} |
\\
& = 2r ( r^{2} + \vert \xi' \vert^{2})^{s-1} | \xi_{1} - N |
\\
& \approx 2^{k'} 2^{2k(s-1)} \vert \xi_{1} - N \vert 
\\
& = 2^{k( 2s- 2 + q) } \vert \xi_{1} - N \vert. 
\end{split}
\end{equation*}

\end{proof}

Next, we prove the following embedding result
\begin{lemma}\label{Embeddings}
Let $ k,j \geq 0 $  then for $ f \in Z_{k}$ we have the following 
\begin{enumerate}
    \item 
    $$
    \Vert Q_{j}(f) \Vert_{X_{k}} \lesssim \Vert f \Vert_{Z_{k}}.
    $$
    \item  If $k \geq 100$ and $ f \in Y^{e}_{k,k'} $ with $ k' \in T_{k} $ then we have 
    $$
    \Vert Q_{j}(f) \Vert_{X_{k}} \lesssim (\min \{ \gamma^{-1}_{k,k'} 2^{ 2(k-k')(1-s)} , \gamma^{-1}_{k,k'} 2^{-j+2sk'} \})^{1/2} \Vert f \Vert_{Y^{e}_{k,k'}}.
    $$
    and in particular 
    $$
    \Vert Q_{j}(f) \Vert_{X_{k}} \lesssim \min \{1 ,  2^{(-j+2sk')/2} \} \Vert f \Vert_{Y^{e}_{k,k'}}.   
    $$
    \item If $f \in Y^{e}_{k,k'}$ and $k' \notin T_{k}$ then we have
    $$
    \Vert f \Vert_{X_{k}} \lesssim \Vert f \Vert_{Z_{k}}.
    $$
    \item If $ k' \in T_{k} $ and $ j_{0} = [2k (s+q -1)]$ i.e the least integer greater than $ 2k(s+q -1)$, then for $ f \in Y^{e}_{k,k'}$ we have
    $$
   \Vert  \sum_{j \geq j_{0}} Q_{j}(f) \Vert_{X_{k}} \lesssim \Vert f \Vert_{Y^{e}_{k,k'}}.
    $$
\end{enumerate}
\end{lemma}
\begin{proof}
(1) follows from (2) and (3) and the trivial estimate 
$$
\Vert Q_{j}(f) \Vert_{X_{k}} \leq \Vert f \Vert_{X_{k}}.
$$
Also we will show below that $(4)$ follows from (2). So we first start proving (2),(3). Define 
$$
h(x,t) = 2^{-k'(2s-1)/2} \mathcal{F}^{-1}_{\R^{n+1}} \big[ ( \tau + \vert \xi \vert^{2s} + \i ) \hat{f}\big] 
$$
then notice that 
\begin{equation}\label{qn5}
\Vert h \Vert_{L^{1}_{e}L^{2}_{e^{\perp}}} \lesssim \gamma^{-1}_{k,k'} \Vert f \Vert_{Y^{e}_{k,k'}}.
\end{equation}
 On the other hand, using that $ \hat{f} = 2^{k'(2s-1)/2} ( \tau + \vert \xi \vert^{2s} + \i)^{-1}\hat{h} $, we obtain  
 \begin{equation}\label{w2}
\begin{split}
&  2^{j/2} \Vert Q_{j}(f) \Vert_{L^{2}}
\\
& \lesssim 2^{-j/2} 2^{k'(2s-1)/2} \Vert \ind_{D_{k,j}} \hat{h} \Vert_{L^{2}} 
\\
& \lesssim 2^{ -j/2} 2^{k'(2s-1)/2} \Vert \ind_{D_{k,j}} \int_{\R} e^{\xi_{1} x_{1} i} \mathcal{F}_{e^{\perp},t} (h)(x_{1} ,\xi',\tau) dx_{1} \Vert_{L^{2}}
\\
& \lesssim \vert \{ \xi_{1} : \xi_{1} \approx 2^{k'} \text{and } \vert \tau + \vert \xi \vert^{2s} \vert \leq 2^{j+1} \}\vert^{1/2} 2^{k'(2s-1)/2} 2^{-j/2} \Vert h \Vert_{L^{1}_{e}L^{2}_{e^{\perp},t}}
\\
& \lesssim \vert \{ \xi_{1} : \xi_{1} \approx 2^{k'} \text{and } \vert \tau + \vert \xi \vert^{2s} \vert \leq 2^{j+1} \}\vert^{1/2} 2^{k'(2s-1)/2} 2^{-j/2}\gamma_{k,k'}^{-1} \Vert f \Vert_{Y^{e}_{k,k'}}
\end{split}
 \end{equation}
 to prove (3) we estimate the measure of the above set by $ 2^{k'}$ to obtain
\begin{equation*}
\begin{split}
&2^{j/2} \Vert Q_{j}(f) \Vert_{L^{2}}
\\
& \lesssim 2^{sk'} \gamma^{-1}_{k,k'} 2^{-j/2} \Vert f \Vert_{Y^{e}_{k,k'}}
\\
& \lesssim 2^{-j/2} \Vert f \Vert_{Y^{e}_{k,k'}}
\end{split}
\end{equation*}
where in the last line we used that $k' \notin T_{k}$.
Summing over $j$ gives us (3). To prove (2) we use \eqref{w2} above to reduce matters into proving 
\begin{equation}\label{w3}
\begin{split}
& \vert \{ \xi_{1} : \xi_{1} \approx 2^{k'} \text{and } \vert \tau + \vert \xi \vert^{2s} \vert \leq 2^{j+1} \}\vert^{1/2} 2^{k'(2s-1)/2} 2^{-j/2}\gamma_{k,k'}^{-1} \Vert f \Vert_{Y^{e}_{k,k'}}
\\
& \lesssim \left( \min \{ \gamma^{-1}_{k,k'} 2^{ 2(k-k')(1-s)} , \gamma^{-1}_{k,k'} 2^{j-2sk'} \}\right)^{1/2} \Vert f \Vert_{Y^{e}_{k,k'}}.
\end{split}
\end{equation}
To prove \eqref{w3}, we claim that 
\begin{equation}\label{w4}
\begin{split}
& \vert \{ \xi_{1} : \xi_{1} \approx 2^{k'} \text{and } \vert \tau + \vert \xi \vert^{2s} \vert \leq 2^{j+1} \}\vert^{1/2}
\\
& \lesssim \min \{ 2^{k'/2} , 2^{j/2} 2^{-k(2s - 2 + q )/2} \}.
\end{split}
\end{equation}
Assume that $ j > 2k (s+ q -1 ) -80$. Then 
$$
2^{j} 2^{-k (2s-2 + q)} \geq C^{-1} 2^{k(2s + 2 q - 2)} 2^{-k (2s-2 + q)} \geq C^{-1} 2^{ q k }=C^{-1} 2^{k'}
$$
and hence \eqref{w4} follows. Therefore, assume that $ j < 2k(s+ q -1) -80$. Then the hypothesis of Lemma \ref{MultiplierN} is satisfied and we obtain 
$$
\vert \xi_{1}  -N \vert \lesssim  2^{-k(2s-2 + q)} 2^{j}
$$
and this yields \eqref{w4}. Combining this with  \eqref{w2} we obtain 
\begin{equation}\label{w5}
\begin{split}
&   2^{j/2} \Vert Q_{j}(f) \Vert_{L^{2}}
\\
& \lesssim 2^{k'(2s-1)/2} 2^{-j/2} \min \{ 2^{k'/2} , 2^{j/2} 2^{-k(2s-2+q)/2} \} \gamma^{-1}_{k,k'} \Vert f \Vert_{Y^{e}_{k,k'}}
\\
& \lesssim \min \{ \gamma^{-1}_{k,k'} 2^{(2sk'-j)/2} , \gamma^{-1}_{k,k'} 2^{k'(2s-1)/2} 2^{-k(2s-2 + q)/2} \} \Vert f \Vert_{Y^{e}_{k,k'}}
\end{split}
\end{equation}
rewrite $ 2^{k'(2s-1)/2} 2^{-k(2s-2+q)/2}$ as $ 2^{ \frac{(k-k')}{2} ( 2 - 2s)} $ then we obtain 
$$
2^{j/2} \Vert Q_{j}(f) \Vert_{L^{2}} \lesssim \min\{ \gamma^{-1}_{k,k'} 2^{(k-k')(1-s)} , \gamma^{-1}_{k,k'} 2^{(2sk' -j)/2} \} \Vert f \Vert_{Y^{e}_{k,k'}}
$$
which yields (2). In particular, notice that since $ \gamma^{-1}_{k,k'} 2^{(k-k')(1-s)} \lesssim 1$ then we obtained the bound 
\begin{equation}\label{w6}
2^{j/2} \Vert Q_{j}(f) \Vert_{L^{2}} \lesssim \min \{ 1 , 2^{(2sk'-j)/2} \} \Vert f \Vert_{Y^{e}_{k,k'}}.
\end{equation}
Finally we proceed to proving (4). Notice that by \eqref{w6} we only need to prove 
\begin{equation}\label{w7}
\sum_{j = [2k(s+q -1)]}^{j = 2sk'} \Vert  2^{j/2} Q_{j}(f) \Vert_{L^{2}} \lesssim \Vert f \Vert_{Y^{e}_{k,k'}}.
\end{equation}
To prove \eqref{w7} we use (2) to obtain
\begin{equation*}
\begin{split}
& \sum_{j = [2k(s+q -1)]}^{j = 2sk'}\Vert  2^{j/2} Q_{j}(f) \Vert_{L^{2}}
\\
& \lesssim \Vert f \Vert_{Y^{e}_{k,k'}} \gamma_{k,k'}^{-1} (2sk' - 2k(s+q -1)) 2^{(k-k')(1-s)}
\\
& \lesssim (2s+2) ( k-k') 2^{(k-k')(1-s)} \gamma^{-1}_{k,k'}\Vert f \Vert_{Y^{e}_{k,k'}}
\\
& \lesssim \Vert f \Vert_{Y^{e}_{k,k'}}.
\end{split}
\end{equation*}
\end{proof}
Another useful lemma is the following, which tells us that $ Z_{k} \hookrightarrow L^{2}$ and the embedding constant depend on the lower bound of the modulation.

\begin{lemma}\label{Emb}
Let $k \geq 0 $ then for any $ j_{0} \geq 0 $, $f \in Z_{k}$ we have the following
    $$
    \Vert \sum_{j \geq j_{0}} Q_{j}(f) \Vert_{L^{2}} \lesssim 2^{-j_{0}/2} \Vert f \Vert_{Z_{k}}
    $$

\end{lemma}
\begin{proof}
Assume first that $ f \in X_{k}$ then clearly we have 
\begin{equation}
\begin{split}
& \Vert \sum_{j \geq j_{0}} Q_{j}(f) \Vert_{L^{2}}
\\
& \leq 2^{-j_{0}/2} \Vert \sum_{j \geq j_{0}} 2^{j/2}Q_{j}(f) \Vert_{L^{2}}
\\
& \leq 2^{-j_{0}/2} \Vert f \Vert_{X_{k}}.
\end{split}
\end{equation}
Next, assume $f \in Y^{e}_{k,k'}$ with $k' \in T_{k}$. If $ j_{0} = 0 $ then we simply need to show that 
\begin{equation}\label{mmm1}
\Vert f \Vert_{L^{2}(\R^{n+1})} \lesssim \Vert f \Vert_{Y^{e}_{k,k'}}
\end{equation}
In light of Lemma \ref{Embeddings} we may assume that $ f$ has Fourier transform supported in $ \{ (\xi,\tau) : \vert \tau + \vert \xi \vert^{2s} \vert \ll 2^{2k(s+q -1)} \}$. Set 
$$
h(x,t) : = 2^{-k' (2s-1)/2} \mathcal{F}^{-1}_{\R^{n+1}} [ (\tau + \vert \xi \vert^{2s} + \i ) \hat{f}].
 $$
 Then it is enough to show 
 \begin{equation}\label{mmmm1}
    \Vert \frac{ 2^{k'(2s-1)/2}}{\vert \tau + \vert \xi \vert^{2s} \vert + 1} \mathcal{F}(h) \Vert_{L^{2}(\R^{n+1})} \lesssim \gamma_{k,k'} \Vert h \Vert_{L^{1}_{e} L^{2}_{e^{\perp},t}}. 
 \end{equation}
To that end, write $ \xi = \xi_{1} e + \xi'$ where $ \xi' \in e^{\perp}$. Then by Minkowski and Lemma \ref{MultiplierN}
\begin{equation*}
\begin{split}
&  \Vert \frac{ 2^{k'(2s-1)/2}}{\vert \tau + \vert \xi \vert^{2s} \vert + 1} \mathcal{F}(h) \Vert_{L^{2}(\R^{n+1})} 
\\
& \lesssim 2^{k'(2s-1)/2} \Vert \int_{\R}\frac{1}{| \tau + \vert \xi \vert^{2s }|+1} e^{i x_{1} \xi_{1}} \mathcal{F}_{e^{\perp} \times \R}(h)(x_{1} e + \xi',\tau) d x_{1} \Vert_{L^{2}_{\xi,\tau}} 
\\
& \lesssim 2^{k'(2s-1)/2} 2^{-k(2s-2+q)} \Vert \left( \int_{\R} |\frac{ 1}{| \xi_{1} -N |+2^{-k (2s-2+q)}}|^{2} d \xi_{1} \right)^{1/2} \mathcal{F}_{e^{\perp} \times \R}(h)(x_{1} e + \xi',\tau)   \Vert_{L^{1}_{x_{1}} L^{2}_{\xi',\tau}}
\\
& \lesssim 2^{k'(2s-1)/2} 2^{-k(2s-2+q)} 2^{k(2s-2+q)/2} \Vert h \Vert_{L^{1}_{x_{1}} L^{2}_{\xi',\tau}}
\\
& \lesssim 2^{k'(2s-1)/2} 2^{-k(2s-2+q)/2} \gamma_{k,k'}^{-1} \Vert f \Vert_{Y^{e}_{k,k'}}
\end{split}
\end{equation*}
and lastly, notice the obvious bound
$$
2^{k'(2s-1)/2} 2^{-k(2s-2+q)/2} \gamma_{k,k'}^{-1} \lesssim 1.
$$
This concludes the lemma in the case $j_{0} = 0$.\\
Assume in what follows $j_{0} \geq 1$ and $ f \in Y^{e}_{k,k'}$ , $k' \in T_{k}$ and $f$ has Fourier transform  supported in $\vert  \vert \xi \vert^{2s} + \tau\vert \geq 2^{j_{0}} $. Then in light of Lemma \ref{Embeddings} we may assume that $ j_{0} \in [1, 2sk(s+q -1) -80] \cap \N$. Put $ h(x,t) = 2^{-(2s-1)k'/2} \mathcal{F}^{-1} ( \tau + \vert \xi \vert^{2s} +i) \hat{f} )$. Then it is enough to prove 
$$
\Vert \frac{ 2^{(2s-1)k'/2}}{\vert \tau + \vert \xi\vert^{2s} \vert +1} \mathcal{F}(h) \Vert_{L^{2}} \lesssim 2^{-j_{0}/2} \gamma_{k,k'} \Vert h \Vert_{L^{1}_{e} L^{2}_{e^{\perp},t}}.
$$
To that end, write $ \xi = \xi_{1} e + \xi'$, where $ \xi' \in e^{\perp}$. Then the left hand side of the above can be written as  
$$
\Vert \int_{\R} e^{ i x_{1} \xi_{1}} \frac{ 2^{(2s-1)k'/2}}{\vert \tau + \vert \xi\vert^{2s} \vert +1} \ind_{ \vert \tau +\vert \xi\vert^{2s} \vert \geq 2^{j_{0}}}(\xi,\tau) \tilde{h}(x_{1},\xi',\tau) d x_{1} \Vert_{L^{2}_{\xi,\tau}}
$$
where $ \tilde{h}(x_{1},\xi',\tau) = \mathcal{F}_{e^{\perp} \times \R} h (x_{1},\xi',\tau)$. Apply Minkowski to get that the above is bounded by 
$$
\int_{\R} \left( \int_{e^{\perp} \times \R} \int_{\R} \frac{2^{(2s-1)k'}}{(\vert \tau + \vert \xi\vert^{2s} + 1 )^{2} } \ind_{ \vert \tau + \vert \xi \vert^{2s} \geq 2^{j_{0}}} d \xi_{1} \vert \tilde{h}(x_{1},\xi',\tau) \vert^{2} d \xi' d\tau \right)^{1/2} dx_{1}.
$$
Next, fix $\xi',\tau$ in the support of $\tilde{h}$ and notice that by using Lemma \ref{MultiplierN}, we have that for each fixed $ \xi',\tau$ the inner integral is controlled by, 
$$
 2^{(2s-1)k'} 2^{-2k (2s-2+q)} \int_{\R} \frac{ 1}{ \vert \xi_{1} - N \vert^{2}} \ind_{\vert \xi_{1} -N \vert \geq C 2^{-(2s-2+q)k} 2^{j_{0}}}(\xi,\tau) d\xi_{1},
$$
and this is clearly controlled by $ \gamma^{2}_{k,k'} 2^{-j_{0}}$. This yields the result. 
\end{proof}

Next, we prove a representation formula for functions in $Y^{e}_{k,k'}$ the proof is essentially the same as the one in \cite{model} with slight modifications for when $k' \neq k$ but $k' \in T_{k}$. First, we need the following lemma. Notice that if $ q =1$, which happens when $k'=k$, then this is precisely \cite[Lemma 2.3]{model}. 
\begin{lemma}\label{dd}
Let $ k \geq 100 $ and $ k' \in T_{k}$ and $ s \in (1/2,1]$ and $ e \in \S^{n-1}$. Write $ \xi = \xi_{1}e + \xi'$ with $ \xi' \in e^{\perp}$. Then the following equality holds
\begin{equation}\label{w8}
\begin{split}
& \ind_{ | \xi'|  \leq  2^{k+1}} \varphi_{k}(\xi) \varphi^{+}_{[k'-1,k'+1]} (\xi_{1}) \frac{ \varphi_{ \leq 2k(s+q -1) -80} ( \tau + \vert \xi \vert^{2s})}{ \tau + \vert \xi \vert^{2s} + \i}
\\
& = \varphi_{k}(\xi) \ind_{ \vert \xi' \vert \leq 2^{k+1}} \varphi^{+}_{[k'-1,k'+1]} (N(\xi',\tau)) \frac{ \varphi_{ \leq k' - 80}( N - \xi_{1})}{ K(\xi,\tau) (\xi_{1} - N + \frac{\i}{2^{k(2s-2+q)}})} +E(\xi,\tau) 
\end{split}
\end{equation}
where $K(\xi',\tau) : = 2s ( N( \xi',\tau)^{2} + \xi'^{2} )^{s-1}N(\xi',\tau) $
and the error term satisfy the estimate 
$$
E(\xi,\tau) \lesssim 2^{-2sk'} 2^{2(1-s)(k-k')} + ( \vert \tau + \vert \xi \vert^{2s}+1)^{-2}
$$
and $E$ is supported in 
$$
S : = \{ (\xi,\tau) : \xi_{1} \approx 2^{k'} , \vert \tau + \vert \xi \vert^{2s} \vert \lesssim 2^{2k(s+q -1)} \}
$$
\end{lemma}
\begin{proof}
For better readability, we will simply write $N$ to mean $N(\xi',\tau)$. By Lemma \ref{MultiplierN} we obtain that the support of the left hand side of \eqref{w8} is contained in the set 
$$
S' : = \{ ( \xi',\tau) \in e^{\perp} \times \R : \vert \xi' \vert \leq 2^{k+1} , N \in [ 2^{k'-2},  2^{k'+2}] \}.
$$
Therefore, 
\begin{equation*}
\begin{split}
& \ind_{ | \xi'|  \leq  2^{k+1}} \varphi^{+}_{[k'-1,k'+1]} (\xi_{1}) \frac{ \varphi_{ \leq 2k(s+q -1)-80} ( \tau + \vert \xi \vert^{2s})}{ \tau + \vert \xi \vert^{2s} + i} 
\\
& =\ind_{S'}(\xi',\tau) \varphi^{+}_{[k'-1,k'+1]} (\xi_{1}) \frac{ \varphi_{ \leq 2k(s+q -1)-80} ( \tau + \vert \xi \vert^{2s})}{ \tau + \vert \xi \vert^{2s} + \i}.
\end{split}
\end{equation*}
Next, we rewrite $ \frac{1}{\tau+\vert \xi \vert^{2s} +\i}$ as follows
\begin{equation}\label{w12}
\begin{split}
& \tau + \vert \xi \vert^{2s} 
\\
& =( \xi_{1}^{2} + \vert \xi' \vert^{2})^{s} - ( N^{2} + \vert \xi' \vert^{2})^{s}
\\
& = L (\xi,\tau) ( \xi_{1} - N) + K (\xi',\tau) ( \xi_{1} - N) 
\end{split}
\end{equation}
where 
$$
L(\xi,\tau) : = \frac{ ( \xi_{1}^{2} + \vert \xi' \vert^{2})^{s}- ( N^{2} + \vert \xi' \vert^{2})^{s} }{ \xi_{1} -N} - K(\xi',\tau).
$$
We claim that $L$ satisfy the following estimate; whenever $ (\xi,\tau)$ is in the support of the left hand side of \eqref{w8} we have
\begin{equation}\label{w10}
\vert L(\xi,\tau) \vert \approx 2^{k(2s-2)} \vert \xi_{1} - N \vert
\end{equation}
To prove this, consider the function $ g(t) = (t^{2} + \vert \xi' \vert^{2})^{s} $. Then by Taylor theorem we obtain $r$ between $ \xi_{1}$ and $N$ so that 
$$
g(\xi_{1} ) = g(N) + g'(N) ( \xi_{1} - N) + \frac{1}{2} g''(r) (\xi_{1} - N)^{2}
$$
using the definition of $K, L $ and the above Taylor approximation, we obtain
\begin{equation*}
\begin{split}
& \vert L ( \xi,\tau) \vert
\\
& \approx_{s} \vert \xi_{1} - N \vert ( r^{2} + \vert \xi' \vert)^{s-2}(r^{2} + \vert \xi' \vert^{2} + (2s-1) r^{2})
\\
& \approx \vert \xi_{1} -N \vert 2^{2k(s-1)}
\end{split}
\end{equation*}
this proves \eqref{w10}. Notice also that by definition we have 
\begin{equation}\label{w11}
\vert K (\xi',\tau) \vert \approx 2^{k'} 2^{2k(s-1)} =2^{k(2s-2+q)}.
\end{equation}
Therefore, we write 

\begin{equation*}
\begin{split}
\frac{1}{ \tau + | \xi |^{2s} + \i}
& = \frac{1}{L(\xi,\tau)\, \brac{\xi_1 - N(\xi',\tau)} + K(\xi',\tau) \brac{\xi_1 -N(\xi',\tau)}+\i}
\\
& = \frac{1}{ K(\xi,\tau) \brac{ \xi_1 - N(\xi',\tau)+ \frac{\i}{2^{k(2s-2+q)}} } } + E_{1} + E_{2}
\end{split}
\end{equation*}
where (for better readability we now drop the arguments for $L$ and $K$)
$$
E_{1} = \frac{1}{ L\, \brac{\xi_1 - N} + K \brac{\xi_1 -N}+\i} - \frac{1}{K \left( \xi_1 - N + \frac{i}{K}\right)}
$$
and
$$
E_{2} =  \frac{1}{ K \left( \xi_1 - N+ \frac{i}{K} \right)} - \frac{1}{ K \left( \xi_1 - N+ \frac{\i}{2^{k(2s-2+q)}}\right)}.
$$
We show $E_{i}$ satisfy the claimed decay estimate. More precisely we show
\begin{equation}\label{eq:E12est}
| E_{i} |\aleq  \left( | \tau + | \xi |^{2s} | +1\right)^{-2} + 2^{-2sk'} 2^{2(1-s)(k-k')}, \quad i=1,2
\end{equation}
We start with $E_{1}$, which we can write as
\begin{equation*}
\begin{split}
& E_{1} = \frac{ -( \xi_{1} -N) L}{ \big[(\xi_{1} -N) L + (\xi_{1} -N) K + \i \big] (\xi_{1} -N+ \i) K}
\\
& = \frac{1}{K} \frac{ - ( \xi_{1} -N) L }{(\tau + \vert \xi \vert^{2s} +\i)(\xi_{1} -N +\frac{\i}{K}) }.
\end{split}
\end{equation*}
Therefore, using \eqref{w11} and \eqref{w10} we obtain
\begin{equation*}
\begin{split}
& \vert E_{1} \vert
\\
& \lesssim 2^{-k'-k(2s-2)} \frac{ \vert \xi_{1} -N \vert^{2} 2^{2k(s-1)}}{ (\vert \tau + \vert \xi \vert^{2s} + 1) ( \vert \xi_{1} - N \vert + 2^{-k'-k(2s-2)})}
\\
& \lesssim 2^{-k'} \frac{ 2^{-k(2s-2+q) }\vert \xi_{1} - N \vert^{2} }{ (|\xi_{1} - N| + 2^{-k(2s-2+q)}) ( \vert \xi_{1} -N \vert + 2^{-k'-k(2s-2)})}
\\
& \lesssim 2^{-k'} 2^{-k(2s-2+q)} = 2^{-2sk'} 2^{2(1-s)(k-k')}.
\end{split}
\end{equation*}
Where we used Lemma \ref{MultiplierN} in the third line . This proves \eqref{eq:E12est} for $E_{1}$. We proceed to do the same with $E_{2}$
\begin{equation}\label{w13}
\begin{split}
& \vert E_{2} \vert 
\\
& = \vert\frac{1}{K} \frac{\i ( \frac{1}{K} - \frac{1}{2^{k(2s-2+q)}})}{( \xi_{1} -N + \frac{\i}{K}) ( \xi_{1} -N + \frac{\i}{2^{k(2s-2+q)}}) }|
\\
& \lesssim \frac{1}{ \big[2^{k(2s-2+q)} \vert \xi_{1} -N \vert + 1  \big]^{2}},
\end{split}
\end{equation}
Therefore, we bound \eqref{w13} by 
\begin{equation*}
\begin{split}
& \frac{1}{ \big[2^{k(2s-2+q)} \vert \xi_{1} -N \vert + 1  \big]^{2}}
\\
& \approx ( \vert \tau + \vert \xi \vert^{2s} + 1 )^{-2}.
\end{split}
\end{equation*}
This proves \eqref{eq:E12est}. So far we proved the following 
\begin{equation*}
\begin{split}
& \ind_{ | \xi'|  \leq  2^{k+1}} \varphi_{k}(\xi) \varphi^{+}_{[k'-1,k'+1]} (\xi_{1}) \frac{ \varphi_{ \leq 2k(s+q -1)-80} ( \tau + \vert \xi \vert^{2s})}{ \tau + \vert \xi \vert^{2s} + i} 
\\
& = \ind_{S'}(\xi',\tau) \varphi_{k}(\xi) \varphi^{+}_{[k'-1,k'+1]}(\xi_{1}) \frac{ \varphi_{ \leq 2k(s+q -1)-80} (\tau+\vert \xi \vert^{2s})}{ K(\xi,\tau) ( \xi_{1} - N + \frac{ \i }{2^{k(2s-2+q)}})} +E 
\end{split}
\end{equation*}
with $E $ is supported in $S $ and satisfy the decay estimate \eqref{eq:E12est}. Next, we change the arguments inside the cutoff functions. To that end, we write
\begin{equation*}
\begin{split}
& \ind_{S'}(\xi',\tau) \varphi_{k}(\xi) \varphi^{+}_{[k'-1,k'+1]}(\xi_{1}) \frac{ \varphi_{ \leq 2k(s+q -1)-80} (\tau+\vert \xi \vert^{2s})}{ K(\xi,\tau) ( \xi_{1} - N + \frac{ \i }{2^{k(2s-2+q)}})}
\\
& = \ind_{S'}(\xi',\tau) \varphi_{k}(\xi) \varphi^{+}_{[k'-1,k'+1]}(N)  \frac{ \varphi_{ \leq k'-80} (\xi_{1} -N)}{ K(\xi,\tau) ( \xi_{1} - N + \frac{ \i }{2^{k(2s-2+q)}})} + E_{3}
\end{split}
\end{equation*}
where 
$$
E_{3}(\xi,\tau) : =  \varphi_{k}(\xi) \ind_{S'}(\xi',\tau) \frac{\varphi^{+}_{[k'-1,k'+1]}(\xi_{1}) \varphi_{\leq 2k(s+q-1)-80} (\tau + \vert \xi \vert^{2s}) -\varphi^{+}_{[k'-c,k'+c]}(N) \varphi_{\leq k' -c} (\xi_{1} -N)}{   K(\xi,\tau) ( \xi_{1} - N + \frac{ \i }{2^{k(2s-2+q)}})}.
$$
it is clear that this error is supported in the set $ S$. Indeed, by definition $E_{3}$ is supported in $S'$ and so $ N \approx 2^{k'}$, if $ \xi_{1} \gg 2^{k'}$ then $ \xi_{1} - N \gg 2^{k'}$ and so $ E_{3} =0$. If on the other hand $ \xi_{1} \ll 2^{k'}$ then $ N - \xi_{1} \approx 2^{k'} \gg 2^{k'-80}$ and so we obtain that $ E_{3} = 0$. Therefore, $ \xi_{1} \approx 2^{k'} \approx N $ and $ \vert \xi' \vert \lesssim 2^{k}$ in the support of $ E_{3} $. Then by Lemma \ref{MultiplierN} we obtain that $ \vert \tau + \vert \xi \vert^{2s} \vert \lesssim 2^{2k(s-1+q)}$. Thus, $E_{3}$ is supported in the set $S$. We proceed to prove the decay estimate for $E_{3}$.
\begin{equation}\label{w14}
\begin{split}
&| E_{3} | 
\\
& \lesssim \frac{ \vert \varphi_{[k'-1,k'+1]}^{+}(\xi_{1}) - \varphi^{+}_{[k'-1,k'+1]}(N)  \vert \varphi_{\leq 2k(s+q -1) -80}(\tau + \vert \xi \vert^{2s})}{\vert K \vert \mbox{ } \vert \xi_{1} -N \vert 
 + 2^{-k(2s-2+q)}}
 \\
 & + \frac{ \varphi^{+}_{[k'-1,k'+1]}(N) \vert \varphi_{ \leq 2k(s + q -1) -80} ( \tau + \vert \xi \vert^{2s}) - \varphi_{k' - 80}(\xi_{1} -N ) \vert }{ \vert K \vert \mbox{ } \vert \xi_{1} -N \vert + 2^{-k(2s-2+q)} }
\end{split}
\end{equation}
we estimate each term on the right hand side of \eqref{w14}. 
\begin{equation*}
\begin{split}
& \frac{ \vert \varphi_{[k'-1,k'+1]}^{+}(\xi_{1}) - \varphi^{+}_{[k'-1,k'+1]}(N) \vert  \varphi_{\leq 2k(s+q -1) -80}(\tau + \vert \xi \vert^{2s})}{\vert K \vert \mbox{ } \vert \xi_{1} -N \vert 
 + 2^{-k(2s-2+q)}}
 \\
 & \lesssim \frac{ \vert \varphi_{[k'-1,k'+1]}^{+}(\xi_{1}) - \varphi^{+}_{[k'-1,k'+1]}(N) \vert }{\vert K \vert \mbox{ } \vert \xi_{1} -N \vert 
 + 2^{-k(2s-2+q)}} 
 \\
 & \lesssim \frac{ 2^{-k'} \vert \xi_{1} -N \vert }{ 2^{k(2s-2+q)} \vert \xi_{1} - N \vert }
 \\
 & \lesssim 2^{-2k(s-1+q)}
 \\
 & \lesssim 2^{-2sk'} 2^{2(1-s) (k-k')},
 \end{split}
\end{equation*}
which is the desired decay estimate. Next, we prove the same for the second term on the right hand side of \eqref{w14}. To that end, 
\begin{equation*}
\begin{split}
& \frac{ \varphi^{+}_{[k'-1,k'+1]}(N) \vert \varphi_{ \leq 2k(s + q -1) -80} ( \tau + \vert \xi \vert^{2s}) - \varphi_{k' - 80}(\xi_{1} -N ) \vert}{ \vert K \vert \mbox{ } \vert \xi_{1} -N \vert + 2^{-k(2s-2+q)} }
\\
& \lesssim \frac{  \vert \varphi_{ \leq 2k(s + q -1) -80} ( \tau + \vert \xi \vert^{2s}) - \varphi_{k' - 80}(\xi_{1} -N )}{ \vert K \vert \mbox{ } \vert \xi_{1} -N \vert + 2^{-k(2s-2+q)} }
\\
& \lesssim 2^{-k(2s-2+q)} \frac{ \vert \frac{ \tau + \vert \xi \vert^{2s}}{2^{2k(s+q-1)}} - \frac{ \xi_{1} -N }{2^{k'}}  \vert }{ \vert \xi_{1} -N \vert + 2^{-k(2s-2+q)}}
\\
& \lesssim 2^{-2k(s-1+q)}  
\frac{ \vert \frac{ \tau + \vert \xi \vert^{2s}}{2^{k(2s+q-2)}} - \xi_{1} - N  \vert }{ \vert \xi_{1} -N \vert + 2^{-k(2s-2+q)}}
\\
& \lesssim 2^{-2sk'} 2^{2(1-s)(k-k')} \frac{ \vert \xi_{1} - N  \vert }{ \vert \xi_{1} -N \vert + 2^{-k(2s-2+q)}}
\\
& \lesssim 2^{-2sk'} 2^{2(1-s)(k-k')}.
\end{split}
\end{equation*}
\end{proof}

With the above, we can state the main representation formula for functions in $Y^{e}_{k,k'}$. Similar to the above, if $ k= k'$ then $ q =1 $ and we recover \cite[Lemma 3.6]{model}. Also, if $ s=1$ then we recover \cite[Lemma 2.4]{cmp}.
\begin{lemma}\label{repf}
Let $k \geq 100 $ and $ f \in Y^{e}_{k,k'}$ then there exists a function $g : \R^{n+1} \to \C $ and a function $h : \R \times e^{\perp} \times \R \to \C$ so that we have the following equality 
\begin{equation*}
\begin{split}
& \mathcal{F}_{\R^{n+1}}(f)(\xi,\tau)
\\
& = \ind_{S'}(\xi',\tau) \varphi_{[k'-1,k'+1]}^{+}(N) \frac{  \varphi_{\leq k'-80}(\xi_{1} -N) }{ \xi_{1} -N + \frac{\i}{2^{k(2s-2+q)}}} \int_{\R} e^{i \xi_{1} x_{1}} h(x_{1} , \xi', \tau) d x_{1} + \hat{g}(\xi,\tau)
\end{split}
\end{equation*}
where 
$$
S': = \{ (\xi',\tau) \in e^{\perp} \times \R : \vert \xi' \vert \leq 2^{k+1} , N(\xi',\tau) \in [ 2^{k'-2},2^{k'+2}] \}
$$
and $ g $ satisfy the estimate 
$$
\Vert g \Vert_{X_{k}} \lesssim  \Vert f \Vert_{Y^{e}_{k,k'}}
$$
and $h$ satisfy the estimate 
$$
\Vert h \Vert_{L^{1}_{e}L^{2}_{e^{\perp},t}} \lesssim  2^{2(1-s)(k-k')} \gamma^{-1}_{k,k'} 2^{-k'(2s-1)/2} \Vert f \Vert_{Y^{e}_{k,k'}}
$$
and lastly, $h$ is supported in the set $ \R \times S'$.
\end{lemma}
\begin{proof}
In light of Lemma \ref{Embeddings} we may assume that $ k' \in T_{k}$ and $$ \hat{f}(\xi,\tau) = \varphi_{\leq 2k(s+q -1)-80}(\tau + \vert \xi \vert^{2s}) \hat{f}(\xi,\tau).
$$
Next, define $ \tilde{h} $ as
$$
\tilde{h}( x,t) : = \mathcal{F}^{-1}_{\R^{n+1}} \big[ ( \tau + \vert \xi \vert^{2s} +\i) \hat{f}\big](x,t) .
$$
Therefore, 
$$
\hat{f}(\xi,\tau) = \varphi^{+}_{[k'-1,k'+1]}(\xi_{1}) \frac{ \varphi_{\leq 2k(s+q -1) -80}}{ \tau + \vert \xi \vert^{2s} + \i} \mathcal{F}_{\R^{n+1}} ( \tilde{h})(\xi,\tau).
$$
We apply Lemma \ref{dd} to obtain
\begin{equation*}
\begin{split}
& \hat{f}(\xi,\tau)
\\
& = \varphi^{+}_{[k'-1,k'+1]}(N) \frac{ \varphi_{k'-80}(\xi_{1} -N)}{K(\xi',\tau) (\xi_{1} -N + \frac{\i}{2^{k(2s-2+q)}})} \mathcal{F}_{\R^{n+1}}(\tilde{h})(\xi,\tau) + E(\xi,\tau) \mathcal{F}_{\R^{n+1}}(\tilde{h})(\xi,\tau)
\end{split}
\end{equation*}
where $ \vert E \vert \lesssim 2^{-2sk' +2(1-s)(k-k')} + \vert \tau + \vert \xi \vert^{2s} + 1 \vert^{-2}$. We show we can absorb the term $ E \mathcal{F}_{\R^{n+1}}(\tilde{h}) $ into $g$. In other words, we show the estimate 
\begin{equation}\label{w16}
\begin{split}
& \Vert \mathcal{F}^{-1}(E \mathcal{F}(\tilde{h})) \Vert_{X_{k}} \lesssim  \Vert f \Vert_{Y^{e}_{k,k'}}.
\end{split}
\end{equation}
To that end, we estimate 
\begin{equation}\label{w17}
\begin{split}
& 2^{j/2}\Vert \varphi_{j}( \tau + \vert \xi \vert^{2s}) E(\xi,\tau) \mathcal{F}_{\R^{n+1}}(\tilde{h}) \Vert_{L^{2}}
\\
& \lesssim 2^{j/2} ( 2^{-2sk' +2(1-s)(k-k')} + 2^{-2j}) \Vert \varphi_{j}(\tau+\vert \xi \vert^{2s}) \int_{\R} e^{-y \xi_{1} \i } \mathcal{F}_{x',t}(\tilde{h})( ye + \xi',\tau) d y \Vert_{L^{2}}
\\
& \lesssim 2^{j/2} ( 2^{-2sk' +2(1-s)(k-k')} + 2^{-2j})  \Vert \tilde{h} \Vert_{L^{1}_{e}L^{2}_{e^{\perp},t}} \vert \{ \xi_{1} :  \xi_{1}  \approx 2^{k'} , \vert \tau + \vert \xi \vert^{2s} \vert  \leq 2^{j+1} \}\vert^{1/2}
\\
& \lesssim 2^{j/2} (2^{-2sk'+2(1-s)(k-k')} + 2^{-2j}) \min \{ 2^{k'/2}, 2^{j/2} 2^{-k(2s-2+q)/2} \} \Vert \tilde{h} \Vert_{L^{1}_{e}L^{2}_{e^{\perp},t}}
\end{split}
\end{equation}
notice that by definition of $\tilde{h}$ we have 
$$
\Vert \tilde{h} \Vert_{L^{1}_{e}L^{2}_{e^{\perp},t}} \leq \gamma^{-1}_{k,k'} 2^{k'(2s-1)/2} \Vert f\Vert_{Y^{e}_{k,k'}}
$$
using this bound for \eqref{w17} we obtain 
\begin{equation*}
\begin{split}
& 2^{j/2} \Vert \varphi_{j}( \tau + \vert \xi \vert^{2s}) E(\xi,\tau) \mathcal{F}_{\R^{n+1}}(\tilde{h}) \Vert_{L^{2}}
\\
& \lesssim 2^{j/2} ( 2^{-2sk'+2(1-s)(k-k')} + 2^{-2j} ) \min \{ 2^{k'/2}, 2^{j/2} 2^{-k(2s-2+q)/2} \} 2^{k'(2s-1)/2} \gamma^{-1}_{k,k'} \Vert f \Vert_{Y^{e}_{k,k'}}.
\end{split}
\end{equation*}
Therefore, we will obtain \eqref{w16} once we prove that 
\begin{equation}\label{w18}
\begin{split}
\sum_{j=0}^{[ 2k(s+q -1)]} 2^{j/2} ( 2^{-2sk'+2(1-s)(k-k')} + 2^{-2j} ) \min \{ 2^{k'/2}, 2^{j/2} 2^{-k(2s-2+q)/2} \} 2^{k'(2s-1)/2} \lesssim \gamma_{k,k'}.
\end{split}
\end{equation}
But this is straightforward to verify. Indeed, 
\begin{equation}
\begin{split}
& 2^{-2sk'+2(1-s)(k-k')} 2^{k'(2s-1)/2}  \sum_{j=0}^{[2k(s+q -1)]} 2^{j/2} \min \{ 2^{k'/2}, 2^{j/2} 2^{-k(2s-2+q)/2} \} 
\\
& \lesssim 2^{-2sk'+2(1-s)(k-k')} 2^{k'(2s-1)/2} 2^{-k(2s-2+q)/2} \sum_{j=0}^{[2k(s+q-1)]} 2^{j}
\\
& \lesssim 2^{-2sk'+2(1-s)(k-k')} 2^{k'(2s-1)/2} 2^{-k(2s-2+q)/2} 2^{2k(s+q -1)} 
\\
& \lesssim 2^{-2sk'+2(1-s)(k-k')} 2^{k'(2s-1)/2} 2^{-k'/2} 2^{k(s-1)} 2^{2k'} 
\\
& \lesssim 2^{k'(1-s)} 2^{k(s-1)} 2^{2(1-s)(k-k')} 
\\
& \lesssim \gamma_{k,k'}.
\end{split}
\end{equation}
Similarly, 
\begin{equation*}
\begin{split}
& 2^{k'(2s-1)/2} \sum_{j=0}^{[2k(s+q-1)]} 2^{-j}   2^{-k(2s-2+q)/2} 
\\
& \lesssim  2^{k'(2s-1)/2 } 2^{-k(2s-2+q)/2}   
\\
& \lesssim  \gamma_{k,k'}.
\end{split}
\end{equation*}

 This proves \eqref{w16}. Therefore, we arrive at 
\begin{equation*}
\begin{split}
& \hat{f}(\xi,\tau) = \hat{g}(\xi,\tau) 
\\
& + \varphi^{+}_{[k'-1,k'+1]}(N) \frac{ \varphi_{\leq k'-80}(\xi_{1} -N)}{K(\xi',\tau) (\xi_{1} -N + \frac{\i}{2^{k(2s-2+q)}})} \mathcal{F}_{\R^{n+1}}(\tilde{h})(\xi,\tau).
\end{split}
\end{equation*}
Lastly, write
$$
\mathcal{F}_{\R^{n+1}}(\tilde{h})(\xi,\tau) = \int_{\R} e^{- y \xi_{1} \i} \mathcal{F}_{x',t} \tilde{h}( y e + \xi',\tau) d y 
$$
then define $h : \R \times e^{\perp} \times \R \to \C$ as 
$$
h(y,\xi',\tau) : = \frac{1}{K(\xi',\tau)} \mathcal{F}_{x',t} ( \tilde{h})(ye+\xi',\tau) 
$$
clearly $h$ satisfy the desired estimate, and in light of the cutoff functions, we may assume $h $ is supported in the indicated set. 

\end{proof}
The following result summarizes the behavior of the $Z_{k}$ norms when one multiplies by a multiplier $m \in L^{\infty}(\R^{n}).$ This will be used extensively in what follows, sometimes without reference.  

\begin{lemma}\label{Multiplier}
\begin{enumerate}
\item For $k,j \geq 0 $ we have 
$$
\Vert Q_{\leq j} (f) \Vert_{Z_{k}} \lesssim \Vert f \Vert_{Z_{k}}.
$$
    \item For any $\ell \in \N$, $p,q \in [1,\infty]$, $e \in \S^{n-1}$, $\beta \in \R$ we have 
\[
 \|\Ds{\beta} \Delta_{\ell} f\|_{L^p_eL^q_{e^\perp t}} \lesssim 2^{\ell \beta} \sum_{\ell': 2^{\ell'} \aeq 2^\ell} \|\Delta_{\ell'} f\|_{L^p_eL^q_{e^\perp t}}.
\]
\item For any $m \in L^{\infty}(\R^{n}) $ with $ \mathcal{F}^{-1}_{\R^{n}} (m) \in L^{1}(\R^{n}) $ and $f \in Z_{k}$ one has 
$$
\Vert \mathcal{F}_{\R^{n+1}}^{-1} ( m \hat{f}  ) \Vert_{Z_{k}} \lesssim \Vert \mathcal{F}_{\R^{n}}^{-1}(m) \Vert_{L^{1}(\R^{n})} \Vert f \Vert_{Z_{k}}
$$
\end{enumerate}

\end{lemma}
\begin{proof}
We begin with the proof of (1). If $ f \in X_{k}$ then we clearly have, by Plancheral, 
$$
\Vert Q_{\leq j} (f) \Vert_{X_{k}} \leq \Vert f \Vert_{X_{k}}.
$$
Therefore, assume in what follows $k \geq 100$ and $ f \in Y^{e}_{k,k'}$ with $k' \in T_{k}$. In view of Lemma \ref{Embeddings} we may assume $f$ has modulation controlled by $ 2^{2k(s+q-1)-80}$. 
That is, we may assume that $ j\leq 2k(s+q-1)-~80 \leq 2sk'-C$, for sufficiently large $C>0$. Then write $ \xi = \xi_{1} e + \xi' $ with $ \xi' \in e^{\perp}$. And notice that we have 
\begin{equation*}
\begin{split}
& \Vert Q_{\leq j}(f) \Vert_{Y^{e}_{k,k'}}
\\
& \lesssim \Vert \int_{\R} e^{\i \xi_{1} x_{1}} \varphi_{0} ( \frac{ \tau + \vert \xi \vert^{2s}}{2^{j}} ) \varphi^{+}_{k'}(\xi_{1}) \varphi_{k}(\xi) d \xi_{1} \Vert_{L^{1}_{x_{1}} L^{\infty}_{\xi',\tau}} \Vert f \Vert_{Y^{e}_{k,k'}}.
\end{split}
\end{equation*}
Therefore, to obtain part (1) of the lemma, it suffices to prove 
\begin{equation}\label{wtre}
\begin{split}
\Vert \int_{\R} e^{\i \xi_{1} x_{1}} \varphi_{0} ( \frac{ \tau + \vert \xi \vert^{2s}}{2^{j}} ) \varphi^{+}_{k'}(\xi_{1}) \varphi_{k}(\xi) d \xi_{1} \Vert_{L^{1}_{x_{1}} L^{\infty}_{\xi',\tau}} \lesssim 1
\end{split}
\end{equation}
where $ j \leq 2k(s+q-1)-80$. To that end,  by integration by parts we have 
\begin{equation}\label{rrrt}
\begin{split}
& | \int_{\R} e^{\i \xi_{1} x_{1}} \varphi_{0} ( \frac{ \tau + \vert \xi \vert^{2s}}{2^{j}} ) \varphi^{+}_{k'} (\xi_{1} ) \varphi_{k}(\xi) d \xi_{1} |
\\
& \lesssim \vert \int_{\R} \frac{1}{x_{1}} e^{\i \xi_{1} x_{1}} 2^{-j} \xi_{1} ( \xi^{2}_{1} + \vert \xi' \vert^{2})^{s-1} . \varphi'_{0}( \frac{ \tau + \vert \xi \vert^{2s}}{2^{j}} )\varphi^{+}_{k'} ( \xi_{1}) \varphi_{k}(\xi) d\xi_{1} \vert
\\
& + \vert \int_{\R} \frac{1}{x_{1}} e^{\i \xi_{1} x_{1}} \varphi_{0} ( \frac{ \tau + \vert \xi \vert^{2s}}{2^{j}}) 2^{-k'} \varphi'^{+}( \frac{ \xi_{1}}{2^{k'}} ) \varphi_{k}(\xi) d\xi_{1} \vert
\\
& + \vert \int_{\R} \frac{1}{x_{1}} e^{\i \xi_{1} x_{1}} \varphi_{0}( \frac{ \tau + \vert \xi \vert^{2s}}{2^{j}}) \varphi^{+}_{k'}(\xi_{1}) 2^{-k} \partial_{\xi_{1}} \varphi( \frac{\xi}{2^{k}}) d \xi_{1} \vert.
\end{split}
\end{equation}
We deal with the first term on the right hand side of \eqref{rrrt}. We integrate by parts again to obtain 
\begin{equation}\label{rrrt2}
\begin{split}
&  \vert \int_{\R} \frac{1}{x_{1}} e^{\i \xi_{1} x_{1}} 2^{-j} \xi_{1} ( \xi^{2}_{1} + \vert \xi' \vert^{2})^{s-1} . \varphi'_{0}( \frac{ \tau + \vert \xi \vert^{2s}}{2^{j}} )\varphi^{+}_{k'} ( \xi_{1}) \varphi(\frac{\xi}{2^{k}}) d\xi_{1} \vert
\\
& \lesssim \frac{1}{x_{1}^{2}} \left( 2^{-j} 2^{2k (s-1)} +2^{-j} 2^{2k'} 2^{2k (s-2)} +2^{-2j} 2^{2k'} 2^{4k(s-1)} + 2^{-j } 2^{2k(s-1)} \right) 
\\
& \times \vert \{ \xi_{1} \approx 2^{k'} : \vert \tau + \vert \xi \vert^{2s} \vert \leq 2^{j+1} \} \vert 
\\
& \lesssim  \frac{1}{x_{1}^{2}} \left( 2^{-j} 2^{2k (s-1)} +2^{-j} 2^{2k'} 2^{2k (s-2)} +2^{-2j} 2^{2k'} 2^{4k(s-1)} + 2^{-j } 2^{2k(s-1)} \right) 
\\
& \times 2^{j} 2^{-k(2s-2+q)} 
\\
& \lesssim \frac{1}{x^{2}_{1}} \left( 2^{-k'} + 2^{k' -2k} + 2^{-j} 2^{k'} 2^{2k (s-1)} + 2^{-kq} \right)
\\
& \lesssim \frac{1}{x_{1}^{2}} 2^{-j} 2^{k'} 2^{2k(s-1)}.
\end{split}
\end{equation}
Next, for the second term on the right hand side of \eqref{rrrt}, we again integrate by parts to obtain 
\begin{equation}\label{rrrt3}
\begin{split}
& \vert \int_{\R} \frac{1}{x_{1}} e^{\i \xi_{1} x_{1}} \varphi_{0} ( \frac{ \tau + \vert \xi \vert^{2s}}{2^{j}}) 2^{-k'} \varphi'^{+}( \frac{ \xi_{1}}{2^{k'}} ) \varphi(\frac{\xi}{2^{k}}) d\xi_{1} \vert
\\
& \lesssim \frac{1}{x^{2}_{1}} \left( 2^{-j} 2^{2k(s-1)} + 2^{-2k'} \right) 2^{j} 2^{-k(2s-2+q)} 
\\
& \lesssim \frac{1}{x_{1}^{2}} 2^{-j} 2^{k'} 2^{2k(s-1)}.
\end{split}
\end{equation}
Lastly, for the third term on the right hand side of \eqref{rrrt} we apply the same trick, namely we integrate by parts again to obtain 
\begin{equation}\label{rrrt4}
\begin{split}
& \vert \int_{\R} \frac{1}{x_{1}} e^{\i \xi_{1} x_{1}} \varphi_{0}( \frac{ \tau + \vert \xi \vert^{2s}}{2^{j}}) \varphi^{+}_{k'}(\xi_{1}) 2^{-k} \partial_{\xi_{1}} \varphi( \frac{\xi}{2^{k}}) d \xi_{1} \vert
\\
& \lesssim \frac{1}{x_{1}^{2}} 2^{-j} 2^{k'} 2^{2k(s-1)}.
\end{split}
\end{equation}
The bounds \eqref{rrrt2}, \eqref{rrrt3}, and \eqref{rrrt4} imply the following bound 
\begin{equation}
\begin{split}
& | \int_{\R} e^{\i \xi_{1} x_{1}} \varphi_{0} ( \frac{ \tau + \vert \xi \vert^{2s}}{2^{j}} ) \varphi^{+}_{k'} (\xi_{1} ) \varphi_{k}(\xi) d \xi_{1} |
\\
& \lesssim \frac{ 2^{j } 2^{-k'} 2^{-2k(s-1)}}{ 1+ ( 2^{j} 2^{-k'} 2^{-2k(s-1)} x_{1} )^{2}}. 
\end{split}
\end{equation}
Integrating in $ x_{1}$ yields the result. Next, we prove (2) and (3).
We first prove the following estimate for any $ p,q \in [1,\infty]$ and $ e \in \S^{n-1}$
\begin{equation}\label{www8}
\begin{split}
& \Vert \mathcal{F}_{\R^{n+1}}^{-1} ( m \hat{f} ) \Vert_{L^{p}_{e}L^{q}_{e^{\perp},t}} \lesssim \Vert \mathcal{F}_{\R^{n}}^{-1}(m) \Vert_{L^{1}(\R^{n})} \Vert f \Vert_{L^{p}_{e}L^{q}_{e^{\perp},t}}.
\end{split}
\end{equation}
To that end, we write 
\begin{equation*}
\begin{split}
& \mathcal{F}_{\R^{n+1}}^{-1} ( m \hat{f} ) (x,t) 
\\
& = (\mathcal{F}_{\R^{n}}^{-1}(m) * f ( \cdot,t))(x).
\end{split}
\end{equation*}
Then Young inequality and Minkowski yield \eqref{www8}. Clearly \eqref{www8} implies (3). 
Next, take $ m = \vert \nabla \vert^{\beta} \varphi_{l}$. It remains to show
\begin{equation}\label{www9}
\Vert \mathcal{F}_{\R^{n}}^{-1}(m) \Vert_{L^{1}(\R^{n})} \lesssim 2^{l\beta}.
\end{equation}
 But this follows by scaling. 
\end{proof}
\begin{remark}\label{remarkM}
Notice that the above proof implies that for $ \beta \geq 0$ and any $ l \geq 0 $ one has the inequality 
$$
\Vert \Delta_{\leq l} \vert \nabla \vert^{\beta} f \Vert_{L^{p}_{e}L^{q}_{e^{\perp},t}} \lesssim 2^{l \beta} \Vert \Delta_{\lesssim l} f \Vert_{L^{p}_{e} L^{q}_{e^{\perp},t}}.
$$
\end{remark}
An immediate consequence of Lemma \ref{Multiplier} is the following

\begin{lemma}\label{dec}
Let $k \geq 0 $ and $ f \in Z_{k} $. Then there exists $L \in \mathbb{N}$, independent of $k$ and $f$, and functions $ f^{1},.....,f^{L} \in Z_{k}$ with 
$$
{f} = \sum_{i=1}^{L} f^{i}
$$
and for each $ f^{i}$ there exists $ v_{i}  $ with $ \vert v_{i}\vert \in I_{k} \setminus \{0 \}$ such that $f^{i}$  has Fourier transform supported in $ \{ (\xi,\tau) \in \text{supp}(\hat{f}) :  \vert \xi - v_{i} \vert \leq 2^{k-c'} \}$. Moreover, 
$$
\Vert f^{i} \Vert_{Z_{k}} \lesssim \Vert f \Vert_{Z_{k}}
$$
\end{lemma}
\begin{proof}
First, let $k \geq 1$. Then
notice that by definition of $Z_{k}$, each $f \in Z_{k}$ has Fourier transform  supported inside the set 
$$
\{ (\xi,\tau) : 2^{k-1} \leq \vert \xi \vert \leq 2^{k+1} \}.
$$
Then
consider the set $ A_{0} = \{ \xi : 2^{-1} \leq \vert \xi \vert \leq 2^{2} \}$. We can take vectors $ v_{1},..,v_{L}$ in $A_{0}$ so that $ B(v_{i}, 2^{-c'} )$ form a cover of $A_{0}$. Next, take a smooth partition of unity for this cover. In other words, $\phi_{1},..,\phi_{L}$ so that each $ \phi_{i}$ is supported in $B(v_{i},2^{-c'} )$. And 
$$
1= \sum_{i=1}^{L} \phi_{i}
$$
on $A_{0}$. Further, for $k \geq 1$ put $ A_{k} = \{ \xi : 2^{k-1} \leq \vert \xi \vert \leq 2^{k+1} \}$. Then clearly the scaled functions $\phi^{k}_{i}= \phi_{i}(\frac{\cdot}{2^{k}} )$ form a partition for $ A_{k}$. And hence given any function $ f \in Z_{k}$ we have
$$
\hat{f}= \sum \hat{f} \phi^{k}_{i}.
$$
Lastly, we  have by Lemma \ref{Multiplier}
$$
\Vert \mathcal{F}^{-1}(\phi^{k}_{i})* f \Vert_{Z_{k}} \lesssim \Vert f \Vert_{Z_{k}}.
$$
If on the other hand $k =0$. Then we take vectors $ v_{1},...,v_{L}$ with $ v_{i} \neq 0$ for any $ i \in \{1,...,L\}$ and $ B(v_{i},2^{-c'})$ form a cover for $ B(0,1)$. Then we take a smooth partition of unity as in the above, i.e $ \{\phi_{i}\}_{i=1}^{L}$ with each $ \phi_{i}$ supported in $ B(v_{i} , 2^{-c'} )$. Then we set 
$$
f_{i} : = \mathcal{F}^{-1} ( \phi_{i} \hat{f}).
$$

\end{proof}

\section{Main estimates for the resolution spaces}
The main estimates for the resolution spaces are the energy, maximal, and smoothing estimates. In this section we reprove these estimates for our modified spaces $Z_{k}$. The proofs are essentially the same as in \cite{model}, which is also similar to original arguments in \cite{cmp}. But for completeness, we include the proofs here. We start with the following energy estimate. 
\begin{lemma}\label{Energy}
Let $ k \geq 0 $, $f \in Z_{k}$ then we have 
\begin{enumerate}
    \item $$
 \sup_{t} \Vert f( \cdot,t) \Vert_{L^{2}} \lesssim \Vert f \Vert_{Z_{k}}
$$
\item  and we have
$$
\Vert f \Vert_{L^{\infty}_{x,t}} \lesssim 2^{nk/2} \Vert f\Vert_{Z_{k}}
$$
\end{enumerate}
\end{lemma}
\begin{proof}
We only need to prove the first inequality. Indeed, assume we have (1), then for (2) the proof is as follows; since $ f \in Z_{k}$ then we have $ f = \Delta_{\lesssim {k}} f$. Therefore, we have 
\begin{equation*}
\begin{split}
& \vert  f(x,t) \vert = \vert \int_{\R^{n+1}} e^{x . \xi i + t \tau i} \hat{f}(\xi,\tau) d \xi d\tau \vert  
\\
& \leq 2^{nk/2}  \Vert  \int_{ \tau}  e^{it \tau} \hat{f} (\cdot, \tau) d \tau \Vert_{L^{2}_{\xi}}
\\
& = 2^{nk/2} \Vert f ( \cdot, t) \Vert_{L^{2}_{x}} 
\\
& \lesssim 2^{nk/2} \Vert f \Vert_{Z_{k}}.
\end{split}
\end{equation*}
Therefore, we proceed to prove (1). Let $ f \in X_{k}$ then by Plancherel it is enough to prove the following inequality for any $t\in \R$ 
\begin{equation}\label{w19}
\begin{split}
\Vert \int_{\R} e^{it\tau} \varphi_{j}(\tau + \vert \xi \vert^{2s}) \hat{f}(\xi,\tau) d \tau \Vert_{L^{2}_{\xi}} \lesssim 2^{j/2} \Vert Q_{j}(f) \Vert_{L^{2}_{\xi,\tau}}
\end{split}
\end{equation}
this follows easily by H\"{o}lder inequality. Indeed, 
\begin{equation*}
\begin{split}
& \Vert \int_{\R} e^{it\tau} \varphi_{j}(\tau + \vert \xi \vert^{2s}) \hat{f}(\xi,\tau) d \tau \Vert_{L^{2}_{\xi}} 
\\
& \leq \Vert \varphi_{j}(\tau + \vert \xi \vert^{2s}) \hat{f} \Vert_{L^{2}_{\xi} L^{1}_{\tau}}
\\
& \leq 2^{j/2} \Vert Q_{j}(f) \Vert_{L^{2}_{\xi,\tau}}
\end{split}
\end{equation*}
this gives \eqref{w19}. Next, assume $k \geq 100$ and $ f \in Y^{e}_{k,k'}$. Then in light of Lemma \ref{Embeddings} we may assume that $ k' \in T_{k}$ and $ \hat{f}$ is supported in the set $D^{e, k'}_{k , \leq 2k(s+q-1)}$. Fix such an $f$.
Then we want to prove 
\begin{equation}\label{ww}
\Vert \int_{\R} e^{\i \tau t} \hat{f}(\xi,\tau) d\tau \Vert_{L^{2}_{\xi}} \lesssim \Vert f \Vert_{Y^{e}_{k,k'}}.
\end{equation}
To that end, define 
$$
h(x,t) : = \mathcal{F}^{-1}_{\R^{n+1}}  \big[( \tau + \vert \xi \vert^{2s} + \i ) \hat{f}\big]
$$
then by definition
$$
\Vert h \Vert_{L^{1}_{e} L^{2}_{e^{\perp},t}} \lesssim 2^{k'(2s-1)/2} \gamma^{-1}_{k,k'} \Vert f \Vert_{Y^{e}_{k,k'}}.
$$
Thus, we will obtain \eqref{ww} once we prove 
\begin{equation}
\begin{split}
& \Vert \int_{\R} e^{\i \tau t} \frac{ \varphi_{ \leq 2k(s + q - 1) } (\tau + \vert \xi \vert^{2s}) }{\tau + \vert \xi \vert^{2s} +\i} \mathcal{F}(h)(\xi,\tau) d \tau \Vert_{L^{2}_{\xi}}
\\
& \lesssim 2^{-k'(2s-1)/2} \gamma_{k,k'} \Vert h \Vert_{L^{1}_{e} L^{2}_{\xi',\tau}}
\end{split}
\end{equation}
for any $ t \in \R$. This inequality follows once we prove the stronger inequality 
\begin{equation}\label{ww1}
\begin{split}
& \Vert \int_{\R}  \varphi^{+}_{[k'-1,k'+1]} (\xi_{1}) e^{\i \tau t} \frac{ \varphi_{ \leq 2k(s + q - 1)  }(\tau+ \vert \xi \vert^{2s}) }{\tau + \vert \xi \vert^{2s} +\i} \tilde{h}(\xi',\tau) d\tau \Vert_{L^{2}_{\xi}}
\\
& \lesssim 2^{- k'(2s-1)/2} \gamma_{k,k'} \Vert \tilde{h} \Vert_{L^{2}_{\xi',\tau}}
\end{split}
\end{equation}
for any function $ \tilde{h}: e^{\perp} \times \R \to \C$ in $L^{2}(e^{\perp} \times \R)$ supported in $ \vert \xi' \vert \lesssim 2^{k}$, and any $t \in \R$. By translating the function $\tilde{h}$ if necessary, we may assume $ t =0$. Then define 
$$
h'(\xi', \mu) : = \int_{\R} \frac{ \varphi_{ \leq 2k(s + q - 1)  } (\tau +\mu) }{\tau + \mu +\i} \tilde{h}(\xi',\tau) d \tau
$$
by Boundedness of the Hilbert transform we have that 
$$
\Vert h' \Vert_{L^{2}_{\xi',\mu}} \lesssim \Vert h \Vert_{L^{2}_{\xi',\tau}}
$$
then we reduced proving \eqref{ww1} into proving 
\begin{equation}\label{ww2}
\begin{split}
& \Vert \varphi_{[k'-1,k'+1]}^{+}(\xi_{1}) h'(\xi', ( \xi_{1}^{2} + \vert \xi' \vert^{2} )^{s}) \Vert_{L^{2}_{\xi}}
\\
& \lesssim 2^{-k'(2s-1)/2} \gamma_{k,k'} \Vert h' \Vert_{L^{2}_{\xi',\mu}}
\end{split}
\end{equation}
to prove \eqref{ww2} we apply a change of variable as follows; for each fixed $ \xi' $ we define $ \mu : = ( \xi_{1}^{2}+ \vert \xi'\vert^{2} )^{s}$. Then $ d \mu = 2s \xi_{1} ( \xi_{1}^{2} + \vert \xi' \vert^{2} )^{s-1} d\xi_{1} $. Put $ \alpha(\xi_{1}) : =2s \xi_{1} ( \xi_{1}^{2} + \vert \xi' \vert^{2} )^{s-1}$. Then clearly $\frac{1}{\alpha} \lesssim 2^{-k'} 2^{2k(1-s)}$. Then we carry out the integral in $\xi_{1}$ and use this change of variable to obtain 
\begin{equation*}
\begin{split}
& \Vert \varphi_{[k'-1,k'+1]}^{+}(\xi_{1}) h'(\xi', ( \xi_{1}^{2} + \vert \xi' \vert^{2} )^{s}) \Vert_{L^{2}_{\xi}}
\\
& \lesssim 2^{-k'(2s-1)/2} \gamma_{k,k'} \Vert  h' \Vert_{L^{2}_{\xi',\mu}}
\end{split}
\end{equation*}
this gives \eqref{ww2} and yields the result.

\end{proof}
\subsection{Local smoothing, maximal estimates}
In this subsection we prove the local smoothing and maximal estimates, these two inequalities were already proven in \cite{model} for any $ s \in (1/2,1)$. Below we reproduce those proofs for the modified $Z_{k}$ spaces. 
\begin{lemma}\label{localsmoothing}
Let $ k  \geq 0$ and $ f \in Z_{k}$. Then for any $ e \in \S^{n-1}$ the following inequality holds
$$
\Vert \mathcal{F}^{-1}_{\R^{n+1}} \left( \vartheta_{e} \hat{f} \right) \Vert_{L^{\infty}_{e} L^{2}_{e^{\perp},t}} \lesssim 2^{-k(2s-1)/2} \Vert f \Vert_{Z_{k}}
$$
\end{lemma}
\begin{proof}
If $ f \in X_{k}$ then the proof can be found in \cite[Lemma 3.7]{model}. Assume next that $ f \in Y^{e'}_{k,k'}$ for some $e' \in \mathscr{E}$. Then in light of Lemma \ref{Embeddings} we may assume $k' \in T_{k}$ and $f$ has modulation controlled by $ 2^{2k(s+q -1) -80 }$. Further, since $k \geq 100$ then 
$$
\vartheta_{e}(\xi) = \vartheta_{e}(\xi) \varphi^{+}_{[k-c,k+c]} ( \xi \cdot e ) \in C^{\infty}_{0}(\R^{n})
$$
and so it suffices to prove the following 
\begin{equation}\label{azx}
\begin{split}
& \Vert \mathcal{F}^{-1}_{\R^{n+1}} \left( \varphi^{+}_{[k-c,k+c]}(\xi \cdot e) \hat{f} \right) \Vert_{L^{\infty}_{e} L^{2}_{e^{\perp},t}} 
\\
& \lesssim 2^{-k(2s-1)/2} \Vert f \Vert_{Y^{e'}_{k,k'}}.
\end{split}
\end{equation}
To that end, write  $\xi = \xi_{1} e' + \xi'$ with $ \xi' \in e'^{\perp}$. Then carrying out the Fourier transform of left hand side of \eqref{azx} we obtain
\begin{equation}\label{azx1}
\begin{split}
& \Vert \mathcal{F}^{-1}_{\R^{n+1}} \left( \varphi^{+}_{[k-c,k+c]}(\xi \cdot e) \hat{f} \right) \Vert_{L^{\infty}_{e} L^{2}_{e^{\perp},t}}
\\
& = \Vert \int_{\R^{n+1}} e^{\i x \cdot \xi + \i \tau t } \varphi_{[k-c,k+c]}^{+}(\xi \cdot e ) \hat{f}(\xi,\tau) d\xi d\tau \Vert_{L^{\infty}_{e} L^{2}_{e^{\perp},t}}
\\
& \lesssim \Vert \int_{\R^{n+1}} e^{\i x \cdot \xi + \i \tau t } \varphi_{[k-c,k+c]}^{+}( (N(\xi',\tau) e' + \xi') \cdot e ) \hat{f}(\xi,\tau) d\xi d\tau \Vert_{L^{\infty}_{e} L^{2}_{e^{\perp},t}}
\\
&+ \Vert \int_{\R^{n+1}} e^{\i x \cdot \xi + \i \tau t } \big[ \varphi^{+}_{[k-c,k+c]} ( \xi \cdot e )- \varphi_{[k-c,k+c]}^{+}( (N(\xi',\tau) e'+\xi' ) \cdot e) \big] \hat{f}(\xi,\tau) d\xi d\tau \Vert_{L^{\infty}_{e} L^{2}_{e^{\perp},t}}.
\end{split}
\end{equation}
We claim that the second term on the right hand side of \eqref{azx1} satisfy the desired inequality. Indeed, we will show
\begin{equation}\label{azx2}
\begin{split}
& \Vert \int_{\R^{n+1}} e^{\i x \cdot \xi + \i \tau t } \big[ \varphi^{+}_{[k-c,k+c]} ( \xi \cdot e )- \varphi_{[k-c,k+c]}^{+}(( Ne'+\xi' ) \cdot e) \big] \hat{f}(\xi,\tau) d\xi d\tau \Vert_{X_{k}}
\\
& \lesssim \Vert f \Vert_{Y^{e'}_{k,k'}},
\end{split}
\end{equation}
and then the smoothing estimate follows by the fact that $X_{k}$ satisfy the estimate. 
To prove \eqref{azx2} notice the bound, which follows from Lemma \ref{MultiplierN}, 
\begin{equation}
\begin{split}
& \vert \varphi^{+}_{[k-c,k+c]} ( \xi \cdot e )- \varphi_{[k-c,k+c]}^{+}( (N e' +\xi' ) \cdot e) \vert 
\\
& \lesssim \min \{1, 2^{-k(2s-1+q)} \vert \tau + \vert \xi \vert^{2s} \vert \},
\end{split}
\end{equation}
and using this we obtain
\begin{equation*}
\begin{split}
& 2^{j/2} \Vert \varphi_{j}(\tau +\vert \xi \vert^{2s} ) \vert \varphi^{+}_{[k-c,k+c]} ( \xi \cdot e )- \varphi_{[k-c,k+c]}^{+}( (Ne' + \xi' )\cdot e)  \hat{f} \Vert_{L^{2}}
\\
& \lesssim 2^{j/2} 2^{-k(2s -1 +q)} \Vert \varphi_{j}(\tau +\vert \xi \vert^{2s} ) \mathcal{F}_{e'^{\perp,t}} \left( \int_{\R} e^{\i \xi_{1} x_{1}} ( \i \partial_{t} -(-\Delta)^s + \i) f dx_{1}  \right) \Vert_{L^{2}_{\xi,\tau}}
\\
& \lesssim 2^{j/2} 2^{-k(2s -1 +q)} \Vert ( \i \partial_{t} -(-\Delta)^s + \i) f \Vert_{L^{1}_{e'} L^{2}_{e'^{\perp},t}} \vert \{ \xi_{1} :  \xi_{1}  \approx 2^{k'} , \vert \tau + \vert \xi \vert^{2s} \lesssim 2^{j} \} \vert^{1/2} 
\\
& \lesssim 2^{j/2} 2^{-k(2s -1 +q)} \Vert ( \i \partial_{t} -(-\Delta)^s + \i) f \Vert_{L^{1}_{e'} L^{2}_{e'^{\perp},t}} 2^{j/2} 2^{-k(2s-2+q)/2}
\end{split}
\end{equation*}
summing over $j \leq 2k(s+q-1)$ we obtain
\begin{equation*}
\begin{split}
&  \Vert \int_{\R^{n+1}} e^{\i x \cdot \xi + \i \tau t } \big[ \varphi^{+}_{[k-c,k+c]} ( \xi_{1} )- \varphi_{[k-c,k+c]}^{+}(( N e' + \xi')\cdot e) ) \big] \hat{f}(\xi,\tau) d\xi d\tau \Vert_{X_{k}}
\\
& \lesssim 
2^{2k(s-1+q)} 2^{-k(2s -1 +q)}2^{-k(2s-2+q)/2}
\Vert ( \i \partial_{t} -(-\Delta)^s + \i) f \Vert_{L^{1}_{e'} L^{2}_{e'^{\perp},t}} 
\\
& \lesssim \Vert f \Vert_{Y^{e'}_{k,k'}}.
\end{split}
\end{equation*}
This proves \eqref{azx2} and reduces the lemma to showing 
\begin{equation}\label{azx3}
\begin{split}
& \Vert \int_{\R^{n+1}} e^{\i x \cdot \xi + \i \tau t } \varphi_{[k-c,k+c]}^{+}( (Ne'+ \xi') \cdot e ) \hat{f}(\xi,\tau) d\xi d\tau \Vert_{L^{\infty}_{e} L^{2}_{e^{\perp},t}}
\\
& \lesssim 2^{-k(2s-1)/2} \Vert f \Vert_{Y^{e'}_{k,k'}}.
\end{split}
\end{equation}
Further, notice that since $ \vert \xi \vert \approx 2^{k}$ in the support of $ \hat{f}$ and since $ N \approx \xi_{1}$ then $  N^{2} + \vert \xi' \vert^{2} \approx 2^{2k}$ in the support of $ \hat{f}$. Therefore, \eqref{azx3} follows once we prove
\begin{equation}\label{asda1}
\begin{split}
&  \Vert \int_{\R^{n+1}} e^{\i x \cdot \xi + \i \tau t } \varphi_{[2k-c,2k+c]}(N^{2} + \vert \xi' \vert^{2}) (  \varphi_{[k-c,k+c]}^{+}( (Ne'+ \xi') \cdot e ) \hat{f}(\xi,\tau) d\xi d\tau \Vert_{L^{\infty}_{e} L^{2}_{e^{\perp},t}}
\\
& \lesssim 2^{-k(2s-1)/2} \Vert f \Vert_{Y^{e'}_{k,k'}}.
\end{split}
\end{equation}

To that end, we use the representation formula, Lemma \ref{repf}, and carry out the integral in $ \xi_{1}$ to reduce matters to proving 
\begin{equation}\label{azx4}
\begin{split}
& \Vert \int_{e^{\perp} \times \R} e^{ \i N \xi_{1} +\i x' \cdot \xi' + \i t \tau} \varphi_{[2k-c,2k+c]} (N^{2}+\vert \xi' \vert^{2}) 
\\
& \times \varphi^{+}_{[k'-1,k'+1]}(N) \varphi^{+}_{[k-c,k+c]}(( Ne' + \xi' ) \cdot e)) h(\xi',\tau) d\xi' d \tau \Vert_{L^{\infty}_{e}L^{2}_{e^{\perp},t}} 
\\
& \lesssim ( 2^{-2(1-s)(k-k')} \gamma_{k,k'}) 2^{(k'-k) (2s-1)/2} \Vert h \Vert_{L^{2}_{\xi',\tau}}
\end{split}
\end{equation}
for any $ h \in L^{2}(e^{\perp} \times \R)$ supported in $ \{ \vert \xi' \vert \lesssim 2^{k} \}$. 
To prove \eqref{azx4} we use the change of variable 
$
\tau : = -(\mu^{2} + \vert \xi' \vert^{2})^{s}
$
and so $N = \mu$. This is a valid change of variable since $ \mu \approx 2^{k'}$. Then using this change of variable, we set
$$
h'(\xi',\mu) : = \varphi^{+}_{[k'-c,k'+c]} (\mu) \varphi_{[2k-c,2k+c]} ( \mu^{2} +\vert \xi' \vert^{2}) \mu ( \mu^{2} +\vert \xi' \vert^{2})^{s-1} h(\xi, - ( \mu^{2} +\vert \xi' \vert^{2})^{s} ).
$$ And notice 
$$
\Vert h' \Vert_{L^{2}_{\xi',\mu}} \lesssim 2^{k'/2} 2^{k(s-1)} \Vert h \Vert_{L^{2}_{\xi',\tau}}.
$$
And since 
$$
2^{k'/2} 2^{k(s-1)} \lesssim 2^{-2(1-s)(k-k')} \gamma_{k,k'} 2^{k'(2s-1)/2},
$$
then for \eqref{azx4} it suffices to prove
\begin{equation}\label{azx5}
\begin{split}
& \Vert \int_{e^{\perp \times \R}} e^{ \i \mu \xi_{1} +\i x' \cdot \xi'  - \i t (\mu^{2} +\vert \xi' \vert^{2})^{s}} \varphi_{[k-c,k+c]}( (\mu e' + \xi')\cdot e) h'(\xi',\mu) d\xi' d \mu \Vert_{L^{\infty}_{e}L^{2}_{e^{\perp},t}} 
\\
& \lesssim 2^{-k(2s-1)/2} \Vert h' \Vert_{L^{2}_{\xi',\mu}}
\end{split}
\end{equation}
for any $h' \in L^{2}_{\xi',\mu}$ supported in $ \{ (\xi,\mu) \in e^{\perp} \times \R : | ( \xi',\mu ) | \lesssim 2^{k} \}.$
However, the inequality \eqref{azx5} was already proven in \cite[(3.28)]{model}. This concludes the lemma. 

\end{proof}
We close this section with the proof of the maximal estimate. Again, the proof is almost identical to one with the old $Z_{k}$ spaces in \cite{model}. We will merely sketch the differences needed for $Y^{e}_{k,k'} $ when $k' \neq k$. 
\begin{lemma}
Let $n \geq 3$, $k \geq 0$ and $ e \in \S^{n-1}$ be any direction. Then the following inequality holds
\begin{equation*}
\begin{split}
& \Vert f \Vert_{L^{2}_{e} L^{\infty}_{e^{\perp},t}} \lesssim 2^{k (n-1)/2} \Vert f \Vert_{Z_{k}}
\end{split}
\end{equation*}
\end{lemma}
\begin{proof}
If $ f \in X_{k}$ then the estimate was proven in \cite[Lemma 3.8]{model}. So assume in what follows that $ f \in Y^{e'}_{k,k'}$ with $k' \in T_{k}$ and $f$ has modulation controlled by $ 2^{2k(s-1+q) -80}$. Then we need to prove 
\begin{equation}\label{azx7}
\begin{split}
& \Vert \int_{\R^{n+1}} e^{\i x \cdot \xi + \i t \tau}  \hat{f}(\xi,\tau) d \xi d\tau \Vert_{L^{2}_{e}L^{\infty}_{e^{\perp},t}}
\\
& \lesssim 2^{k(n-1)/2} \Vert f \Vert_{Y^{e'}_{k,k'}}.
\end{split}
\end{equation}
To prove \eqref{azx7} write $ \xi = \xi_{1} e' + \xi'$ with $ \xi' \in e'^{\perp}$. Then notice that in the support of $ \hat{f}$ we have $ N^{2} + \vert \xi' \vert^{2} \approx 2^{2k}$. Therefore, \eqref{azx7} follows if we prove
\begin{equation}\label{azxx1}
\begin{split}
&  \Vert \int_{\R^{n+1}} e^{\i x \cdot \xi + \i t \tau}  \varphi_{[2k-c,2k+c]} (N^{2} +\vert \xi' \vert^{2}) \hat{f}(\xi,\tau) d \xi d\tau \Vert_{L^{2}_{e}L^{\infty}_{e^{\perp},t}}
\\
& \lesssim 2^{k(n-1)/2} \Vert f \Vert_{Y^{e'}_{k,k'}}.
\end{split}
\end{equation}
To that end, we use 
Lemma \ref{repf} and carry out the integral in $ \xi_{1}$. This reduces  \eqref{azxx1} to proving 
\begin{equation}\label{azx8}
\begin{split}
& \Vert \int_{ e^{\perp} \times \R} e^{\i \xi' \cdot x' } e^{\i \tau t} e^{ N x_{1} \i} \varphi^{+}_{k'}(N) \varphi_{[2k-c,2k+c]}(N^{2} + \vert \xi' \vert^{2}) h(\xi',\tau) d \xi' d \tau \Vert_{L^{2}_{e} L^{\infty}_{e^{\perp},t}}
\\
& \lesssim 2^{-2(1-s) (k-k')} \gamma_{k,k'} 2^{k(n-1)/2} 2^{k'(2s-1)/2} \Vert h \Vert_{L^{2}_{\xi',\tau}}
\end{split}
\end{equation}
for any $h \in L^{2}(e^{\perp} \times \R)$ supported in $ \{ \vert \xi' \vert \lesssim 2^{k} \}$. To prove \eqref{azx8} introduce the change of variable $ \tau = - ( \mu^{2} + \vert \xi' \vert^{2})^{s}$ and set
$$
h'(\xi',\mu) : = \varphi^{+}_{k'}(\mu) \varphi_{[2k-c,2k+c]} ( \mu^{2} + \vert \xi' \vert^{2}) \mu ( \mu^{2} + \vert \xi' \vert^{2})^{s-1} h(\xi',-(\mu^{2} + \vert \xi' \vert^{2})^{s} ).
$$ Then notice that 
$$
\Vert h' \Vert_{L^{2}_{\xi',\mu}} \lesssim  2^{k'/2 + k(s-1)} \Vert h \Vert_{L^{2}_{\xi,\tau}}.
$$
But since 
$$
2^{k'/2 + k(s-1)} \lesssim 2^{-2(1-s)(k-k')} \gamma_{k,k'} 2^{k'(2s-1)/2}
$$
then \eqref{azx8} will follow once we prove 
\begin{equation}\label{azx9}
\begin{split}
& \Vert \int_{ e^{\perp} \times \R} e^{\i \xi' \cdot x' } e^{- \i t ( \mu^{2} + \vert \xi' \vert^{2})^{s}} e^{ \mu x_{1} \i}  h'(\xi',\mu) d \xi' d \tau \Vert_{L^{2}_{e} L^{\infty}_{e^{\perp},t}} 
\\
& \lesssim 2^{k(n-1)/2} \Vert h' \Vert_{L^{2}_{\xi',\mu}}
\end{split}
\end{equation}
for any $ h' \in L^{2}_{\xi',\mu}$ supported in $ \{ (\xi',\mu) \in e^{\perp} \times \R: \vert (\xi', \mu ) \vert \lesssim 2^{k} \}. $
The inequality \eqref{azx9} was proven in \cite[Theorem B1]{model}. 
\end{proof}

\section{Commutator estimates in resolution spaces}
In this section we prove several commutator estimates between the resolution spaces that will be used extensively in the proof of the nonlinear estimates. First, we define the commutator operators under consideration. 
$$
H_{s}(f,g) : = (-\Delta)^s( fg ) - (-\Delta)^s(f) g - f (-\Delta)^s(g).
$$
As we will see, if $ g $ has  frequency comparable to $f$ then the factor $ f (-\Delta)^s(g)$ will satisfy certain estimates trivially, for this reason we define another commutator as follows
\begin{equation}\label{onesidedcommutator}
\tilde{H}(f,g) := (-\Delta)^s(fg) - (-\Delta)^s(f) g.
\end{equation}
To motivate the results in this section we re-examine the nonlinearity in \eqref{reductionmain} again.
If one replaces the factor $ \frac{1}{1+\vert f \vert^{2}}$  with $ f$ in the second nonlinear term, then the nonlinearity is essentially of the form $ f H_{s}(g,h)$ which is a Trilinear expression. Then by using Taylor expansion and considering several cases, one notices that, in the relevant cases, $ f H_{s}(g,h)$ behaves more or less like $ f (\vert \nabla \vert^{2s-1} g )\vert \nabla \vert h$. Then the plan is to place the high frequency in the smoothing space $ L^{\infty}_{e} L^{2}_{e^{\perp},t}$ and the remaining terms in the maximal space $L^{2}_{e}L^{\infty}_{e^{\perp},t}$. This gives us a gain of $ (2s-1)/2$ derivatives. But since the product is in $Y^{e}_{k,k'}$ this will give the other $ (2s-1)/2$ gain in derivative. 
However, the presence of the expressions $ \frac{1}{1+ \vert f \vert^{2}} $ and $ H_{s}(f, \frac{1}{1+|f|^{2}})$ complicate the problem. For that reason the lemmata in this section are stated with these terms in mind.\\
We start with the following simple estimate which follows by Taylor expansion.
\begin{lemma}\label{Taylorstuff}
Let $ f,g \in L^{2}(\R^{n+1}), $ $k_{1} ,k_{2} \geq 0$ with $ k_{2} \leq k_{1} -10 n$, and assume $ \hat{f} $ is supported in $ \{ ( \xi,\tau) : \vert \xi \vert \in I_{k_{1}} \} \times \R$ and $ \hat{g} $ is supported in $\{ \vert \xi \vert \leq 2^{k_{2}} \} \times \R $. Then for any $ p,q \in [1,\infty]$ the following estimate holds 
$$
\Vert \Delta_{k_{1}} \tilde{H}(f,g) \Vert_{L^{p}_{e}L^{q}_{e^{\perp},t}} \lesssim \sum_{\vert \alpha \vert =1}^{\infty} \frac{a_{\vert \alpha \vert}}{\alpha !} 2^{k_{1} (2s -\vert \alpha \vert)} \Vert f D^{\alpha}(g) \Vert_{L^{p}_{e} L^{q}_{e^{\perp},t}}
$$
where the constants $ a_{\vert \alpha \vert}$ satisfy the following 
$$
\sum_{\vert \alpha \vert=1 }^{\infty} \frac{ a_{\vert \alpha \vert}}{\alpha !} 2^{ -10 n \vert \alpha \vert} \lesssim 1
$$
\end{lemma}
\begin{proof}
For any multi-index $\alpha$, write $m$ to mean $ m: = \vert \alpha \vert$. 
We write $ \Delta_{k_{1}} \tilde{H}(f,g)$ in the Fourier space
\begin{equation*}
\begin{split}
& \mathcal{F} ( \Delta_{k_{1}} \tilde{H}(f,g))(\xi,\tau)
\\
& = \varphi_{k_{1}}(\xi) \int_{\R^{n+1}} ( \vert \xi \vert^{2s} - \vert \xi - \eta \vert^{2s} ) \hat{f} (\xi-\eta,\tau -r ) \hat{g} ( \eta,r) d \eta dr  
\end{split}
\end{equation*}
fix $ \xi$ and define the function $ h: \R^{n} \to \R$ as follows
$$
h(\eta) : = \vert \xi + \eta \vert^{2s}
$$
since this function is smooth at $0$ we can use Taylor theorem to obtain 
\begin{equation}\label{c2}
\vert \xi \vert^{2s} - \vert \xi - \eta \vert^{2s} = (-1)^{\vert \alpha \vert +1} \sum_{ \vert \alpha \vert = 1}^{\infty} \frac{1}{ \alpha !} D^{\alpha}h(0). \eta^{\alpha}
\end{equation}
We need to compute $ D^{\alpha}h(0)$. 
Write $ D^{\alpha} = \partial_{i_{m}} .... \partial_{i_{1}} $ with each $ i_{k} \in \{ 1,...,n\}$. To write $D^{\alpha}h(0)$ in closed form, we need to introduce the following notation:
Let $k \in [ 0,...,[\frac{m}{2}] ]$, $[\frac{m}{2}]$ is the greatest integer less than or equal to $\frac{m}{2}$, and consider the set $P_{k}$ as the set of all distinct $k$ pairings of indices of $\{ i_{1},...,i_{m} \}$. In other words, 
$$
P_{k}: = \{ \{ \{ i_{t_{1}},i_{r_{1}}\},...,\{i_{t_{k}},i_{r_{k}}\} \} : t_{q} \neq r_{q} \text{ and } \{i_{t_{q}},i_{r_{q}}\} \neq \{ i_{t_{q'}},i_{r_{q'}} \} \} 
$$
simple computations give us that the cardinality of $P_{k}$ is 
$$
|P_{k} | = \frac{ m!}{(m-2k)! 2^{k} k!}.
$$
Next, if $p \in P_{k}$ then $p = \{ \{i_{t_{1}},i_{r_{1}}\},..., \{i_{t_{k}}, i_{r_{k}} \} \}$ with $i_{t},i_{r}$ satisfying the above constraints. Then define   
$$
S(p) = \{ i_{1},...,i_{m} \} \setminus \{ i_{t_{1}} ,i_{r_{1}},....,i_{t_{k}}, i_{r_{k}} \}.
$$
then $S(p)$ has cardinality $ m -2k$. Lastly, define for any $j \in S(p)$ the following 
$$
R^{j}: = \frac{\xi_{j}}{\vert \xi \vert}
$$
and for any $p \in P_{k}$ define $ \delta^{p} $ as the product of the corresponding kronecker deltas. More precisely,
$$
\delta^{p} : = \delta_{i_{q_{1}},i_{r_{1}}}.... \delta_{i_{q_{k}},i_{r_{k}}}.
$$
With this notation we claim that $D^{\alpha} h(0)$ is of the following form
\begin{equation}\label{www4}
\begin{split}
& D^{\alpha} h(0)
\\
& = \vert \xi \vert^{2s - m} \big[C^{0}_{s,m} R^{i_{1}}...R^{i_{m}}+ \sum_{k=1}^{ [m/2]} C^{k}_{s,m} \sum_{p \in P_{k}} \delta^{p}  \prod_{j \in S(p)} R^{j}  \big]
\end{split}
\end{equation}
where
$$
C^{j}_{s,m} : = 2s ( 2s-2)(2s-4)....(2s-2(m-j-1)).
$$
This formula can be proven via induction. If $ | \alpha | =1 $ then $ D^{\alpha} = \partial_{i} $ for some $ i \in \{1,...,n \}$ and clearly we have 
$$
D^{\alpha} h(0) = 2s \vert  \xi  \vert^{2s-1} \frac{ \xi_{i}}{\vert \xi \vert}
$$
which is consistent with \eqref{www4}. Assume the formula \eqref{www4} holds for any multi-index with $ \vert \alpha \vert = m$. Consider a multi-index $ \alpha $ with $ | \alpha | = m+1$. Then write $ D^{\alpha} = \partial_{i_{m+1}} D^{\alpha'} $ where $ |\alpha' | = m$. Then by the induction hypothesis we have
\begin{equation}\label{www5}
\begin{split}
& D^{\alpha} h(0) 
\\
&= \partial_{i_{m+1}} ( \vert \xi \vert^{2s - m} \big[C^{0}_{s,m} R^{i_{1}}...R^{i_{m}}+ \sum_{k=1}^{ [m/2]} C^{k}_{s,m} \sum_{p \in P_{k}} \delta^{p}  \prod_{j \in S(p)} R^{j}  \big]
 )
 \\
 & = \partial_{i_{m+1}} ( \vert \xi \vert^{2s-2m} C_{s,m}^{0} \prod_{k=1}^{m} \xi_{i_{k}} ) + \partial_{i_{m+1}} ( \vert \xi \vert^{2s-m}  \big[ \sum_{k=1}^{[m/2]} C_{s,m}^{k} \sum_{p \in P_{k}} \delta^{p} \prod_{j \in S(p)} R^{j} \big]
 \\
 & = \vert \xi \vert^{2s-m-1}  C^{0}_{s,m+1} R^{i_{1}}...R^{i_{m+1}}
  + \vert \xi \vert^{2s-m-1} \sum_{ a=1}^{m} C^{1}_{s,m+1} \delta_{i_{m+1},i_{a}} \prod_{ j \in S ( \{ i_{m+1},i_{a} \})} R^{j} 
 \\
 & + \sum_{k=1}^{ [m/2]} C_{s,m}^{k} \sum_{p \in P_{k}} \delta^{p} \partial_{i_{m+1}}(  \vert \xi \vert^{2s-m} \prod_{j \in S(p)} R^{j} ) 
\end{split}
\end{equation}
and notice that for any $ p \in P_{k}$ we have 
\begin{equation}\label{www6}
\begin{split}
& \partial_{i_{m+1}} ( \vert \xi \vert^{2s -m } \prod_{j \in S(p)} R^{j} )
\\
& = \partial_{i_{m+1}} ( \vert \xi \vert^{2s - 2m +2k} \xi_{i_{q_{1}}} .... \xi_{q_{m-2k}} ) 
\\
& = (2s-2m+2k) \vert \xi \vert^{2s-m-1} R^{i_{q_{1}}}...R^{i_{q_{m-2k}}} R^{i_{m+1}}  + \vert \xi \vert^{2s-m-1} \sum_{j=1}^{m-2k} \delta_{i_{m+1}, i_{q_j}} \prod_{t \in S(\{i_{m+1},i_{q_{j}} \} \cap p) } R^{t}
\end{split}
\end{equation}
plugging \eqref{www6} back into \eqref{www5} yields that $ D^{\alpha} h(0)$ is of the form given in \eqref{www4}.

We rewrite $ \tilde{H}(f,g)$ using this Taylor expansion to obtain, using Lemma \ref{Multiplier}, 
\begin{equation}\label{www3}
\begin{split}
& \Vert \Delta_{k_{1}} \tilde{H}(f,g) \Vert_{L^{p}_{e}L^{q}_{e^{\perp},t}}
\\
& \lesssim \sum_{\vert \alpha \vert =1}^{\infty} \frac{a_{m}}{\alpha !} 2^{k_{1}(2s-m)} \Vert f D^{\alpha} g \Vert_{L^{p}_{e} L^{q}_{e^{\perp},t}}
\end{split}
\end{equation}
where 
\begin{equation}
a_{m} : = \sum_{k=0}^{[m/2]} | C_{s,m}^{k} | | P_{k}| 
\end{equation}
and it remains to show that the series has the claimed convergence. To that end, we first bound $a_{m}$. By definition, we had 
$$
a_{m} = \sum_{k=0}^{[m/2]} |C^{k}_{s,m}| \vert P_{k} \vert
$$
We bound $C^{k}_{s,m}$. Recall that 
$C^{k}_{s,m} = 2s (2s-2) ....(2s-2(m-k-1))$.
an obvious bound is 
$$
| C^{k}_{s,m}| \lesssim_{s} 2^{m-k} (m- k)!
$$
and $|P_{k} |$ was explicit, we had
$$
|P_{k}| = \frac{m!}{(m-2k!) 2^{k} k!}
$$
therefore,
$$
|C^{k}_{s,m}| |P_{k} | \lesssim m!  \frac{2^{m-2k}}{k!} \frac{ (m-k!)}{(m-2k)!}
$$
hence,
$$
\frac{a_{m} 2^{-10(n)m}}{\alpha !} \lesssim  \frac{m!}{\alpha !} 2^{(-10n+1)m} \sum_{k=0}^{[m/2]} \frac{2^{-2k}}{k!} \frac{(m-k)!}{(m-2k)!}
$$
using the bound $ \frac{m!}{\alpha !} \lesssim (n)^{m}$ then we obtain
\begin{equation}\label{z1}
    \frac{a_{m} 2^{-10(n)m}}{\alpha !} \lesssim  (n)^{m} 2^{(-10n+1)m} \sum_{k=0}^{[m/2]} \frac{2^{-2k}}{k!} \frac{(m-k)!}{(m-2k)!}
 \end{equation}
the right hand side can be bounded by 
$$
(n 2^{(-10n +1)} e)^{m}
$$
therefore summing over $m$, the series converges.

\end{proof}

 The following lemma gives us control of the $L^{2} $ norm of the commutator in terms of the $Z_{k}$ norm
\begin{lemma}\label{comest}
Let $ k_{1}, k_{2} \geq 0 $ with $ k_{1} \leq k_{2}+10$. Let $f_{k_{1}} \in Z_{k_{1}}$ and $f_{k_{2}} \in Z_{k_{2}}$. Then the following estimate holds
\begin{equation}\label{a}
\begin{split}
& \Vert  H_{s}( \tilde{f}_{k_{1}},  \tilde{f}_{k_{2}}) \Vert_{L^{2}(\R^{n+1})} 
\\
& \lesssim  2^{(k_{1} + (2s-1)k_{2})/2} 2^{nk_{1}/2} \Vert f_{k_{1}} \Vert_{Z_{k_{1}}} \Vert f_{k_{2}} \Vert_{Z_{k_{2}}}.
\end{split}
\end{equation}
Where $ \tilde{f} \in \{ f , \bar{f} \}$.
\end{lemma}

\begin{proof}
Notice that by Lemma \ref{dec} we may assume $\hat{f}_{k_{2}}$ is supported in $ \{ ( \xi,\tau ) : \vert \xi - v \vert \leq 2^{k_{2} - 50} \} $ with $ \vert v \vert \approx 2^{k_{2}}$. Put $ \tilde{v} = \frac{ v }{\vert v \vert} $. Further, we may assume $k_{1} \leq k_{2} -10n$, since otherwise the result is trivial. Indeed, if $k_{1} \in [k_{2}-10n,k_{2} +10]$ then $2^{k_{1}} \approx 2^{k_{2}}$, so we estimate as follows  
\begin{equation}\label{cx4}
\begin{split}
& \Vert  H_{s} ( \tilde{f}_{k_{1}},\tilde{f}_{k_{2}}) \Vert_{L^{2}}
\\
& \lesssim \Vert (-\Delta)^s (  \tilde{f}_{k_{1}} \tilde{f}_{k_{2}} ) \Vert_{L^{2}(\R^{n+1})}
\\
& +  \Vert \tilde{f}_{k_{1}} (-\Delta)^s  \tilde{f}_{k_{2}} ) \Vert_{L^{2}(\R^{n+1})}
\\
& +   \Vert \tilde{f}_{k_{2}} (-\Delta)^s (  \tilde{f}_{k_{1}} ) \Vert_{L^{2}(\R^{n+1})}.
\end{split}
\end{equation}
We estimate each term on the right hand side of \eqref{cx4}. Starting with the first term, using the maximal and smoothing estimates and the fact that $ 2^{k_{1}} \approx 2^{k_{2}} $ we obtain
\begin{equation*}
\begin{split}
& \Vert (-\Delta)^s (  \tilde{f}_{k_{1}}  \tilde{f}_{k_{2}} ) \Vert_{L^{2}(\R^{n+1})}
\\
& \lesssim 2^{2sk_{2}}   \Vert f_{k_{1}} f_{k_{2}} \Vert_{L^{2}(\R^{n+1})}
\\
& \lesssim 2^{k_{1}} 2^{k_{2}(2s-1)}  \Vert f_{k_{2}} \Vert_{L^{\infty}_{\tilde{v}} L^{2}_{\tilde{v}^{\perp},t}} \Vert f_{k_{1}} \Vert_{L^{2}_{\tilde{v}} L^{\infty}_{\tilde{v}^{\perp},t}}
\\
& \lesssim  2^{(k_{1} + (2s-1)k_{2})/2} 2^{nk_{1}/2} \Vert f_{k_{1}} \Vert_{Z_{k_{1}}} \Vert f_{k_{2}} \Vert_{Z_{k_{2}}}.
\end{split}
\end{equation*}
The remaining terms on the right hand side of \eqref{cx4} are handled in exactly the same way.  
Assume in what follows that $k_{1} \leq k_{2} -10 n$. Then this implies that $ H_{s} ( \tilde{f}_{k_{1}} ,  \tilde{f}_{k_{2}})$ has Fourier transform supported in $ \{ \vert \xi \vert \approx 2^{k_{2}} \} \times \R$. 
Recall the commutator defined in \eqref{onesidedcommutator}
$$
\tilde{H}( \tilde{f}_{k_{2}},  \tilde{f}_{k_{1}}) : =   (-\Delta)^{s} ( \tilde{f}_{k_{1}}  \tilde{f}_{k_{2}}) - ( -\Delta)^{s}( \tilde{f}_{k_{2}}) .\tilde{f}_{k_{1}}
$$
then it is enough to prove the estimate for this term. Indeed, notice that 
$$
H_{s}( \tilde{f}_{k_{1}},  \tilde{f}_{k_{2}}) = \tilde{H}( \tilde{f}_{k_{2}}, \tilde{f}_{k_{1}}) -  \tilde{f}_{k_{2}}. (-\Delta)^s( \tilde{f}_{k_{1}})
$$
the second term satisfy the desired estimate 
\begin{equation*}
\begin{split}
& \Vert   \tilde{f}_{k_{2}}  (-\Delta)^s( \tilde{f}_{k_{1}}) \Vert_{L^{2}}
\\
& \lesssim  \Vert \tilde{f}_{k_{2}} \Vert_{L^{\infty}_{\tilde{v}} L^{2}_{\tilde{v}^{\perp},t}} 2^{2sk_{1}}  \Vert \tilde{f}_{k_{1}} \Vert_{L^{2}_{\tilde{v}} L^{\infty}_{\tilde{v}^{\perp},t}}
\\
& \lesssim   2^{(k_{1} + (2s-1)k_{2})/2} 2^{nk_{1}/2} \Vert f_{k_{1}} \Vert_{Z_{k_{1}}} \Vert f_{k_{2}} \Vert_{Z_{k_{2}}}.
\end{split}
\end{equation*}
So it remains to prove 
\begin{equation}\label{c1}
\begin{split}
 & \Vert \Delta_{\approx k_{2}} \tilde{H}( \tilde{f}_{k_{2}},  \tilde{f}_{k_{1}}) \Vert_{L^{2}(\R^{n+1})} 
 \\
 &\lesssim   2^{(k_{1} + (2s-1)k_{2})/2} 2^{nk_{1}/2} \Vert f_{k_{1}} \Vert_{Z_{k_{1}}} \Vert f_{k_{2}} \Vert_{Z_{k_{2}}}.
\end{split}
\end{equation}
To that end, we use Lemma \ref{Taylorstuff} with $p=2=q$ to arrive at 
\begin{equation*}
\begin{split}
 & \Vert \Delta_{\approx k_{2}} \tilde{H}( \tilde{f}_{k_{2}}, \tilde{f}_{k_{1}}) \Vert_{L^{2}(\R^{n+1})}
 \\
 & \lesssim \sum_{ \vert \alpha \vert =1}^{\infty} \frac{a_{m}}{\alpha !} 2^{k_{2}(2s-m)} \Vert \tilde{f}_{k_{2}} D^{\alpha}(  \tilde{f}_{k_{1}}) \Vert_{L^{2}(\R^{n+1})}
 \end{split}
\end{equation*}
and notice for any multi-index $ \alpha$ with $ \vert \alpha \vert =m$ we have 
\begin{equation*}
\begin{split}
& \Vert \tilde{f}_{k_{2}} D^{\alpha}(  \tilde{f}_{k_{1}}) \Vert_{L^{2}(\R^{n+1})}
\\
& \lesssim  \Vert \tilde{f}_{k_{2}} \Vert_{L^{\infty}_{\tilde{v}} L^{2}_{\tilde{v}^{\perp},t}} \Vert D^{\alpha} (  \tilde{f}_{k_{1}} ) \Vert_{L^{2}_{\tilde{v}} L^{\infty}_{\tilde{v}^{\perp},t}}
\\
& \lesssim  2^{k_{1} m } 2^{k_{1}(n-1)/2} 2^{-k_{2}(2s-1)/2} \Vert f_{k_{1}} \Vert_{Z_{k_{1}}} \Vert f_{k_{2}} \Vert_{Z_{k_{2}}} 
\end{split}
\end{equation*}

this implies
\begin{equation*}
\begin{split}
&  \Vert \Delta_{\approx k_{2} } \tilde{H}(  \tilde{f}_{k_{2}},  \tilde{f}_{k_{1}}) \Vert_{L^{2}(\R^{n+1})}
\\
& \lesssim (2^{k_{1}/2 + k_{2}(2s-1)/2} )^{-1} 2^{nk_{1}/2}  \Vert f_{k_{1}} \Vert_{Z_{k_{1}}} \Vert f_{k_{2}} \Vert_{Z_{k_{2}}} \sum_{\vert \alpha \vert =1}^{\infty}  \frac{a_{m}}{\alpha !} 2^{k_{2}(2s-m)} 2^{k_{1} m} 
\end{split}
\end{equation*}
but notice that 
$$
2^{k_{2}(2s-m)} 2^{k_{1} m} = 2^{k_{1} + k_{2}(2s-1)} 2^{(k_{1}-k_{2})(m-1) } \lesssim_{n} 2^{k_{1} + k_{2}(2s-1)} 2^{(-10n)m }
$$
this yields the result.
\end{proof}

The next two lemmata give an $L^{1}_{e}L^{2}_{e^{\perp},t}$ control of the commutator $H_{s}$ in terms of the resolution spaces. 
\begin{lemma}\label{comest2}
Let $k_{1},k_{2},k_{3}  \geq 0 $ with $ \max \{ {k_{2},k_{3}} \} \leq k_{1} + 10$ ,$ e \in \S^{n-1}$. Assume that $ f_{k_{i}} \in Z_{k_{i}} $ for $i \in \{ 1,2,3\}$. Assume further that $ f_{k_{1}}$ has Fourier transform supported in $ \{ \langle \xi , e \rangle \geq 2^{-c} \vert \xi \vert \} \times \R$ for some $ e \in \S^{n-1}$. Then the following estimate holds
$$
\Vert  H_{s} (\tilde{f}_{k_{1}}, (\tilde{f}_{k_{2}}\cdot \tilde{f}_{k_{3}} ) ) \Vert_{L^{1}_{e} L^{2}_{e^{\perp},t}} \lesssim 2^{k_{3} ( \frac{n +1}{2})} 2^{k_{2} ( \frac{n+1}{2})} 2^{k_{1}(2s-1)/2}  \Vert f_{k_{1}} \Vert_{Z_{k_{1}}} \Vert f_{k_{2}} \Vert_{Z_{k_{2}}} \Vert f_{k_{3}} \Vert_{Z_{k_{3}}},
$$
where $\tilde{f} \in \{ f , \bar{ f } \}$. 
\end{lemma}
\begin{proof}
Similar to the proof of Lemma \ref{comest} we break the situation into two cases. If $ \max \{ k_{2},k_{3} \} \in [ k_1 - 10n , k_{1} +10 ]$ then we argue as follows: Assume without loss of generality that $ k_{2} = \max \{ k_{2},k_{3} \}$, then 
\begin{equation}\label{qs1}
\begin{split}
& \Vert H_{s} (\tilde{f}_{k_{1}}, (\tilde{f}_{k_{2}}\cdot \tilde{f}_{k_{3}} ) ) \Vert_{L^{1}_{e} L^{2}_{e^{\perp},t}}
\\
& \lesssim \Vert  (-\Delta)^s ( \tilde{f}_{k_{1}} \tilde{f}_{k_{2}} \tilde{f}_{k_{3}} ) \Vert_{L^{1}_{e}L^{2}_{e^{\perp},t}}
\\
& + \Vert  \tilde{f}_{k_{1}} (-\Delta)^s ( \tilde{f}_{k_{2}} \tilde{f}_{k_{3}} )\Vert_{L^{1}_{e}L^{2}_{e^{\perp},t}}
\\
& + \Vert  (-\Delta)^s ( \tilde{f}_{k_{1}} ) \tilde{f}_{k_{2}} \tilde{f}_{k_{3}} \Vert_{L^{1}_{e}L^{2}_{e^{\perp},t}}
\end{split}
\end{equation}
and it is enough to estimate each term on the right hand side of \eqref{qs1}. For the first term one obtains, using H\"{o}lder, maximal and smoothing estimates, 
\begin{equation*}
\begin{split}
& \Vert  (-\Delta)^s ( \tilde{f}_{k_{1}} \tilde{f}_{k_{2}} \tilde{f}_{k_{3}} ) \Vert_{L^{1}_{e}L^{2}_{e^{\perp},t}}
\\
& \lesssim 2^{2sk_{1}} \Vert f_{k_{1}} \Vert_{L^{\infty}_{e}L^{2}_{e^{\perp},t}} \Vert f_{k_{2}} \Vert_{L^{2}_{e}L^{\infty}_{e^{\perp},t}} \Vert f_{k_{3}} \Vert_{L^{2}_{e} L^{\infty}_{e^{\perp},t}} 
\\
& \lesssim 2^{2sk_{1} } 2^{-k_{1}(2s-1)/2} 2^{k_{2}(n-1)/2} 2^{k_{3}(n-1)/2} \Vert f_{k_{1}} \Vert_{Z_{k_{1}}} \Vert f_{k_{2}} \Vert_{Z_{k_{2}}} \Vert f_{k_{3}} \Vert_{Z_{k_{3}}} 
\end{split}
\end{equation*}
Using that 
$$
2^{2sk_{1}} = 2^{k_{1}(2s-1) } 2^{k_{1}} \lesssim 2^{k_{1}(2s-1)} 2^{k_{2}}
$$
we obtain the result easily. The other terms on the right hand side of \eqref{qs1} are handled in the same way, in particular we use $ 2^{k_{2}} \approx 2^{k_{1}}$ and the local smoothing and maximal estimates. Therefore, assume in what follows that $ \max \{ {k_{2},k_{3}} \} = k_{2} \leq k_{1} - 10n$. The argument is similar to the proof of Lemma \ref{comest}. In particular, it suffices to prove the estimates for the operator
$$
\tilde{H}( f_{k_{1}}, f_{k_{2}} f_{k_{3}} ) : = (-\Delta)^s ( f_{k_{1}} f_{k_{2}} f_{k_{3}} ) - (-\Delta)^s (f_{k_{1}} ) f_{k_{2}} f_{k_{3}} 
$$
Then using Lemma \ref{Taylorstuff} we obtain
\begin{equation*}
\begin{split}
& 
\Vert \Delta_{\approx k_{1}} \tilde{H} ( \tilde{f}_{k_{1}}, \tilde{f}_{k_{2}} \tilde{f}_{k_{3}} ) \Vert_{L^{1}_{e}L^{2}_{e^{\perp},t}}
\\
& \lesssim \sum_{ \vert \alpha \vert =1}^{\infty} \frac{a_{m}}{  \alpha !} 2^{k_{1} ( 2s - m)}  \Vert \tilde{f}_{k_{1}} D^{\alpha} (  \tilde{ f }_{k_{2}} \tilde{f}_{k_{3}} ) \Vert_{L^{1}_{e}L^{2}_{e^{\perp},t}}
\end{split}
\end{equation*}
Using H\"{o}lder, local smoothing, and maximal estimate, we obtain 
\begin{equation}\label{e2}
\begin{split}
& \Vert \Delta_{\approx k_{1}} \tilde{H} ( \tilde{f}_{k_{1}}, \tilde{f}_{k_{2}} \tilde{f}_{k_{3}} ) \Vert_{L^{1}_{e}L^{2}_{e^{\perp},t}}
\\
& \lesssim \sum_{ \vert \alpha \vert =1}^{\infty} \frac{a_{m}}{ \vert \alpha \vert} 2^{k_{1} ( 2s - m)} 2^{k_{2} m } 2^{-k_{1}(2s-1)/2} 2^{k_{2} ( n-1)/2} 2^{k_{3} ( n-1)/2} \Vert f_{k_{1}} \Vert_{Z_{k_{1}}} \Vert f_{k_{2}} \Vert_{Z_{k_{2}}} \Vert f_{k_{3}} \Vert_{Z_{k_{3}}} 
\end{split}
\end{equation}
write 
$$
2^{k_{1} ( 2s-m) } 2^{k_{2} m} = 2^{k_{1}(2s-1) + k_{2}} 2^{(k_{2} -k_{1})(m-1)} \lesssim_{n}  2^{k_{1}(2s-1) + k_{2}} 2^{-10n m}
$$
plug this back into the right hand side of \eqref{e2} to obtain
\begin{equation*}
\begin{split}
&  \Vert \Delta_{\approx k_{1}} \tilde{H} ( \tilde{f}_{k_{1}}, \tilde{f}_{k_{2}} \tilde{f}_{k_{3}} ) \Vert_{L^{1}_{e}L^{2}_{e^{\perp},t}}
\\
& \lesssim  2^{k_{1}(2s-1)/2} 2^{k_{2} ( n+1)/2} 2^{k_{3} ( n+1)/2} \Vert f_{k_{1}} \Vert_{Z_{k_{1}}} \Vert f_{k_{2}} \Vert_{Z_{k_{2}}} \Vert f_{k_{3}} \Vert_{Z_{k_{3}}} \sum_{ | \alpha | =1}^{\infty} \frac{a_{m}}{|\alpha |} 2^{(-10n)m }
\end{split}
\end{equation*}
and the series converges as was shown in Lemma \ref{Taylorstuff}. 
\end{proof}

The next lemma is the same as Lemma \ref{comest2} but under the assumption that $ f_{k_{2}}$ is the function with highest frequency. 
\begin{lemma}\label{comest3}
Let $k_{1}, k_{2},k_{3}  \geq 0 $ with $ \max \{ {k_{1},k_{3}} \} \leq k_{2} + 10$, $ e \in \S^{n-1}$. Assume that $ f_{k_{i}} \in Z_{k_{i}} $ for $i \in \{ 1,2,3\}$. Assume further that $ f_{k_{2}}$ has Fourier transform supported in $ \{ \langle \xi , e \rangle \geq 2^{-c} \vert \xi \vert  \} \times \R$ for some $ e \in \S^{n-1}$. Then the following estimate holds
$$
\Vert  H_{s} (\tilde{f}_{k_{1}}, (\tilde{f}_{k_{2}}\cdot \tilde{f}_{k_{3}} ) ) \Vert_{L^{1}_{e} L^{2}_{e^{\perp},t}} \lesssim 2^{k_{1} ( \frac{n +1}{2})} 2^{k_{3} ( \frac{n+1}{2})} 2^{k_{2}(2s-1)/2}  \Vert f_{k_{1}} \Vert_{Z_{k_{1}}} \Vert f_{k_{2}} \Vert_{Z_{k_{2}}} \Vert f_{k_{3}} \Vert_{Z_{k_{3}}},
$$
where $ \tilde{f} \in \{ f , \bar{f} \}.$

\end{lemma}

\begin{proof}
We consider two cases. \\
\textbf{Case 1:} $ \max \{ k_{1} , k_{3} \} = k_{1}  $. Notice that we may assume $k_{1} \leq k_{2} - 10n$. Indeed, if
 $ k_{1} \in [ k_{2} - 10n, k_{2} + 10]$.  Then we argue as follows 
\begin{equation}\label{qs2}
\begin{split}
& \Vert  H_{s} (\tilde{f}_{k_{1}}, (\tilde{f}_{k_{2}}\cdot \tilde{f}_{k_{3}} ) ) \Vert_{L^{1}_{e} L^{2}_{e^{\perp},t}}
\\
& \lesssim \Vert  (-\Delta)^s ( \tilde{f}_{k_{1}} \tilde{f}_{k_{2}} \tilde{f}_{k_{3}} ) \Vert_{L^{1}_{e}L^{2}_{e^{\perp},t}}
\\
& + \Vert  \tilde{f}_{k_{1}} (-\Delta)^s ( \tilde{f}_{k_{2}} \tilde{f}_{k_{3}} )\Vert_{L^{1}_{e}L^{2}_{e^{\perp},t}}
\\
& + \Vert  (-\Delta)^s ( \tilde{f}_{k_{1}} ) \tilde{f}_{k_{2}} \tilde{f}_{k_{3}} \Vert_{L^{1}_{e}L^{2}_{e^{\perp},t}}
\end{split}
\end{equation}
We estimate the first term on the right hand side of \eqref{qs2} 
\begin{equation*}
\begin{split}
& \Vert  (-\Delta)^s ( \tilde{f}_{k_{1}} \tilde{f}_{k_{2}} \tilde{f}_{k_{3}} ) \Vert_{L^{1}_{e}L^{2}_{e^{\perp},t}}
\\
& \lesssim 2^{2sk_{2}} \Vert \tilde{f}_{k_{1}} \tilde{f}_{k_{2}} \tilde{f}_{k_{3}} \Vert_{L^{1}_{e}L^{2}_{e^{\perp},t}}
\\
& \lesssim 2^{k_{2}(2s-1)} 2^{k_{1}} \Vert f_{k_{2}} \Vert_{L^{\infty}_{e}L^{2}_{e^{\perp},t}} \Vert f_{k_{1}} \Vert_{L^{2}_{e}L^{\infty}_{e^{\perp},t}} \Vert f_{k_{3}} \Vert_{L^{2}_{e} L^{\infty}_{e^{\perp},t}}
\\
& \lesssim 2^{k_{2}(2s-1)/2} 2^{k_{1} ( n+1)/2} 2^{k_{3}(n+1)/2} \Vert f_{k_{1}} \Vert_{Z_{k_{1}}} \Vert f_{k_{2}} \Vert_{Z_{k_{2}}} \Vert f_{k_{3}} \Vert_{Z_{k_{3}}} 
\end{split}
\end{equation*}
which is the desired estimate. We handle the second term in the right hand side of \eqref{qs2} similarly 
\begin{equation*}
\begin{split}
&  \Vert  \tilde{f}_{k_{1}} (-\Delta)^s ( \tilde{f}_{k_{2}} \tilde{f}_{k_{3}} )\Vert_{L^{1}_{e}L^{2}_{e^{\perp},t}}
\\
& \lesssim \int_{\R} \Vert \tilde{f}_{k_{1}}( x_{1} e + \cdot, \cdot) \Vert_{L^{\infty}_{e^{\perp},t}} \Vert (-\Delta)^s ( \tilde{f}_{k_{2}} \tilde{f}_{k_{3}} ) ( x_{1} e + \cdot,\cdot) \Vert_{L^{2}_{e^{\perp},t}} dx_{1} 
\\
& \lesssim 2^{2s k_{2}} \Vert f_{k_{1}} \Vert_{L^{2}_{e} L^{\infty}_{e^{\perp},t}} \Vert f_{k_{3}} \Vert_{L^{2}_{e}L^{\infty}_{e^{\perp},t}} \Vert f_{k_{2}} \Vert_{L^{\infty}_{e} L^{2}_{e^{\perp},t}}
\\
& \lesssim 2^{2sk_{2}} 2^{-k_{2} (2s-1)/2} 2^{k_{1}(n-1)/2} 2^{k_{3}(n-1)/2}  \Vert f_{k_{1}} \Vert_{Z_{k_{1}}} \Vert f_{k_{2}} \Vert_{Z_{k_{2}}} \Vert f_{k_{3}} \Vert_{Z_{k_{3}}}
\\
& \lesssim 2^{k_{2}(2s-1)/2} 2^{k_{1}} 2^{k_{1}(n-1)/2} 2^{k_{3}(n-1)/2} \Vert f_{k_{1}} \Vert_{Z_{k_{1}}} \Vert f_{k_{2}} \Vert_{Z_{k_{2}}} \Vert f_{k_{3}} \Vert_{Z_{k_{3}}}
\end{split}
\end{equation*}
The estimate for the third term in the right hand side of \eqref{qs2} is identical. Therefore, we may assume that $ k_{1} \leq k_{2} - 10n$. Then in this case notice that it is enough to prove the estimate for the operator
$$
\tilde{H} ( \tilde{f}_{k_{1}}, \tilde{f}_{k_{2}} \cdot \tilde{f}_{k_{3}} ) : = (-\Delta)^s ( \tilde{f}_{k_{1}} \tilde{f}_{k_{2}} \tilde{f}_{k_{3}} ) - \tilde{f}_{k_{1}} (-\Delta)^s ( \tilde{f}_{k_{2}} \tilde{f}_{k_{3}} ) 
$$
Indeed, since 
$$
H_{s} ( \tilde{f}_{k_{1}} , \tilde{f}_{k_{2}} \tilde{f}_{k_{3}} ) = \tilde{H}_{s} ( \tilde{f}_{k_{1}}, \tilde{f}_{k_{2}} \cdot \tilde{f}_{k_{3}} ) - (-\Delta)^s(\tilde{f}_{k_{1}}) \tilde{f}_{k_{2}} \tilde{f}_{k_{3}} 
$$
and notice that we have 
\begin{equation*}
\begin{split}
& \Vert (-\Delta)^s (\tilde{f}_{k_{1}}) \tilde{f}_{k_{2}} \tilde{f}_{k_{3}} \Vert_{L^{1}_{e} L^{2}_{e^{\perp},t}}
\\
& \lesssim 2^{2sk_{1}} 2^{k_{1}(n-1)/2}  2^{-k_{2}(2s-1)/2} 2^{k_{3} (n-1)/2} \Vert f_{k_{1}}  \Vert_{Z_{k_{1}}}  \Vert f_{k_{2}} \Vert_{Z_{k_{2}}} \Vert f_{k_{3}} \Vert_{Z_{k_{3}}} 
\\
& \lesssim 2^{k_{2}(2s-1)/2} 2^{k_{1}(n+1)/2} 2^{k_{3}(n+1)/2} \Vert f_{k_{1}} \Vert_{Z_{k_{1}}} \Vert f_{k_{2}} \Vert_{Z_{k_{2}}} \Vert f_{k_{3}} \Vert_{Z_{k_{3}}}.
\end{split}
\end{equation*}
Therefore, it is enough to prove the estimate 
\begin{equation}\label{qs3}
\begin{split}
& \Vert \Delta_{\approx k_{2}} \tilde{H} ( \tilde{f}_{k_{1}}, \tilde{f}_{k_{2}} \cdot \tilde{f}_{k_{3}} ) \Vert_{L^{1}_{e} L^{2}_{e^{\perp},t}}
\\
& \lesssim 2^{k_{2} (2s-1)/2} 2^{k_{1} (n+1)/2} 2^{k_{3}(n+1)/2} \Vert f_{k_{1}} \Vert_{Z_{k_{1}}} \Vert f_{k_{2}}\Vert_{Z_{k_{2}}} \Vert f_{k_{3}} \Vert_{Z_{k_{3}}}
\end{split}
\end{equation}
To that end, we use Lemma \ref{Taylorstuff}, local smoothing and the maximal estimate to obtain 
\begin{equation*}
\begin{split}
&  \Vert \Delta_{\approx k_{2}} \tilde{H} ( \tilde{f}_{k_{1}}, \tilde{f}_{k_{2}} \cdot \tilde{f}_{k_{3}} ) \Vert_{L^{1}_{e} L^{2}_{e^{\perp},t}}
\\
& \lesssim \sum_{ \vert \alpha \vert =1}^{\infty} \frac{a_{m}}{\alpha !} 2^{k_{2}(2s-m)} \Vert \tilde{f}_{k_{2}} \tilde{f}_{k_{3}} D^{\alpha} (\tilde{f}_{k_{1}}) \Vert_{L^{1}_{e}L^{2}_{e^{\perp},t}}
\\
& \lesssim 2^{k_{1} (n-1)/2} 2^{k_{3}(n-1)/2} 2^{-k_{2} (2s-1)/2} \Vert f_{k_{1}} \Vert_{Z_{k_{1}}} \Vert f_{k_{2}} \Vert_{Z_{k_{2}}} \Vert f_{k_{3}} \Vert_{Z_{k_{3}}} \sum_{ \vert \alpha \vert =1} \frac{a_{m}}{\alpha !} 2^{k_{2} (2s-m)} 2^{k_{1} m }.
\end{split}
\end{equation*}
and the result follows by the fact that $ k_{1} \leq k_{2} - 10n$. \\
\textbf{Case 2:} $ \max \{ k_{1}, k_{3} \} = k_{3}$. Again, we may assume $k_{3} \leq k_{2} -10n$. Indeed, if $ k_{3} \in [ k_{2} -10n , k_{2} +10]$ then we estimate 
\begin{equation}\label{qs4}
\begin{split}
& \Vert H_{s} ( \tilde{f}_{k_{1}} , \tilde{f}_{k_{2}} \tilde{f}_{k_{3}} ) \Vert_{L^{1}_{e}L^{2}_{e^{\perp},t}}
\\
& \lesssim \Vert  (-\Delta)^s ( \tilde{f}_{k_{1}} \tilde{f}_{k_{2}} \tilde{f}_{k_{3}} ) \Vert_{L^{1}_{e}L^{2}_{e^{\perp},t}}
\\
& + \Vert  \tilde{f}_{k_{1}} (-\Delta)^s ( \tilde{f}_{k_{2}} \tilde{f}_{k_{3}} )\Vert_{L^{1}_{e}L^{2}_{e^{\perp},t}}
\\
& + \Vert  (-\Delta)^s ( \tilde{f}_{k_{1}} ) \tilde{f}_{k_{2}} \tilde{f}_{k_{3}} \Vert_{L^{1}_{e}L^{2}_{e^{\perp},t}}
\end{split}
\end{equation}
we estimate each term on the right hand side of \eqref{qs4} as was done in \eqref{qs2}. For the first term we obtain
\begin{equation*}
\begin{split}
& \Vert  (-\Delta)^s ( \tilde{f}_{k_{1}} \tilde{f}_{k_{2}} \tilde{f}_{k_{3}} ) \Vert_{L^{1}_{e}L^{2}_{e^{\perp},t}}
\\
& \lesssim 2^{2sk_{2}} \Vert \tilde{f}_{k_{1}} \tilde{f}_{k_{2}} \tilde{f}_{k_{3}}  \Vert_{L^{1}_{e}L^{2}_{e^{\perp},t}}
\\
& \lesssim 2^{k_{2}(2s-1) +k_{3}} \Vert \tilde{f}_{k_{1}} \tilde{f}_{k_{2}} \tilde{f}_{k_{3}}  \Vert_{L^{1}_{e}L^{2}_{e^{\perp},t}}
\\
& \lesssim 2^{k_{1} (n+1)/2} 2^{k_{3}(n+1)/2} 2^{k_{2}(2s-1)/2} \Vert f_{k_{1}} \Vert_{Z_{k_{1}}} \Vert f_{k_{2}} \Vert_{Z_{k_{2}}} \Vert f_{k_{3}} \Vert_{Z_{k_{3}}}.
\end{split}
\end{equation*}
The remaining terms on the right hand side of \eqref{qs4} are handled in the same way.
So assume in what follows that $ k_{3} \leq k_{2} -10 n$. Then  it is enough to prove the estimate
\begin{equation*}
\begin{split}
& \Vert \Delta_{\approx k_{2}} \tilde{H} ( \tilde{f}_{k_{1}}, \tilde{f}_{k_{2}} \cdot \tilde{f}_{k_{3}} ) \Vert_{L^{1}_{e} L^{2}_{e^{\perp},t}}
\\
& \lesssim 2^{k_{2}(2s-1)/2} 2^{k_{1}(n+1)/2} 2^{k_{3}(n+1)/2} \Vert f_{k_{1}} \Vert_{Z_{k_{1}}} \Vert f_{k_{2}}\Vert_{Z_{k_{2}}} \Vert f_{k_{3}} \Vert_{Z_{k_{3}}}.
\end{split}
\end{equation*}
To that end, we use Lemma \ref{Taylorstuff} to obtain 
\begin{equation*}
\begin{split}
& \Vert \Delta_{\approx k_{2}} \tilde{H} ( \tilde{f}_{k_{1}}, \tilde{f}_{k_{2}} \cdot \tilde{f}_{k_{3}} ) \Vert_{L^{1}_{e} L^{2}_{e^{\perp},t}}
\\
& \lesssim \sum_{\vert \alpha \vert =1}^{\infty} \frac{a_{m}}{\alpha !} 2^{k_{2} (2s-m)} \Vert \tilde{f}_{k_{2}} \tilde{f}_{k_{3}} D^{\alpha} (\tilde{f}_{k_{1}}) \Vert_{L^{1}_{e}L^{2}_{e^{\perp},t}}
\end{split}
\end{equation*}
and the result follows by H\"{o}lder, local smoothing, and maximal inequality.

\end{proof}
Next, we will need an $L^{1}_{e}L^{\infty}_{e^{\perp},t}$ estimate. 
\begin{lemma}\label{comest4}
Let $k_{1},k_{2} \geq 0$ with $ k_{2} \leq k_{1} +10$. Assume $f_{k_{1}} \in Z_{k_{1}}$ and $f_{k_{2} } \in Z_{k_{2}}$. Then the following estimate holds
$$
\Vert H_{s} ( \tilde{f}_{k_{1}},  \tilde{f}_{k_{2}} ) \Vert_{L^{1}_{e}L^{\infty}_{e^{\perp},t}} \lesssim    2^{(2s-1)k_{1}} 2^{(k_{1} + k_{2})(n+1)/2} \Vert f_{k_{1}} \Vert_{Z_{k_{1}}} \Vert f_{k_{2}} \Vert_{Z_{k_{2}}},
$$
where $ \tilde{f} \in \{ f , \bar{f} \}.$
\end{lemma}
\begin{proof}
By triangle inequality we have
\begin{equation}\label{cx1}
\begin{split}
& \Vert H_{s} (  \tilde{f}_{k_{1}}, \tilde{f}_{k_{2}} ) \Vert_{L^{1}_{e}L^{\infty}_{e^{\perp},t}} 
\\
& \lesssim  \Vert  (-\Delta)^s (  \tilde{f}_{k_{1}}  \tilde{f}_{k_{2}}) \Vert_{L^{1}_{e} L^{\infty}_{e^{\perp},t}} 
\\
& + \Vert  (-\Delta)^s (  \tilde{f}_{k_{1}})  
  \tilde{f}_{k_{2}} \Vert_{L^{1}_{e} L^{\infty}_{e^{\perp},t}}
\\
& + \Vert  \tilde{f}_{k_{1}} (-\Delta)^s (  \tilde{f}_{k_{2}}) \Vert_{L^{1}_{e} L^{\infty}_{e^{\perp},t}}.
\end{split}
\end{equation}
We estimate each term on the right hand side of \eqref{cx1}. Starting with the first term, using Lemma \ref{Multiplier} and the maximal estimate we obtain
\begin{equation*}
\begin{split}
& \Vert  (-\Delta)^s (  \tilde{f}_{k_{1}}  \tilde{f}_{k_{2}}) \Vert_{L^{1}_{e} L^{\infty}_{e^{\perp},t}}
\\
& \lesssim 2^{2sk_{1}}   \Vert \tilde{f}_{k_{1}} \tilde{f}_{k_{2}} \Vert_{L^{1}_{e}L^{\infty}_{e^{\perp},t}}
\\
& \lesssim    2^{(2s-1)k_{1}} 2^{(k_{1} + k_{2})(n+1)/2} \Vert f_{k_{1}} \Vert_{Z_{k_{1}}} \Vert f_{k_{2}} \Vert_{Z_{k_{2}}}
\end{split}
\end{equation*}
The remaining terms on the right hand side of \eqref{cx1} are handled in the same way, namely using only the maximal estimate and Lemma \ref{Multiplier}.

\end{proof}

The next result give estimates of trilinear expressions of the form $ f H_{s}(g,h)$. 
\begin{lemma}\label{CommTri}
Let $ k_{1} , k_{2} , k_{3} \geq 0$ with $ \max \{ k_{2}, k_{3} \}=k_{2} \leq k_{1}+10$. Let $ f_{k_{i}} \in Z_{k_{i}}$.
Then we have the following estimate;
    fix $ e \in \S^{n-1}$, if $ \hat{f}_{k_{1}}$ is supported in $ \{ (\xi,\tau) : \langle \xi, e \rangle \geq 2^{-c} \vert \xi \vert \}$ then we have 
    \begin{equation}\label{cm2}
    \begin{split}
    &  \Vert  \tilde{f}_{k_{1}} H_{s} (  \tilde{f}_{k_{2}},  \tilde{f}_{k_{3}} ) \Vert_{L^{1}_{e} L^{2}_{e^{\perp},t}}
    \\
    & \lesssim 2^{(k_{2}+k_{3})(n+1)/2} 2^{k_{1}(2s-1)/2} \Vert f_{k_{1}} \Vert_{Z_{k_{1}}} \Vert f_{k_{2}} \Vert_{Z_{k_{2}}} \Vert f_{k_{3}} \Vert_{Z_{k_{3}}}    
    \end{split}
    \end{equation}

\end{lemma}
\begin{proof}
\eqref{cm2} follows directly from Lemma \ref{comest4}. Indeed, 
\begin{equation*}
\begin{split}
&  \Vert  \tilde{f}_{k_{1}} H_{s} (  \tilde{f}_{k_{2}}, \tilde{f}_{k_{3}} ) \Vert_{L^{1}_{e} L^{2}_{e^{\perp},t}}
\\
& \lesssim  \Vert \tilde{f}_{k_{1}} \Vert_{L^{\infty}_{e} L^{2}_{e^{\perp},t}} \Vert
H_{s} (  \tilde{f}_{k_{2}},  \tilde{f}_{k_{3}} ) \Vert_{L^{1}_{e}L^{\infty}_{e^{\perp},t}}
\\
& \lesssim  2^{-k_{1}(2s-1)/2} \Vert f_{k_{1}} \Vert_{Z_{k_{1}}} \Vert
H_{s} (  \tilde{f}_{k_{2}},  \tilde{f}_{k_{3}} ) \Vert_{L^{1}_{e}L^{\infty}_{e^{\perp},t}}
\end{split}
\end{equation*}
and the result follows by Lemma \ref{comest4}.

\end{proof}

\section{The algebra properties of \texorpdfstring{$F^{\sigma}$}{F^sigma} spaces}
The main purpose of this section is to show that the $F^{\sigma}$ spaces are closed under multiplication when $ \sigma \geq \frac{n}{2}$. \\
We fix the following notation in this section:  whenever we write $ g_{k,j}$ we mean that $g_{k,j}$ is an $L^{2}$ function with frequency localized at $2^{k}$ and modulation localized at $ 2^{j}$. In other words, $ \text{supp}(\hat{g}_{k,j}) \subset  D_{k,j}$. Also, given a function $f_{k} \in Z_{k}$ we write $ f_{k, j}$ to mean $ Q_{j} (f_{k})$. Similarly, if we write $ f_{k, \geq j}$ and $ f_{k,\leq j}$. 
\\\\
\subsection{Bilinear estimate}
The following lemma will be useful in what follows
\begin{lemma}\label{cl}
Let $ j_{1},j_{2} ,j, k , k_{1} ,k_{2} \geq 0$. Take $ g_{k_{1,j_{1}}}, g_{k_{2},j_{2}} \in L^{2}(\R^{n+1}) $ with the property that $ \text{supp} (\hat{g}_{k_{i},j_{i}} ) \subset D_{k_{i},j_{i}}$ for $i=1,2$. Then the following holds 
$$
\Vert \ind_{D_{k,j}} . \hat{\tilde{g}}_{k_{1},j_{1}} * \hat{\tilde{g}}_{k_{2},j_{2}} \Vert_{L^{2}(\R^{n+1})} \lesssim 2^{ n \min \{ k, k_{1},k_{2} \}/2 } 2^{ \min \{ j, j_{1} , j_{2} \}/2 } \Vert g_{k_{1},j_{1}} \Vert_{L^{2}(\R^{n+1})} \Vert g_{k_{2},j_{2}} \Vert_{L^{2}(\R^{n+1})},
$$ 
where $ \tilde{g} \in \{ g, \bar{g} \}$.
\end{lemma}
\begin{proof}
This is an immediate consequence of Young and H\"{o}lder inequalities. Indeed, for simplicity assume $ \min\{j,j_{1},j_{2} \} = j$ and $ \min \{ k ,k_{1},k_{2} \} = k$ then
\begin{equation*}
\begin{split}
& \Vert \ind_{D_{k,j}}( \hat{g}_{k_{1},j_{1}} * \hat{g}_{k_{2},j_{2}}) \Vert_{L^{2}(\R^{n+1})}
\\
& \leq \Vert \ind_{D_{k,j}} \Vert_{L^{2}} \Vert \hat{g}_{k_{1},j_{1}} * \hat{g}_{k_{2},j_{2}} \Vert_{L^{\infty}}
\\
& \leq 2^{j/2} 2^{nk/2} \Vert g_{k_{1},j_{1}} \Vert_{L^{2}} \Vert g_{k_{2},j_{2}} \Vert_{L^{2}}.
\end{split}
\end{equation*}
For the other cases we use the same trick. Indeed, assume $ \min \{ j_{1},j_{2} , j \} = j_{1}$ and $ \min \{ k_{1},k_{2},k \} = k_{2}$. Then use the notation $ *_{s} $ to mean convolution in the variable $ \xi$ and $ *_{t}$ to mean convolution in the variable $ \tau$. Then we estimate using Young inequality
\begin{equation*}
\begin{split}
& \Vert \ind_{D_{k,j}} ( \hat{\tilde{g}}_{k_{1},j_{1}} * \hat{\tilde{g}}_{k_{2},j_{2}} ) \Vert_{L^{2}} 
\\
& \leq  \Vert ( \hat{\tilde{g}}_{k_{1},j_{1}} * \hat{\tilde{g}}_{k_{2},j_{2}} ) \Vert_{L^{2}} 
\\
& \leq \Vert (\Vert \hat{\tilde{g}}_{k_{1},j_{1}} \Vert_{L^{1}_{\tau}} *_{s} \Vert \hat{\tilde{g}}_{k_{2},j_{2}} \Vert_{L^{2}_{\tau}}) \Vert_{L^{2}_{\xi}}
\\
& \lesssim 2^{j_{1}/2} \Vert  \Vert \hat{\tilde{g}}_{k_{1},j_{1}}\Vert_{L^{2}_{\tau}} *_{s} \Vert \hat{\tilde{g}}_{k_{2},j_{2}} \Vert_{L^{2}_{\tau}} \Vert_{L^{2}_{\xi}}
\\
& \lesssim 2^{j_{1}/2} \Vert \hat{\tilde{g}}_{k_{1},j_{1}} \Vert_{L^{2}_{\xi,\tau}} \Vert \hat{\tilde{g}}_{k_{2},j_{2}} \Vert_{L^{1}_{\xi}L^{2}_{\tau}}
\\
& \lesssim 2^{j_{1}/2} \Vert g_{j_{1},k_{1}} \Vert_{L^{2}_{\xi,\tau}} 2^{k_{2}n/2} \Vert g_{k_{2},j_{2}} \Vert_{L^{2}_{\xi,\tau}}.
\end{split}
\end{equation*}

\end{proof}

The following lemma is the fractional version of \cite[Lemma 5.1]{cmp}
\begin{lemma}\label{es1}
Let $ j_{1},j_{2}, k_{1}, k_{2} \geq 0 $ with $ k_{1} \leq k_{2} + C $. Assume that $ f_{k_{1}} \in Z_{k_{1}}$, and $ f_{k_{2}} \in Z_{k_{2}}$. Then we have the following inequality 
\begin{equation*}
\begin{split}
 & \Vert \tilde{f}_{k_{1}, \geq j_{1}}   \tilde{f}_{k_{2}, \geq j_{2}} \Vert_{L^{2}} 
 \\
 & \lesssim \left( 2^{j_{2}/2} + 2^{k_{1}/2 +(k_{2})(2s-1)/2} \right)^{-1} 2^{nk_{1}/2} \Vert f_{k_{1}} \Vert_{Z_{k_{1}}} \Vert f_{k_{2}} \Vert_{Z_{k_{2}}} 
\end{split}
\end{equation*}
where $ \tilde{f} \in \{ f , \bar{f} \}.$
\end{lemma}
\begin{proof} 
We break the proof to two cases. Case 1: if $ k_{1} + (2s-1)k_{2} \geq j_{2}$.
 First, by Lemma \ref{dec} we may assume $ \hat{f}_{{k_{2}}} $ is supported in $ \{ (\xi, \tau) : \vert \xi - v \vert \ll 2^{k_{2}} \} $ for some $v$ with $ | v | \approx 2^{k_{2}}$. Take $ \tilde{v} = \frac{v}{\vert v \vert}$. Then using local smoothing we have 
 \begin{equation}\label{Q1}
  \Vert ( \tilde{f}_{k_{2}, \geq j_{2}}) \Vert_{L^{\infty}_{\tilde{v}} L^{2}_{\tilde{v}^{\perp},t}} \lesssim 2^{-k_{2}(2s-1)/2} \Vert f_{k_{2}} \Vert_{Z_{k_{2}}}.
  \end{equation}
  Similarly, we have by the maximal estimate 
  \begin{equation}\label{Q2}
    \Vert  \tilde{f}_{k_{1}, \geq j_{1}} \Vert_{L^{2}_{\tilde{v}} L^{\infty}_{\tilde{v}^{\perp},t}} \lesssim 2^{(n-1)k_{1}/2} \Vert f_{k_{1}} \Vert_{Z_{k_{1}}}
  \end{equation}
Combining these two estimates we obtain 
\begin{equation*}
\begin{split}
& \Vert  (\tilde{f}_{k_{1}, \geq j_{1}}) (\tilde{f}_{k_{2}, \geq j_{2}}) \Vert_{L^{2}} 
 \\
 & \lesssim \left(  2^{k_{1}/2 +(k_{2})(2s-1)/2} \right)^{-1} 2^{nk_{1}/2} \Vert f_{k_{1}} \Vert_{Z_{k_{1}}} \Vert f_{k_{2}} \Vert_{Z_{k_{2}}} 
 \\
 & \lesssim
 \left( 2^{j_{2}/2}+ 2^{k_{1}/2 +(k_{2})(2s-1)/2} \right)^{-1} 2^{nk_{1}/2} \Vert f_{k_{1}} \Vert_{Z_{k_{1}}} \Vert f_{k_{2}} \Vert_{Z_{k_{2}}}.
\end{split}
\end{equation*}
Next, if $ k_{1} + k_{2}(2s-1) \leq j_{2} $. Then by Lemma \ref{Emb} we have
\begin{equation}\label{Q3}
\Vert   f_{k_{2}, \geq j_{2}} \Vert_{L^{2}} \lesssim 2^{-j_{2}/2} \Vert f_{k_{2}} \Vert_{Z_{k_{2}}}.
\end{equation}
We further have by Lemma \ref{Energy} the following 
\begin{equation} \label{Q4}
\Vert (f_{k_{1}, \geq j_{1}}) \Vert_{L^{\infty}} \lesssim 2^{nk_{1}/2} \Vert f_{k_{1}} \Vert_{Z_{k_{1}}}.
\end{equation}
Combining \eqref{Q3} and \eqref{Q4} we get the result. 
 
\end{proof}

Next, we will prove the following Bilinear estimate. Compare with \cite[Lemma 5.2]{cmp}.

\begin{theorem}\label{Bilin1}
Let $k_{1},k_{2},k  \geq 0 $ with $ k_{1} \leq k_{2}+10 $. And $f_{k_{1}} \in Z_{k_{1}} $ and $ f_{k_{2}} \in Z_{k_{2}}$ then we have the following 
$$
2^{nk/2} \Vert \Delta_{k} (  \tilde{f}_{k_{1}} f_{k_{2}}) \Vert_{Z_{k}} \lesssim 2^{- \frac{n}{2}\vert k - k_{2} \vert} ( 2^{nk_{1}/2} \Vert f_{k_{1}} \Vert_{Z_{k_{1}}} ) ( 2^{nk_{2}/2} \Vert f_{k_{2}} \Vert_{Z_{k_{2}}} ),
$$
where $ \tilde{f}_{k_{1}} \in \{ f_{k_{1}}, \bar{f}_{k_{1}} \}$.
\end{theorem}
\begin{proof}
First, notice that if $k >  k_{2} +20$ then we must have that $ \Delta_{k} ( \tilde{f}_{k_{1}} f_{k_{2}}) = 0 $. Therefore, assume that $ k \leq k_{2} + 20$.
Define $ \tilde{K} :=  k_{1} + (2s-1) k_{2} + C' $ for sufficiently large $C'$. Then write 
$$
\Delta_{k} ( \tilde{f}_{k_{1}} f_{k_{2}}) = Q_{\geq \tilde{K}} ( \Delta_{k}(\tilde{f}_{k_{1}} f_{k_{2}})) + Q_{\leq \tilde{K}-1} ( \Delta_{k}(\tilde{f}_{k_{1}} f_{k_{2}})).
$$
We estimate the second term. Using Lemma \ref{es1} ( with $j_{1} = 0 = j_{2} $ ) we get
\begin{equation}
\begin{split}
& 2^{nk/2} \Vert  Q_{\leq \tilde{K}-1} ( \Delta_{k}( \tilde{f}_{k_{1}} f_{k_{2}})) \Vert_{Z_{k}}
\\
& \leq 2^{nk/2} \Vert  Q_{\leq \tilde{K}-1} ( \Delta_{k}(\tilde{f}_{k_{1}} f_{k_{2}})) \Vert_{X_{k}}
\\
& \lesssim 2^{nk/2} 2^{\tilde{K}/2} \Vert \Delta_{k}(\tilde{f}_{k_{1}} f_{k_{2}}) \Vert_{L^{2}} 
\\
& \lesssim 2^{nk/2} 2^{\tilde{K}/2} \left( 2^{-\tilde{K}/2} 2^{nk_{1}/2} \Vert f_{k_{1}} \Vert_{Z_{k_{1}}} \Vert f_{k_{2}} \Vert_{Z_{k_{2}}} \right)
\\
& = 2^{n(k-k_{2})/2}  2^{nk_{1}/2} 2^{nk_{2}/2} \Vert f_{k_{1}} \Vert_{Z_{k_{1}}} \Vert f_{k_{2}} \Vert_{Z_{k_{2}}}
\end{split}
\end{equation}
which is the desired result. It remains to prove the following
\begin{equation}\label{Q6}
2^{nk/2} \Vert Q_{\geq \tilde{K}} ( \Delta_{k}(\tilde{f}_{k_{1}} f_{k_{2}})) \Vert_{Z_{k}} \lesssim 2^{- \frac{n}{2}\vert k - k_{2} \vert } 2^{k_{2}n/2} 2^{k_{1}n/2} \Vert f_{k_{1}} \Vert_{Z_{k_{1}}} \Vert f_{k_{2}}\Vert_{Z_{k_{2}}}. 
\end{equation} 
 First consider the case if $k_{2} <100$.  Then by virtue of $k_{1} \leq k_{2} +10 \lesssim 1 $, we have $\Vert f_{k_{i}} \Vert_{X_{k_{i}}} \lesssim \Vert f_{k_{i}} \Vert_{Z_{k_{i}}} $ for each $i=1,2$. Hence, it is enough to prove the following 
\begin{equation}\label{Q5}
\begin{split}
& 2^{j/2} 2^{nk/2} \Vert Q_{\geq \tilde{K}} \ind_{D_{k,j}} (\tilde{g}_{k_{1},j_{1}} * g_{k_{2}j_{2}}) \Vert_{L^{2}} 
\\
& \lesssim 2^{-\frac{n}{2}\vert k - k_{2} \vert } 2^{nk_{1}/2} 2^{j_{1}/2} \Vert g_{k_{1}j_{1}} \Vert_{L^{2}} 2^{j_{2}/2} 2^{k_{2}n/2} \Vert g_{k_{2}j_{2}} \Vert_{L^{2}}
\end{split}
\end{equation}
for any $g_{k_{i},j_{i}}$ supported in $ D_{k_{i},j_{i}}$. For future use, we prove this without the restriction that $k_{2} < 100$. This follows from Lemma \ref{cl}. Indeed, notice that $j < \max \{ j_{1},j_{2} \} +C$ for some constant $C$. Because otherwise we will have $0$ on the left hand side of \eqref{Q5}, see Lemma \ref{prod1}. Using this, and Lemma \ref{cl} we obtain
\begin{equation*}
\begin{split}
&  2^{j/2} 2^{nk/2} \Vert \ind_{D_{k,j}} (\tilde{g}_{k_{1},j_{1}} * g_{k_{2}j_{2}}) \Vert_{L^{2}} 
\\
& \lesssim 2^{\max \{ j_{1},j_{2} \}/2} 2^{nk/2} 2^{k_{1}n/2} 2^{\min \{j_{1},j_{2}\}/2} \Vert g_{k_{1},j_{1}} \Vert_{L^{2} } \Vert g_{k_{2},j_{2}} \Vert_{L^{2}} 
\\
& \lesssim 2^{n(k-k_{2})/2} 2^{nk_{2}/2}   2^{n k_{1}/2} 2^{j_{1}/2} 2^{j_{2}/2} \Vert g_{k_{1},j_{1}} \Vert_{L^{2} } \Vert g_{k_{2},j_{2}} \Vert_{L^{2}}
\end{split}
\end{equation*}
this gives \eqref{Q5} since $ k \leq k_{2} +20$. And this proves the result under the assumption that $ k_{2} < 100$.  \\
Assume in what follows that $k_{2} \geq 100$. 
To prove \eqref{Q6}, we write 
$$
f_{k_{1}} = \sum_{j \geq  0 } g_{k_{1},j} + \sum_{e \in \mathscr{E} ,k' \in [1,...,k_{1}+1]} f^{e}_{k_{1},k'}, 
$$
where each $ g_{k_{1},j} $ has Fourier transform supported in $ D_{k_{1},j}$ and each $ f^{e}_{k_{1},k'} \in Y^{e}_{k_{1},k'}$. We decompose $f_{k_{2}} $ similarly.  We consider several cases
\\
\textbf{Case 1:} $ f_{k_{1}} = g_{k_{1},j_{1}}$ and $f_{k_{2}} = g_{k_{2},j_{2}}$, this case was already proven in \eqref{Q5}. \\
\textbf{Case 2:} $f_{k_{1}} = f^{e}_{k_{1},k'} \in Y^{e}_{k_{1},k'}$ and $ f_{k_{2}} = g_{k_{2},j_{2}}$. In view of case $1$ we may assume that $ \mathcal{F}f^{e}_{k_{1},k'}$ is supported in $ D^{e,k'}_{k_{1}, \leq 2sk_{1}-80}$ and $ k' \in T_{k_{1}}$.  We claim that $j_{2} \geq \tilde{K}-C$ or the left hand side of \eqref{Q6} is $0$. Indeed, this follows immediately from Lemma \ref{prod1} because $ 2sk_{1} -80 < \tilde{K} - C' - 70$. Therefore, by Lemma \ref{prod1} we may assume that $ j_{2} \geq \tilde{K} -C $. We estimate
\begin{equation*}
\begin{split}
&2^{nk/2} \Vert Q_{\geq \tilde{K}} ( \Delta_{k}(\tilde{f}_{k_{1}}  g_{k_{2},j_{2}}) ) \Vert_{Z_{k}}
\\
& \leq2^{nk/2} \Vert Q_{\geq \tilde{K}} ( \Delta_{k}(\tilde{f}_{k_{1}}  g_{k_{2},j_{2}}) \Vert_{X_{k}}
\\
&  \lesssim2^{nk/2} \sum_{j=  j_{2} -C}^{j_{2}+C} 2^{j/2} \Vert Q_{j}( Q_{\geq \tilde{K}}( \Delta_{k}(\tilde{f}_{k_{1}} g_{k_{2},j_{2}} )) \Vert_{L^{2}}
\\
& \lesssim 2^{nk/2} 2^{j_{2}/2} \Vert \Delta_{k}(\tilde{f}_{k_{1}} g_{k_{2},j_{2}} )\Vert_{L^{2}} 
\\
& \lesssim 2^{nk/2} 2^{j_{2}/2} \Vert g_{k_{2},j_{2}} \Vert_{L^{2}} \Vert \tilde{f}_{k_{1}} \Vert_{L^{\infty}}
\\
& \lesssim 2^{nk/2} 2^{j_{2}/2} \Vert g_{k_{2},j_{2}} \Vert_{L^{2}}  2^{nk_{1}/2} \Vert f_{k_{1}} \Vert_{Z_{k_{1}}}.
\end{split}
\end{equation*}
\textbf{Case 3:} $ f_{k_{2}} \in Y^{e}_{k_{2},k'}$ and $ f_{k_{1}} = g_{k_{1},j_{1}}$. With $k_{1} + 10 \geq k_{2}$ ( This implies that $k_{1} $ is essentially the same as $k_{2}$). In this case we repeat the argument from case 2.  
\begin{equation*}
\begin{split}
& 2^{nk/2} \Vert Q_{\geq \tilde{K}} \left( \Delta_{k} ( f_{k_{2}} \tilde{g}_{k_{1},j_{1}} \right)\Vert_{Z_{k}}
\\
& \lesssim 2^{nk/2} \Vert Q_{\geq \tilde{K}} \left( \Delta_{k} ( f_{k_{2}} \tilde{g}_{k_{1},j_{1}} \right) \Vert_{X_{k}}
\\
& \lesssim 2^{nk/2} 2^{j_{1}/2} \Vert f_{k_{2}} \tilde{g}_{k_{1},j_{1}} \Vert_{L^{2}(\R^{n+1})} 
\\
& \lesssim 2^{nk/2} 2^{j_{1}/2} \Vert g_{k_{1},j_{1}} \Vert_{L^{2}(\R^{n+1})} 
\Vert f_{k_{2}} \Vert_{L^\infty(\R^{n+1})} 
\\
& \lesssim 2^{nk/2} \Vert g_{k_{1},j_{1}} \Vert_{X_{k_{1}}} 2^{nk_{2}/2} \Vert f_{k_{2}} \Vert_{Z_{k_{2}}} 
\end{split}
\end{equation*}
and the result follows because $ 2^{k_{1}} \approx 2^{k_{2}}$ in this case.\\
\textbf{Case 4:} $f_{k_{1}} \in Y^{e'}_{k_{1}}$ and $f_{k_{2}} \in Y^{e}_{k_{2},k'}$ and $ k_{1} +10 \geq k'$. Then in light of case 2 we may assume that $ \hat{f}_{k_{2}} $ is supported in $D_{k_{2}, \leq 2sk'}$.  $ \hat{f}_{k_{1}}$ is supported in $ D_{k_{1},\leq 2sk_{1}+10}$ by definition. So, the modulation of both functions is bounded by $2^{2sk_{1}+10}$. On the other hand, we have $Q_{\geq \tilde{K}}$. Therefore, by Lemma \ref{prod1} we get $ Q_{ \geq \tilde{K}} ( \tilde{f}_{k_{1}} f_{k_{2}}) = 0 $.
\\
These cases imply the result when $k_{1} \geq k_{2} -10$. So assume in the remaining cases that $k_{1} \leq k_{2} -10$. Then we must have $ k \geq k_{2} -c'$ for some constant $c'$. Indeed, if $ 2^{k_{1}} \ll 2^{k_{2}}$ and $ 2^{k} \ll 2^{k_{2}}$ then we will have $ \Delta_{k} (\tilde{f}_{k_{1}} f_{k_{2}}) = 0 $ and this gives the result trivially. So assume in what follows that $ | k - k_{2} | \lesssim 1 $ and $ 2^{k_{1}} \ll 2^{k_{2}}$.\\
\textbf{Case 5 :} $ f_{k_{2},k'} \in Y^{e}_{k_{2},k'}$ with $k'\in T_{k_{2}}$ and $ f_{k_{1}} = g_{k_{1},j_{1}}$ with $k' \leq k_{1} +10$.
In view of case $1$ we may assume that the Fourier transform of $ f_{k_{2}}$ is supported in $D^{e,k'}_{k_{2}, \leq 2k_{2}(s+q-1)-80}$. In particular, by Lemma \ref{prod1} we may assume that $ j_{1} \geq \tilde{K} - C$ for some constant $C$. Then 
\begin{equation*}
\begin{split}
& 2^{nk/2} \Vert \mathcal{F}^{-1} \left( \varphi_{k} (\xi) \varphi_{\geq \tilde{K}} (\tau+\vert \xi \vert^{2s}) ( \hat{f}_{k_{2},k'} * \hat{\tilde{g}}_{k_{1},j_{1}} \right) \Vert_{Z_{k}}
\\
& \leq 2^{nk/2} \Vert \mathcal{F}^{-1} \left( \hat{f}_{k_{2}, k'} * \hat{\tilde{g}}_{k_{1},j_{1}} \right) \Vert_{X_{k}} 
\\
& \lesssim 2^{nk/2} 2^{j_{1}/2} \Vert \hat{f}_{k_{2}, k'} * \hat{\tilde{g}}_{k_{1},j_{1}} \Vert_{L^{2}}.
\end{split}
\end{equation*}
So it suffices to prove 
\begin{equation}\label{M1}
\Vert \hat{f}_{k_{2}} * \hat{\tilde{g}}_{k_{1},j_{1}} \Vert_{L^{2}} \lesssim 2^{nk_{1}/2} \Vert g_{k_{1},j_{1}} \Vert_{L^{2}} \Vert f_{k_{2}} \Vert_{Y^{e}_{k_{2},k'}}.
\end{equation}
whenever $\hat{f}_{k_{2}}$ is supported in $D^{e,k'}_{k_{2},\leq 2k_{2}(s+q -1)-80}$. 
For future use, we prove \eqref{M1} without the restriction that $k' \leq k_{1} +10$. Since $ \hat{f}_{k_{2}}$ is supported $D^{e,k'}_{k_{2},\leq 2k_{2}(s+q-1) -80}$, then using Lemma \ref{repf} we may assume 
$$
\hat{f}_{k_{2}} = \frac{\varphi_{\leq k' -80} (\xi_{1} - N)}{ \xi_{1} -N + \i/2^{k_{2}(2s-2+q)}} h(\xi',\tau)
$$
with 
$$ 
\Vert h \Vert_{L^{2}_{\xi',\tau}} \lesssim 2^{2(1-s)(k-k')} \gamma^{-1}_{k,k'} 2^{-k'(2s-1)/2} \Vert f_{k_{2}} \Vert_{Y^{e}_{k_{2},k'}}, 
$$ and $h$ is supported in $ \{ \vert \xi' \vert \lesssim 2^{k_{2}}, N(\xi',\tau) \approx 2^{k'}, N^{2 } + \vert \xi' \vert^{2} \approx 2^{2k} \}$. In particular,
$$
\Vert h \Vert_{L^{2}_{\xi',\tau}} \lesssim 2^{-k'(2s-1)/2} \Vert f_{k_{2}} \Vert_{Y^{e}_{k,k'}}.
$$
(Notice that in view of case 1, we are ignoring the $X_{k}$ error term coming from Lemma \ref{repf}).\\
Then \eqref{M1} is reduced to showing 
\begin{equation}\label{M2}
\Vert \hat{f}_{k_{2}} * \hat{\tilde{g}}_{k_{1},j_{1}} \Vert_{L^{2}} \lesssim 2^{(2s-1)k'/2} 2^{nk_{1}/2} \Vert h \Vert_{L^{2}_{\xi',\tau}} \Vert g_{k_{1},j_{1}} \Vert_{L^{2}} 
\end{equation}
using duality we obtain that the left hand side of \eqref{M2} is equal to 
\begin{equation}\label{M3}
\begin{split}
& \Vert \hat{f}_{k_{2}} * \hat{\tilde{g}}_{k_{1},j_{1}} \Vert_{L^{2}}
\\
& \Vert \int_{\R \times e^{\perp} \times \R} \hat{\tilde{g}}_{k_{1},j_{1}}(\eta_{1}e+ \eta',\beta) \hat{f}_{k_{2}}((\xi_{1}-\eta_{1})e+ \xi'- \eta', \tau - \beta) d\eta_{1} d \eta' d \beta \Vert_{L^{2}} 
\\
&
= \sup_{a \in L^{2}, \Vert a \Vert =1} \vert \int_{\R \times e^{\perp} \times \R} \int_{\R \times e^{\perp} \times \R}  \hat{\tilde{g}}_{k_{1},j_{1}}(\eta,\beta) \hat{f}_{k_{2}}(\xi,\tau) a( \xi_{1} + \eta_{1}, \xi'+ \eta', \tau+ \beta) d \xi d \tau d \eta d \beta \vert
\end{split}
\end{equation}
fix $ a $ with norm 1. Then notice 
\begin{equation}\label{M4}
\begin{split}
& \int_{\R \times e^{\perp} \times \R} \hat{f}_{k_{2}} (\xi,\tau) a( \xi_{1} + \eta_{1}, \xi'+ \eta', \tau+ \beta) d \xi_{1} d \xi' d \tau
\\
& = \int_{\R \times e^{\perp} } h(\xi',\tau) \int_{\R} \frac{ \varphi_{\leq k'- 80}(\xi_{1} - N)}{ \xi_{1} - N + \frac{i}{2^{k(2s-2+q)}}} a(\xi_{1}+ \eta_{1}, \xi'+\eta', \tau+ \beta) d\xi_{1} d \tau d \xi' 
\end{split}
\end{equation}
use the change of variable $ \tau = - ( N^{2} + \vert \xi' \vert^{2})^{s} $ then the Jacobian would be $c_{s} N ( N^{2} + \vert \xi' \vert^{2})^{s-1} $ and define $ \tilde{h}(\xi',N) = h(\xi', -(N^{2}+ \vert \xi' \vert^{2})^{s}) N ( N^{2}+\vert \xi' \vert^{2})^{s-1}$. Then using this change of variable, \eqref{M4} becomes
\begin{equation}\label{M5}
\int_{\R \times e^{\perp} } \tilde{h}( \xi',N) \int_{\R} \frac{ \varphi_{\leq k'-80}(\xi_{1} -N)}{\xi_{1} -N + \frac{i}{2^{k(2s-2+q)}}} a(\xi_{1} + \eta_{1}, \xi'+ \eta', -( N^{2}+ \vert \xi' \vert^{2})^{s} + \beta) d \xi_{1} d N d \xi'
\end{equation}
plug \eqref{M5} into \eqref{M3} to get 
\begin{equation}\label{M6}
\begin{split}
& \vert \int_{\R \times e^{\perp} \times \R}  \hat{\tilde{g}}_{k_{1},j_{1}} (\eta_{1}e + \eta',\beta) \int_{e^{\perp}} \int_{\R} \tilde{h}(\xi',N) \\
& \times \left( \int_{\R} \frac{ \eta_{\leq k'-80}(\xi_{1} -N)}{\xi_{1} -N + \frac{i}{2^{(k(2s-2+q)}}} a(\xi_{1}+ \eta_{1}, \xi'+\eta', -(N^{2}+\vert \xi'\vert^{2})^{s}+\beta) d \xi_{1} \right) d N d \xi' d \beta d \eta \vert.
\end{split}
\end{equation}
Apply H\"{o}lder in the variable $ \beta$ to obtain that \eqref{M6} is bounded by 
\begin{equation}\label{M7}
\begin{split}
& \vert \int_{ \R \times e^{\perp}} \Vert \hat{\tilde{g}}_{k_{1},j_{1}}( \eta_{1} e+ \eta',\cdot ) \Vert_{L^{2}_{\beta}} \int_{e^{\perp} \times \R} \tilde{h}(\xi',N) \Vert H( a( \cdot+ \eta_{1}, \xi'+ \eta', \beta )(N) \Vert_{L^{2}_{\beta}} d N d \xi' d \eta_{1} d \eta' \vert,
\end{split}
\end{equation}
where $H$ is the Hilbert transform. 
Next, apply H\"{o}lder in the variable $N$ and using boundedness of the Hilbert transform on $L^{2} $ we bound \eqref{M7} by 
\begin{equation}\label{M8}
\begin{split}
& \vert \int_{ \R \times e^{\perp}} \Vert \hat{\tilde{g}}_{k_{1},j_{1}}( \eta_{1} e+ \eta',\cdot ) \Vert_{L^{2}_{\beta}} \int_{e^{\perp} } \Vert \tilde{h}(\xi',\cdot)\Vert_{L^{2}_{N}} \Vert a( \cdot, \xi'+ \eta', \cdot ) \Vert_{L^{2}_{\xi_{1}, \beta}} d \xi' d \eta_{1} d \eta' \vert
\end{split}
\end{equation}
Lastly, we apply H\"{o}lder in the variable $ \xi'$ and obtain that \eqref{M8} is bounded by 
\begin{equation}
\begin{split}
& \Vert \tilde{h} \Vert_{L^{2}_{\xi',N}} \int_{\R \times e^{\perp}} \Vert \hat{\tilde{g}}_{k_{1},j_{1}} ( \eta_{1}e+\eta',\cdot) \Vert_{L^{2}_{\beta}} 
\\
& \lesssim 2^{k_{1}n/2} \Vert \tilde{h} \Vert_{L^{2}_{\xi',N}} \Vert g_{k_{1},j_{1}} \Vert_{L^{2}}
\end{split}
\end{equation}
and notice that $ \Vert \tilde{h} \Vert_{L^{2}} \lesssim 2^{(2s-1)k'/2} \Vert h \Vert_{L^{2}}$. Thus, we have proven \eqref{M2}. Which proves the result in this case.
\\
\textbf{Case 6: } $f_{k_{2}} \in Y^{e}_{k_{2},k'}$, $k'\in T_{k_{2}}$ and $ f_{k_{1}} = g_{k_{1},j_{1}}$ , $k' \geq k_{1} +10$ and $j_{1} \geq \tilde{K}$. In light of case $1$ we may assume that $f_{k_{2}}$ is supported in $D^{e,k'}_{k_{2}, \leq 2k_{2}(s+q-1)-80} $. Then we want to prove 
\begin{equation}\label{Q9}
2^{nk/2} \Vert \mathcal{F}^{-1} \left( \varphi_{k}(\xi) \varphi_{\geq \tilde{K}} (\tau+\vert \xi \vert^{2s}) ( \hat{f}_{k_{2}} * \hat{\tilde{g}}_{k_{1},j_{1}}) \right) \Vert_{Z_{k}} \lesssim 2^{j_{1}/2} 2^{nk_{1}/2} \Vert g_{k_{1},j_{1}} \Vert_{L^{2}} 2^{k_{2}n/2} \Vert f_{k_{2}} \Vert_{Z_{k_{2}}}.
\end{equation}
To that end, write $ f_{k_{2}} = f_{k_{2},\leq j_{1}-c} + f_{k_{2}, \geq j_{1}+c} + f_{k_{2}, [j_{1}-c,j_{1}+c]} $ then 
\begin{equation}\label{sxcv}
\begin{split}
& 2^{nk/2} \Vert \mathcal{F}^{-1} \left( \varphi_{k}(\xi) \varphi_{\geq \tilde{K}} (\tau+\vert \xi \vert^{2s}) ( \hat{f}_{k_{2}} * \hat{\tilde{g}}_{k_{1},j_{1}}) \right) \Vert_{Z_{k}}
\\
& \leq 2^{nk/2} \Vert  \mathcal{F}^{-1} \left( \varphi_{k}(\xi) \varphi_{\geq \tilde{K}} (\tau+\vert \xi \vert^{2s}) ( \hat{f}_{k_{2},\leq j_{1}-c} * \hat{\tilde{g}}_{k_{1},j_{1}} ) \right) \Vert_{Z_{k}} 
\\
& +  2^{nk/2} \Vert \mathcal{F}^{-1} \left( \varphi_{k}(\xi) \varphi_{\geq \tilde{K}} (\tau+\vert \xi \vert^{2s}) ( \hat{f}_{k_{2}, \geq j_{1}+c} * \hat{\tilde{g}}_{k_{1},j_{1}})  \right) \Vert_{Z_{k}} 
\\
&+ 
2^{nk/2} \Vert \mathcal{F}^{-1} \left( \varphi_{k}(\xi) \varphi_{\geq \tilde{K}} (\tau+\vert \xi \vert^{2s}) ( \hat{f}_{k_{2},[j_{1}-c,j_{1}+c]} * \hat{\tilde{g}}_{k_{1},j_{1}} ) \right) \Vert_{Z_{k}}.
\end{split}
\end{equation}
For the last term we use the embedding of $X_{k} $ into $Y^{e}_{k_{2},k'}$ and case 1. It remains to estimate the first two terms. To that end, we start with the first term on the right hand side of \eqref{sxcv}
\begin{equation*}
\begin{split}
& 2^{nk/2} \Vert \mathcal{F}^{-1} \left( \varphi_{k}(\xi) \varphi_{\geq \tilde{K}} (\tau+\vert \xi \vert^{2s}) ( \hat{f}_{k_{2},\leq j_{1}-c} * \hat{\tilde{g}}_{k_{1},j_{1}} ) \right) \Vert_{Z_{k}}
\\
& \leq 2^{nk/2} \Vert \mathcal{F}^{-1} \left( \varphi_{\geq \tilde{K}}(\tau + \vert \xi \vert^{2s}) (\hat{f}_{k_{2}, \leq j_{1}-c} * \hat{\tilde{g}}_{k_{1},j_{1}} ) \right) \Vert_{X_{k}} 
\\
& \lesssim 2^{nk/2} 2^{j_{1}/2} \Vert \hat{f}_{k_{2}, \leq j_{1} -c} * \hat{\tilde{g}}_{k_{1},j_{1}} \Vert_{L^{2}}
\end{split}
\end{equation*}
and we want to prove 
\begin{equation}\label{Q10}
\Vert \hat{f}_{k_{2}} * \hat{\tilde{g}}_{k_{1},j_{1}} \Vert_{L^{2}} \lesssim 2^{nk_{1}/2} \Vert g_{k_{1},j_{1}} \Vert_{L^{2}} \Vert f_{k_{2}} \Vert_{Y^{e}_{k_{2},k'}}
\end{equation}
But this follows from \eqref{M1}. To conclude case 6, we need to prove the estimate for the second term in \eqref{Q9}. In other words, we want to prove
\begin{equation}\label{Q18}
2^{nk/2} \Vert \Delta_{k} Q_{\geq \tilde{K}} ( {f}_{k_{2}, \geq j_{1} + c} \tilde{g}_{k_{1},j_{1}}) \Vert_{Z_{k}} \lesssim 2^{nk_{1}/2}  2^{nk_{2}/2} \Vert f_{k_{2}} \Vert_{Z_{k_{2}}} \Vert g_{k_{1},j_{1}} \Vert_{X_{k_{1}}}.
\end{equation}
To that end, write
\begin{equation}\label{Q19}
\begin{split}
& ( \i \partial_{t} -  (-\Delta)^s + \i ) (f_{k_{2},\geq j_{1} +c} \tilde{g}_{k_{1},j_{1}} )
\\
 & = ( \i \partial_{t} -  (-\Delta)^s + \i )(f_{k_{2},\geq j_{1}+c}). \tilde{g}_{k_{1},j_{1}} 
 \\
 & + f_{k_{2}, \geq j_{1}+c} ( \i \partial_{t} -  (-\Delta)^s ) \tilde{g}_{k_{1},j_{1}} 
 \\
 & -H_{s}(f_{k_{2}, \geq j_{1}+c } ,  \tilde{g}_{k_{1},j_{1}}) 
 \end{split}
\end{equation}
where the last term is the commutator of the fractional laplacian. Then we bound the left hand side of \eqref{Q18} by 
\begin{equation}\label{Q20}
\begin{split}
& 2^{nk/2} \Vert \Delta_{k} Q_{\geq \tilde{K}}  \left(\frac{   ( \i \partial_{t} -  (-\Delta)^s + \i ) (f_{k_{2},\geq j_{1} +c} ) . \tilde{g}_{k_{1},j_{1}}}{ \i \partial_{t} -   (-\Delta)^s +\i}  \right) \Vert_{Z_{k}}
\\
& + 2^{nk/2} \Vert \Delta_{k} Q_{\geq \tilde{K}} \left( \frac{  f_{k_{2}, \geq j_{1}+c} ( \i \partial_{t} -  (-\Delta)^s ) \tilde{g}_{k_{1},j_{1}}}{\i \partial_{t} - (-\Delta)^s +\i} \right) \Vert_{Z_{k}}
\\
& + 2^{nk/2} \Vert \Delta_{k} Q_{\geq \tilde{K}} \left( \frac{ H_{s}(f_{k_{2}, \geq j_{1}+c } ,  \tilde{g}_{k_{1},j_{1}})}{\i \partial_{t} -  (-\Delta)^s + \i} \right) \Vert_{Z_{k}}
\end{split}
\end{equation}
if we bound each term  in \eqref{Q20} with the right hand side of \eqref{Q18} then this concludes case 6. Starting with the first term, notice that since $2^{k_{1}} \ll 2^{k'}$ then the support of $ \hat{f}_{k_{2} \geq j_{1}+c} * \hat{\tilde{g}}_{k_{1},j_{1}} $ is contained in $\{ (\xi,\tau): \xi. e \approx 2^{k'} \}$. Therefore, we bound the first term as follows 
\begin{equation*}
\begin{split}
& 2^{nk/2} \Vert \Delta_{k} Q_{\geq \tilde{K}}  \left(\frac{   ( \i \partial_{t} - (-\Delta)^s + \i ) (f_{k_{2},\geq j_{1} +c} ) . \tilde{g}_{k_{1},j_{1}}}{ \i \partial_{t} -   (-\Delta)^s +\i}  \right) \Vert_{Z_{k}}
\\
& \sum_{k'' = k'-c}^{\min\{k'+c,k+1\}} \Vert \Delta_{k} Q_{\geq \tilde{K}}  \left(\frac{   ( \i \partial_{t}  - (-\Delta)^s + \i ) (f_{k_{2},\geq j_{1} +c} ) . \tilde{g}_{k_{1},j_{1}}}{ \i \partial_{t} -  (-\Delta)^s +\i}  \right) \Vert_{Y^{e}_{k,k''}}
\\
& \lesssim 2^{nk/2} \gamma_{k,k'} 2^{-k'(2s-1)/2} \Vert ( \i \partial_{t} - (-\Delta)^s + \i ) (f_{k_{2},\geq j_{1} +c} ) . \tilde{g}_{k_{1},j_{1}} \Vert_{L^{1}_{e}L^{2}_{e^{\perp},t}}
\\
& \leq 2^{nk/2} \Vert g_{k_{1},j_{1}} \Vert_{L^{\infty}} \gamma_{k,k'} 2^{-k'(2s-1)/2} \Vert ( \i \partial_{t} -(-\Delta)^s + \i ) (f_{k_{2},\geq j_{1} +c} ) \Vert_{L^{1}_{e}L^{2}_{e^{\perp},t}}
\\
& \lesssim  2^{nk/2} \Vert f_{k_{2}} \Vert_{Z_{k_{2}}} \Vert g_{k_{1},j_{1}} \Vert_{L^{\infty}}
\\
& \lesssim 2^{nk_{2}/2} \Vert f_{k_{2}} \Vert_{Z_{k_{2}}} 2^{nk_{1}/2} \Vert g_{k_{1},j_{1}} \Vert_{X_{k_{1}}}
\end{split}
\end{equation*}
which is the right hand side of \eqref{Q18}. Next, we do the same with second term in \eqref{Q20}. Using Lemma \ref{prod2} we obtain
\begin{equation*}
\begin{split}
& 2^{nk/2} \Vert \Delta_{k} Q_{\geq \tilde{K}} \left( \frac{  f_{k_{2}, \geq j_{1}+c} ( \i \partial_{t} -  (-\Delta)^s ) \tilde{g}_{k_{1},j_{1}}}{\i \partial_{t} - (-\Delta)^s +\i} \right) \Vert_{Z_{k}}
\\
& \lesssim 2^{nk/2} \Vert \Delta_{k} Q_{\geq \tilde{K}} \left( \frac{  f_{k_{2}, \geq j_{1}+c} ( \i \partial_{t} -   (-\Delta)^s ) \tilde{g}_{k_{1},j_{1}}}{\i \partial_{t} - (-\Delta)^s +\i} \right) \Vert_{X_{k}}
\\
& \lesssim 2^{nk/2} \Vert f_{k_{2}, \geq j_{1}+c}   ( \i \partial_{t} - (-\Delta)^s ) \tilde{g}_{k_{1},j_{1}} \Vert_{L^{2}} \sum_{j=j_{1}-c}^{\infty} 2^{-j/2} 
\\
& \lesssim 2^{nk/2} 2^{-j_{1}/2} \Vert f_{k_{2}, \geq j_{1}+c}   ( \i \partial_{t} -   (-\Delta)^s ) \tilde{g}_{k_{1},j_{1}} \Vert_{L^{2}}
\end{split}
\end{equation*}
using Lemma \ref{es1} we bound this by 
\begin{equation*}
2^{nk/2} \Vert f_{k_{2}} \Vert_{Z_{k_{2}}} 2^{nk_{1}/2} \Vert g_{k_{1},j_{1}} \Vert_{X_{k_{1}}}
\end{equation*}
which gives the result since $ 2^{k } \approx 2^{k_{2}}$.  It remains to prove the estimate for the third term in \eqref{Q20}. To that end, we estimate as above
\begin{equation*}
\begin{split}
& 2^{nk/2} \Vert \Delta_{k}Q_{\geq \tilde{K}} \left( \frac{ H_{s}(f_{k_{2}, \geq j_{1}+c } ,  \tilde{g}_{k_{1},j_{1}})}{\i \partial_{t} - (-\Delta)^s + \i} \right) \Vert_{Z_{k}}
\\
& \lesssim 2^{nk/2} \Vert \Delta_{k} Q_{\geq \tilde{K}} \left( \frac{ H_{s}(f_{k_{2}, \geq j_{1}+c } ,  \tilde{g}_{k_{1},j_{1}})}{\i \partial_{t} - (-\Delta)^s + \i} \right) \Vert_{X_{k}}
\\
& \lesssim 2^{nk/2} \Vert H_{s} ( f_{k_{2}, \geq j_{1}+c} , \tilde{g}_{k_{1,j_{1}}} )\Vert_{L^{2}} \sum_{j= \tilde{K} }^{\infty} 2^{-j/2} 
\\
& \lesssim 2^{nk/2} 2^{-\tilde{K}/2} \Vert H_{s}(f_{k_{2},\geq j_{1}+c} , \tilde{g}_{k_{1},j_{1}}) \Vert_{L^{2}}
\\
& \lesssim 2^{nk/2} 2^{-\tilde{K} /2} 2^{\tilde{K}/2} 2^{nk_{1}/2} \Vert f_{k_{2}} \Vert_{Z_{k_{2}}} \Vert g \Vert_{Z_{k_{1}}}
\end{split}
\end{equation*}
where we used Lemma \ref{comest} in the last line. Using that $ 2^{k_{2}} \approx 2^{k}$ we obtain
\begin{equation*}
\begin{split}
& 2^{nk/2} \Vert \Delta_{k} Q_{\geq \tilde{K}} \left( \frac{ H_{s}(f_{k_{2}, \geq j_{1}+c } ,  \tilde{g}_{k_{1},j_{1}})}{\i \partial_{t} - (-\Delta)^s + \i} \right) \Vert_{Z_{k}}
\\
& \lesssim 2^{nk_{2}/2} \Vert f_{k_{2}} \Vert_{Z_{k_{2}}} . 2^{nk_{1}/2} \Vert g_{k_{1},j_{1}} \Vert_{Z_{k_{1}}}
\end{split}
\end{equation*}
this concludes case 6. 
\\
\textbf{Case 7:} the last case to consider is $ f_{k_{2}} \in Y^{e}_{k_{2},k'}$ and $ f_{k_{1} \leq \tilde{K}} \in Z_{k_{1}}$ with $ 2^{k_{1}} \leq 2^{k'-10}$, and $k' \in T_{k_{2}}$, and $f_{k_{1},\leq \tilde{K}}$ has modulation controlled by $2^{\tilde{K}}$.  Using  \eqref{Q20} we estimate 
\begin{equation}\label{Q21}
\begin{split}
& 2^{nk/2} \Vert \Delta_{k} Q_{\geq \tilde{K}} ( \tilde{f}_{k_{1}, \leq \tilde{K}} f_{k_{2}}) \Vert_{Z_{k}}
\\
& \lesssim  2^{nk/2} \Vert \Delta_{k} Q_{\geq \tilde{K}}  \left(\frac{   ( \i \partial_{t} - (-\Delta)^s + \i ) (f_{k_{2}} ) . \tilde{f}_{k_{1} , \leq \tilde{K}}}{ \i \partial_{t} -  (-\Delta)^s +\i}  \right) \Vert_{Z_{k}}
\\
& + 2^{nk/2} \Vert \Delta_{k} Q_{\geq \tilde{K}} \left( \frac{  f_{k_{2}} ( \i \partial_{t} -   (-\Delta)^s ) \tilde{f}_{k_{1}, \leq \tilde{K}}}{\i \partial_{t} - (-\Delta)^s +\i} \right) \Vert_{Z_{k}}
\\
& + 2^{nk/2} \Vert \Delta_{k} Q_{\geq \tilde{K}} \left( \frac{ H_{s}(f_{k_{2} } ,  f_{k_{1}, \leq \tilde{K}})}{\i \partial_{t} - (-\Delta)^s + \i} \right) \Vert_{Z_{k}}
\end{split}
\end{equation}
and it is enough to estimate each term in the right hand side of \eqref{Q21} by 
$$
2^{nk_{1}/2} \Vert f_{k_{1}, \leq \tilde{K}} \Vert_{Z_{k_{1}}} . 2^{nk_{2}/2} \Vert f_{k_{2}} \Vert_{Z_{k_{2}}}.
$$
To that end, we start with the first term.  Notice that the support of $ \mathcal{F} (f_{k_{2}} \tilde{f}_{k_{1}, \leq \tilde{K}})$ is contained in $ \{ (\xi,\tau) : \langle \xi , e\rangle \approx 2^{k'} \} $. Hence, 

\begin{equation*}
\begin{split}
& 2^{nk/2} \Vert \Delta_{k} Q_{\geq \tilde{K}}  \left(\frac{   ( \i \partial_{t} - (-\Delta)^s + \i ) (f_{k_{2}} ) . \tilde{f}_{k_{1}, \leq \tilde{K}}}{ \i \partial_{t} -  (-\Delta)^s +\i}  \right) \Vert_{Z_{k}}
\\
& \leq 2^{nk/2} \sum_{k''=k'-c}^{\min \{ k'+c,k+1 \} } \Vert \Delta_{k} Q_{\geq \tilde{K}}  \left(\frac{   ( \i \partial_{t} - (-\Delta)^s + \i ) (f_{k_{2}} ) . \tilde{f}_{k_{1}, \leq \tilde{K}}}{ \i \partial_{t} -   (-\Delta)^s +\i}  \right) \Vert_{Y^{e}_{k,k''}}
\\
& \lesssim 2^{nk/2} \gamma_{k_{2},k'} 2^{-k'(2s-1)/2} \Vert ( \i \partial_{t} -  (-\Delta)^s + \i ) (f_{k_{2}} ) . \tilde{f}_{k_{1}, \leq \tilde{K}} \Vert_{L^{1}_{e}L^{2}_{e^{\perp},t}}
\\
& \leq 2^{nk/2} \Vert f_{k_{1}, \leq \tilde{K}} \Vert_{L^{\infty}} \gamma_{k_{2},k'} 2^{-k'(2s-1)/2} \Vert ( \i \partial_{t} - (-\Delta)^s + \i ) (f_{k_{2}} ) \Vert_{L^{1}_{e}L^{2}_{e^{\perp},t}}
\\
& \lesssim  2^{nk/2} \Vert f_{k_{2}} \Vert_{Z_{k_{2}}} \Vert f_{k_{1}, \leq \tilde{K}} \Vert_{L^{\infty}}
\\
& \lesssim 2^{nk_{2}/2} \Vert f_{k_{2}} \Vert_{Z_{k_{2}}} 2^{nk_{1}/2} \Vert f_{k_{1}, \leq \tilde{K}} \Vert_{Z_{k_{1}}}.
\end{split}
\end{equation*}
Next, we estimate the second term on the right hand side of \eqref{Q21}
\begin{equation*}
\begin{split}
& 2^{nk/2} \Vert \Delta_{k} Q_{\geq \tilde{K}} \left( \frac{  f_{k_{2}} ( \i \partial_{t} -  (-\Delta)^s ) \tilde{f}_{k_{1}, \leq \tilde{K}}}{\i \partial_{t} -(-\Delta)^s +\i} \right) \Vert_{Z_{k}}
\\
& \lesssim 2^{nk/2} \Vert \Delta_{k} Q_{\geq \tilde{K}} \left( \frac{  f_{k_{2}} ( \i \partial_{t} -  (-\Delta)^s)  \tilde{f}_{k_{1}, \leq \tilde{K}}}{\i \partial_{t} - (-\Delta)^s +\i} \right) \Vert_{X_{k}}
\\
& \lesssim 2^{nk/2} \Vert f_{k_{2}}   ( \i \partial_{t} - (-\Delta)^{s}   ) \tilde{f}_{k_{1}, \leq \tilde{K}} \Vert_{L^{2}} \sum_{j=\tilde{K}}^{\infty} 2^{-j/2} 
\\
& \lesssim 2^{nk/2} 2^{-\tilde{K}/2} \Vert f_{k_{2}}   ( \i \partial_{t} - (-\Delta)^s ) \tilde{f}_{k_{1}, \leq \tilde{K}} \Vert_{L^{2}}
\end{split}
\end{equation*}
using Lemma \ref{es1} and the fact that the modulation of $f_{k_{1}, \leq \tilde{K}}$ is controlled by $ 2^{\tilde{K}}$ we obtain the bound
\begin{equation*}
\begin{split}
& 2^{nk/2} 2^{-\tilde{K}/2} \Vert f_{k_{2}}   ( \i \partial_{t} -   (-\Delta)^s)  \tilde{f}_{k_{1}, \leq \tilde{K}} \Vert_{L^{2}}
\\
& \lesssim 2^{-\tilde{K}/2}  2^{nk_{2}/2} \Vert f_{k_{2}} \Vert_{Z_{k_{2}}} 2^{n k_{1}/2} \Vert f_{k_{1}, \leq \tilde{K}} \Vert_{Z_{k_{1}}}  2^{\tilde{K}} 2^{-\tilde{K}/2}
\end{split}
\end{equation*}
this gives the desired bound. Lastly, we estimate the third term in the right hand side of \eqref{Q21}. 
\begin{equation*}
\begin{split}
& 2^{nk/2} \Vert \Delta_{k} Q_{\geq \tilde{K}} \left( \frac{ H_{s}(f_{k_{2} } ,  \tilde{f}_{k_{1}, \leq \tilde{K}})}{\i \partial_{t} -(-\Delta)^s + \i} \right) \Vert_{Z_{k}}
\\
& \lesssim 2^{nk/2} \Vert \Delta_{k} Q_{\geq \tilde{K}} \left( \frac{ H_{s}(f_{k_{2} } ,  \tilde{f}_{k_{1}, \leq \tilde{K}})}{\i \partial_{t} - (-\Delta)^s + \i} \right) \Vert_{X_{k}}
\\
& \lesssim 2^{nk/2} \Vert H_{s} ( f_{k_{2}} , \tilde{f}_{k_{1}, \leq \tilde{K}} )\Vert_{L^{2}} \sum_{j= \tilde{K} }^{\infty} 2^{-j/2} 
\\
& \lesssim 2^{nk/2} 2^{-\tilde{K}/2} \Vert H_{s}(f_{k_{2}} , \tilde{f}_{k_{1}, \leq \tilde{K}}) \Vert_{L^{2}}
\\
& \lesssim 2^{nk/2} 2^{-\tilde{K} /2} 2^{nk_{1}/2} 2^{\tilde{K}/2} \Vert f_{k_{1}, \leq \tilde{K}} \Vert_{Z_{k_{1}}} \Vert f_{k_{2}} \Vert_{Z_{k_{2}}}
\end{split}
\end{equation*}
where we used Lemma \ref{comest} in the last line. 
This finishes the proof of case 7.

\end{proof}

An immediate corollary of Theorem \ref{Bilin1} is the following
\begin{corollary}\label{Algebra}
Let $ \sigma \geq \frac{n}{2}$. 
Then there exists a constant $C_{A}:=C_{A}(n,s,\sigma)$ so that for any $ f, g \in F^{\sigma}$, we have the following
$$
\Vert f g \Vert_{F^{\sigma}} \leq C_{A} \left(  \Vert f \Vert_{F^{n/2}} \Vert g \Vert_{F^{\sigma}} + \Vert f \Vert_{F^{\sigma}} \Vert g \Vert_{F^{n/2}}\right).
$$
\end{corollary}
Lastly, define 
$$
\bar{F}^{\sigma} : = \{ f \in L^{2} ( \R^{n+1}) : \bar{f} \in F^{\sigma} \}.
$$
Then define $ \tilde{F}^{\sigma} := F^{\sigma} + \bar{F}^{\sigma}$ with the infimum norm. Then we have 
\begin{corollary}\label{Algebratilde}
Let $ \sigma \geq \frac{n}{2}$. 
Then there exists a constant $C_{\tilde{A}}:=C_{\tilde{A}}(n,s,\sigma)$ so that for any $ f, g \in F^{\sigma}$, we have the following
$$
\Vert f g \Vert_{\tilde{F}^{\sigma}} \leq C_{\tilde{A}} \left(  \Vert f \Vert_{\tilde{F}^{n/2}} \Vert g \Vert_{\tilde{F}^{\sigma}} + \Vert f \Vert_{\tilde{F}^{\sigma}} \Vert g \Vert_{\tilde{F}^{n/2}}\right).
$$
\end{corollary}

\section{Nonlinear estimates}
In this section we prove the nonlinear estimates. Let 
$$
\mathcal{N}(f) : =  H_{ {s}} ( f, \frac{1}{1+|f|^{2}}) +  f H_{{s}} ( f, \frac{ \bar{f}}{1+|f|^{2}} ) .
    $$
Throughout this section, the implicit constants in the inequalities may depend on the differentiation parameter $ \sigma$.
The main result of this section is the following    
\begin{theorem}\label{NonlinearEst} Let $ \sigma \geq \frac{n+1}{2}$. Then there exists $ \epsilon : = \epsilon(n,s,\sigma) > 0$ so that for any $ f \in F^{\sigma} $ with $ \Vert f \Vert_{F^{\frac{n+1}{2}}} \leq \epsilon$, one has the following inequality
    $$
    \Vert \mathcal{N} ( f) \Vert_{N^{\sigma}} \lesssim \Vert f \Vert^{2}_{F^{(n+1)/2}} \Vert f \Vert_{F^{\sigma}}.
    $$
\end{theorem}
the proof of theorem \ref{NonlinearEst} will be a sequence of propositions. Namely, we estimate each of the nonlinearities  separately.

We start with the first term
\begin{proposition}\label{firstn}
Let $ \sigma \geq \frac{n+1}{2}$. Then
there exists $ \epsilon: = \epsilon (n,s, \sigma)> 0$ so that for any
 $ f \in F^{\sigma}$ with $ \Vert f \Vert_{F^{\frac{n+1}{2}}} \leq \epsilon$, one has the following 
\begin{equation}\label{nonlinear1}
\Vert H_{s} ( f , \frac{1}{1+|f|^{2}} ) \Vert_{N^{\sigma}} \lesssim \Vert f \Vert^{2}_{F^{(n+1)/2}} \Vert f \Vert_{F^{\sigma}}
\end{equation}
    
\end{proposition}

\begin{proof}
By expanding using Taylor, and the bi linearity of the commutator $H_{s}$, we can write 
$$
H_{s} ( f , \frac{1}{1+ | f |^{2}}) = \sum_{j=1}^{\infty} (-1)^{j} H_{s}( f , (f \bar{f} )^{j} ).
$$
We will prove for any $j \geq 1$ we have the following
\begin{equation}\label{nonll}
\begin{split}
& \Vert H_{s} ( f , (f)^{j} ( \bar{f})^{j} ) \Vert_{N^{\sigma}}
\\
& \lesssim \Vert f \Vert_{F^{\sigma}} \Vert (f)^{j} \Vert^{2}_{F^{n/2}} +  2 \Vert f \Vert_{F^{n/2}} \Vert f^{j} \Vert_{F^{n/2}} \Vert f^{j} \Vert_{F^{\sigma}}
\end{split}
\end{equation}
To simplify the notation, write $ f_{k,j}$ to mean $ \Delta_{k}((f)^{j})$. 
Then we have the following
\begin{equation}\label{qz4}
\begin{split}
& \Vert H_{s} ( f , (f \bar{f})^{j} ) \Vert_{N^{\sigma}}
\\
& = \sum_{k=0}^{\infty} 2^{\sigma k } \Vert (\i \partial_{t} - (-\Delta)^{s}   + \i)^{-1}  \Delta_{k} H_{s} ( f ,(f \bar{f})^{j}) \Vert_{Z_{k}}
\\
& \lesssim  \sum_{k_{1},k_{2},k_{3} =0 }^{\infty} \sum_{k=0}^{\infty} 2^{\sigma k}  \Vert(\i \partial_{t} - (-\Delta)^{s}   + \i)^{-1} \Delta_{k} H_{s}( f_{k_{1}} , (f_{k_{2},j} \bar{f}_{k_{3},j}) ) \Vert_{Z_{k}}. 
\end{split}
\end{equation}
Fix integers $k_{1},k_{2},k_{3} \geq 0$ then we consider several cases.
\\
\textbf{Case 1:} if $ k_{1} = \max\{ k_{1},k_{2} ,k_{3} \}$. Then in this case it suffices to prove
\begin{equation}\label{vv4}
\begin{split}
& \sum_{k_{1},k_{2},k_{3}\geq 0 } \sum_{k=0}^{k_{1} +10 } 2^{\sigma k}  \Vert (\i \partial_{t} - (-\Delta)^{s}   + \i)^{-1} \Delta_{k} H_{s}( f_{k_{1}} , (f_{k_{2},j} \bar{f}_{k_{3},j}) ) \Vert_{Z_{k}}
\\
& \lesssim \sum_{k_{1}, k_{2}, k_{3} \geq 0} 2^{\sigma k_{1} } 2^{ k_{2} (n+1)/2 } 2^{k_{3}(n+1)/2 } \Vert f_{k_{1}} \Vert_{Z_{k_{1}}} \Vert f_{k_{2},j} \Vert_{Z_{k_{2}}} \Vert f_{k_{3},j} \Vert_{Z_{k_{3}}}.
\end{split}
\end{equation}
To prove \eqref{vv4} we consider several sub-cases, depending on whether $ 2^{ \max \{ k_{2},k_{3} \}} \approx 2^{k_{1}}$ or not.\\
\textbf{Case 1a:} if $ k_{1} = \max \{ k_{1} , k_{2} , k_{3} \} $ and $ \max \{ k_{2} ,k_{3} \} \in [k_{1} -30 , k_{1}]$.
Without loss of generality assume $ k_{2} \geq k_{3}$. Then we estimate as follows; for any fixed $k \in [0, k_{1} +10 ]$ we have
\begin{equation}\label{yu1}
\begin{split}
& \Vert (\i \partial_{t} - (-\Delta)^{s}   + \i)^{-1} \Delta_{k} H_{s}( f_{k_{1}} , (f_{k_{2},j} \bar{f}_{k_{3},j}) ) \Vert_{Z_{k}}
\\
& \lesssim \sup_{e \in \mathscr{E}}
\Vert \mathcal{F}^{-1} \left( \vartheta_{e} \mathcal{F} \left( \Delta_{k} (\i \partial_{t} - (-\Delta)^{s}   + \i)^{-1} \Delta_{k} H_{s}( f_{k_{1}} , f_{k_{2},j} \bar{f}_{k_{3},j} ) \right) \right)\Vert_{Z_{k}}
\\
& \lesssim 2^{-k(2s-1)/2} \sup_{e \in \mathscr{E}}  \Vert H_{s}( f_{k_{1}} , (f_{k_{2},j} \bar{f}_{k_{3},j}) ) \Vert_{L^{1}_{e} L^{2}_{e^{\perp},t}},
\end{split}
\end{equation}
and for any $ e \in \mathscr{E}$ we have 
\begin{equation}\label{yu2}
\begin{split}
& \Vert H_{s}( f_{k_{1}} , (f_{k_{2},j} \bar{f}_{k_{3},j}) ) \Vert_{L^{1}_{e} L^{2}_{e^{\perp},t}}
\\
& \lesssim \Vert (-\Delta)^s ( f_{k_{1}} f_{k_{2} ,j} \bar{f}_{k_{3},j} ) \Vert_{L^{1}_{e}L^{2}_{e^{\perp},t}} 
\\
& + \Vert f_{k_{2},j} \bar{f}_{k_{3},j}  (-\Delta)^s ( f_{k_{1}}) \Vert_{L^{1}_{e}L^{2}_{e^{\perp},t}} 
\\
& + \Vert f_{k_{1}} (-\Delta)^s ( f_{k_{2},j} \bar{f}_{k_{3},j} ) \Vert_{L^{1}_{e}L^{2}_{e^{\perp},t}}.
\end{split}
\end{equation}
We estimate each term on the right hand side of \eqref{yu2}. Starting with the first term, using Lemma \ref{Multiplier} and the maximal estimate we obtain
\begin{equation}\label{yu3}
\begin{split}
& \Vert (-\Delta)^s ( f_{k_{1}} f_{k_{2} ,j} \bar{f}_{k_{3},j} ) \Vert_{L^{1}_{e}L^{2}_{e^{\perp},t}}
\\
& \lesssim 2^{2sk_{1}} \Vert  ( f_{k_{1}} f_{k_{2} ,j} \bar{f}_{k_{3},j} ) \Vert_{L^{1}_{e}L^{2}_{e^{\perp},t}}
\\
& \lesssim 2^{k_{1} (2s-1) +k_{1}} \Vert f_{k_{1}} \Vert_{L^{2}_{e} L^{\infty}_{e^{\perp},t}} \Vert f_{k_{2},j} \bar{f}_{k_{3},j} \Vert_{L^{2}(\R^{n+1})}
\\
& \lesssim 2^{k_{1}(2s-1)} 2^{k_{1} (n+1)/2} \Vert f_{k_{1}} \Vert_{Z_{k_{1}}} \Vert f_{k_{2},j} \bar{f}_{k_{3},j} \Vert_{L^{2}(\R^{n+1})}.
\end{split}
\end{equation}
Using Lemma \ref{es1} and the assumption $ 2^{k_{1}} \approx 2^{k_{2}}$, we bound the right hand side of \eqref{yu3} by
\begin{equation}\label{yu4}
\begin{split}
2^{k_{1}(2s-1)/2} 2^{k_{2} (n+1)/2} \Vert f_{k_{1}} \Vert_{Z_{k_{1}}} 2^{k_{3} (n+1)/2} \Vert f_{k_{2},j} \Vert_{Z_{k_{2}}} \Vert f_{k_{3},j } \Vert_{Z_{k_{3}}}.
\end{split}
\end{equation}
In the same manner we can bound the remaining terms on the right hand side of \eqref{yu2} with 
$$
2^{k_{1}(2s-1)/2} 2^{k_{2} (n+1)/2} \Vert f_{k_{1}} \Vert_{Z_{k_{1}}} 2^{k_{3} (n+1)/2} \Vert f_{k_{2},j} \Vert_{Z_{k_{2}}} \Vert f_{k_{3},j } \Vert_{Z_{k_{3}}}.
$$
This implies 
\begin{equation}\label{yu5}
\begin{split}
& \Vert  H_{s}( f_{k_{1}} , (f_{k_{2},j} \bar{f}_{k_{3},j}) ) \Vert_{L^{1}_{e}L^{2}_{e^{\perp},t}}
\\
& \lesssim 2^{k_{1}(2s-1)/2} 2^{k_{2} (n+1)/2} \Vert f_{k_{1}} \Vert_{Z_{k_{1}}} 2^{k_{3} (n+1)/2} \Vert f_{k_{2},j} \Vert_{Z_{k_{2}}} \Vert f_{k_{3},j } \Vert_{Z_{k_{3}}}.
\end{split}
\end{equation}
Plugging the bound \eqref{yu5} back into \eqref{yu1} we obtained
\begin{equation}\label{yu6}
\begin{split}
& \Vert (\i \partial_{t} -(-\Delta)^s + \i)^{-1} \Delta_{k} H_{s}( f_{k_{1}} , (f_{k_{2},j} \bar{f}_{k_{3},j}) ) \Vert_{Z_{k}}
\\
& \lesssim 2^{(k_{1} - k) (2s-1)/2} 2^{k_{2} (n+1)/2} 2^{k_{3} (n+1)/2} \Vert f_{k_{1}} \Vert_{Z_{k_{1}}} \Vert f_{k_{2},j} \Vert_{Z_{k_{2}}} \Vert f_{k_{3},j} \Vert_{Z_{k_{3}}}.
\end{split}
\end{equation}
The bound \eqref{yu6} yields \eqref{vv4}. 
\\
\textbf{Case 1b:}
If $ \max\{ k_{1},k_{2},k_{3} \} =k_{1}$ and $\max\{ k_{2},k_{3} \} \leq k_{1} -30 $. In this case we use Lemma \ref{comest2}
\begin{equation}\label{qz6}
\begin{split}
& \sum_{k=0}^{k_{1}+10} 2^{\sigma k } \Vert(\i \partial_{t} -(-\Delta)^s + \i)^{-1} \Delta_{k} H_{s} ( f_{k_{1}}, ( f_{k_{2},j} \bar{ f}_{k_{3},j})) \Vert_{Z_{k}}
\\
& \lesssim \sum_{k=0}^{k_{1} + 10} 2^{\sigma k} \sup_{e \in \mathscr{E}} \Vert (\i \partial_{t} - (-\Delta)^{s}   + \i)^{-1} \Delta_{k} H_{s} ( \mathcal{F}^{-1} ( \vartheta_{e} \hat{f}_{k_{1}}), ( f_{k_{2},j} \bar{ f}_{k_{3},j})) \Vert_{Z_{k}}
\\
& \lesssim  \sum_{k=0}^{k_{1} +10} \sup_{e \in \mathscr{E}} 2^{\sigma k} 2^{-k(2s-1)/2} \Vert \Delta_{k} H_{s} ( \mathcal{F}^{-1}(\vartheta_{e} \hat{f}_{k_{1}}) , ( f_{k_{2},j} \bar{f}_{k_{3},j} ) ) \Vert_{L^{1}_{e} L^{2}_{e^{\perp},t}}
\\
& \lesssim 2^{k_{1} \sigma} 2^{k_{2} (n+1)/2} 2^{k_{3} (n+1)/2} \Vert f_{k_{1}} \Vert_{Z_{k_{1}}} \Vert f_{k_{2},j} \Vert_{Z_{k_{2}} } \Vert f_{k_{3},j} \Vert_{Z_{k_{3}} }
\end{split}
\end{equation}
This proves \eqref{vv4}.
\\
\textbf{Case 2:}
Next, we consider the case $\max \{ k_{1} , k_{2} , k_{3} \} \in \{ k_{2}, k_{3} \}$. By symmetry we may assume $ k_{2} \geq k_{3}$. 
Then we need to prove 
\begin{equation}\label{vvv4}
\begin{split}
& \sum_{k_{1},k_{2},k_{3}\geq 0 } \sum_{k=0}^{k_{2} +10 } 2^{\sigma k}  \Vert (\i \partial_{t} - (-\Delta)^{s}   + \i)^{-1} \Delta_{k} H_{s}( f_{k_{1}} , (f_{k_{2},j} \bar{f}_{k_{3},j}) ) \Vert_{Z_{k}}
\\
& \lesssim \sum_{k_{1}, k_{2}, k_{3} \geq 0} 2^{k_{1}(n+1)/2 } 2^{\sigma k_{2}} 2^{k_{3}(n+1)/2 } \Vert f_{k_{1}} \Vert_{Z_{k_{1}}} \Vert f_{k_{2},j} \Vert_{Z_{k_{2}}} \Vert f_{k_{3},j} \Vert_{Z_{k_{3}}}.
\end{split}
\end{equation}
As before, to prove \eqref{vvv4} we consider several sub-cases
\\
\textbf{Case 2a: }If $ \max \{ k_{1} , k_{3} \} \in [k_{2} - 30,k_{2} ]$ then assume without loss of generality that $k_{1} \geq k_{3} $ ( the argument is the same if $k_{3} \geq k_{1}$ ) then this is precisely Case 1a above.\\
\textbf{Case 2b:} If $ \max \{ k_{1} , k_{3} \} \leq k_{2} -30$. Then we use Lemma \ref{comest3} to obtain
\begin{equation}\label{qz2}
\begin{split}
& \sum_{k=0}^{k_{2} + 10 } 2^{\sigma k}  \Vert(\i \partial_{t} -(-\Delta)^s + \i)^{-1} \Delta_{k} H_{s} ( f_{k_{1}},  f_{k_{2},j} \bar{ f}_{k_{3},j} )\Vert_{Z_{k}}
\\
& \lesssim \sum_{k=0}^{k_{2} + 10} 2^{\sigma k} \sup_{e \in \mathscr{E}} \Vert (\i \partial_{t} - (-\Delta)^{s}   + \i)^{-1} \Delta_{k} H_{s} ( f_{k_{1}}, \mathcal{F}^{-1} ( \vartheta_{e} \hat{f}_{k_{2},j}) \bar{ f}_{k_{3},j}) \Vert_{Z_{k}}
\\
& \lesssim 2^{\sigma k_{2}} 2^{k_{1} (n+1)/2} 2^{k_{3}(n+1)/2} \Vert f_{k_{1}} \Vert_{Z_{k_{1}}} \Vert  f_{k_{2},j} \Vert_{Z_{k_{2}}} \Vert f_{k_{3},j}  \Vert_{Z_{k_{3}}} 
\end{split}
\end{equation}
which clearly implies \eqref{vvv4}. 
Case 1, and Case 2 imply \eqref{nonll}. Now using Corollary \ref{Algebra}, we obtain  
\begin{equation}
\begin{split}
& \Vert H_{s} ( f , \frac{1}{1+ \vert f \vert^{2}} ) \Vert_{N^{\sigma}} 
\\
& \lesssim \sum_{ j \geq 1} \Vert f \Vert_{F^{\sigma}} \Vert f^{j} \Vert^{2}_{F^{n/2}} + 2\Vert f \Vert_{F^{n/2}} \Vert f^{j} \Vert_{F^{n/2}} \Vert f^{j} \Vert_{F^{\sigma}}
\\
& \lesssim \Vert f \Vert_{F^{\sigma}} \Vert f \Vert^{2}_{F^{n/2}} \sum_{j \geq 1} ( C_{A} \epsilon)^{2j-1} + \Vert f \Vert^{2}_{F^{n/2}} \sum_{j \geq 1} (j+1)( C_{A} \epsilon)^{2(j-1)} \Vert f \Vert_{F^{\sigma}}
\end{split}
\end{equation}
This concludes the Lemma once we choose $ \epsilon \ll \frac{1}{C_{A} + 100}$.
\end{proof}
Notice that the above proof yields the following inequality; for any $ f,g,h \in F^{\sigma}$ with $ \sigma \geq \frac{n+1}{2}$ one has  
\begin{equation}\label{nonl1}
\begin{split}
 & \Vert H_{s} ( \tilde{f}, \tilde{g} \tilde{h} ) \Vert_{N^{\sigma}} 
 \lesssim \Vert f \Vert_{F^{\sigma}} \Vert g \Vert_{F^{\frac{n+1}{2}}} \Vert h \Vert_{F^{\frac{n+1}{2}}}
 \\
 & + \Vert f \Vert_{F^{\frac{n+1}{2}}} \Vert g \Vert_{F^{\sigma}} \Vert h \Vert_{F^{\frac{n+1}{2}}}
 \\
 & + \Vert f \Vert_{F^\frac{n+1}{2}} \Vert g \Vert_{F^{\frac{n+1}{2}}} \Vert h \Vert_{F^{\sigma}}, 
\end{split}
\end{equation}
where $ \tilde{f} \in \{ f, \bar{f} \}$, and similarly for $ \tilde{g}, \tilde{h}$. 

Next, we prove the remaining nonlinear term. This will follow from the following trilinear estimates 
\begin{lemma}\label{TrilinearProp}
Let $ f,g,h \in F^{\sigma}$ with $ \sigma \geq \frac{n+1}{2}$. Then we have 
\begin{equation}\label{eqmod1}
\begin{split}
& \Vert \tilde{f} H_{s} ( \tilde{g},\tilde{h} ) \Vert_{N^{\sigma}} 
 \lesssim \Vert f \Vert_{F^{\sigma}} \Vert g \Vert_{F^{\frac{n+1}{2}}} \Vert h \Vert_{F^{\frac{n+1}{2}}}
 \\
 & + \Vert f \Vert_{F^{\frac{n+1}{2}}} \Vert g \Vert_{F^{\sigma}} \Vert h \Vert_{F^{\frac{n+1}{2}}}
 \\
 & + \Vert f \Vert_{F^\frac{n+1}{2}} \Vert g \Vert_{F^{\frac{n+1}{2}}} \Vert h \Vert_{F^{\sigma}},
\end{split}
\end{equation}
where $ \tilde{f} \in \{ f, \bar{f} \}$ and similarly for $\tilde{g},\tilde{h} $. 
\end{lemma}
\begin{proof}
To prove \eqref{eqmod1} it suffices to prove 
\begin{equation}\label{eqmod2}
\begin{split}
& \sum_{k \geq 0}^{\max \{ k_{1},k_{2},k_{3} \} +C} 2^{\sigma k} \Vert \Delta_{k} ( \i \partial_{t } - ( - \Delta)^{s} + \i)^{-1} \tilde{f}_{k_{1}} H_{s} ( \tilde{g}_{k_{2}},\tilde{h}_{k_{3}}) \Vert_{Z_{k}} 
\\
& \lesssim 2^{\sigma k_{1}} \Vert f_{k_{1}} \Vert_{Z_{k_{1}}} 2^{k_{2} \frac{n+1}{2}} \Vert g_{k_{2}} \Vert_{Z_{k_{2}}} 2^{k_{3} \frac{n+1}{2}} \Vert h_{k_{3}} \Vert_{Z_{k_{3}}}
\\
&+ 2^{k_{1} \frac{n+1}{2}} \Vert f_{k_{1}} \Vert_{Z_{k_{1}}} 2^{k_{2}\sigma} \Vert g_{k_{2}} \Vert_{Z_{k_{2}}} 2^{k_{3} \frac{n+1}{2}} \Vert h_{k_{3}} \Vert_{Z_{k_{3}}}
\\
&+ 2^{k_{1}\frac{n+1}{2}} \Vert f_{k_{1}} \Vert_{Z_{k_{1}}} 2^{k_{2} \frac{n+1}{2}} \Vert g_{k_{2}} \Vert_{Z_{k_{2}}} 2^{k_{3} \sigma } \Vert h_{k_{3}} \Vert_{Z_{k_{3}}}.
\end{split}
\end{equation}
To that end, we consider several cases.
\\
\textbf{Case 1:} If $ \max \{ k_{1},k_{2} ,k_{3} \} =k_{1}$. First, assume $ 2^{ \max \{ k_{2}, k_{3} \}} \approx 2^{k_{1}}$. Then by symmetry we may assume $ k_{2} = \max\{ k_{2},k_{3} \}$ and estimate as follows 
\begin{equation}\label{eqmod3}
\begin{split}
& \Vert \Delta_{k}( \i \partial_{t } - ( - \Delta)^{s} + \i)^{-1} \tilde{f}_{k_{1}} H_{s} ( \tilde{g}_{k_{2}},\tilde{h}_{k_{3}}) \Vert_{Z_{k}} 
\\
& \lesssim 2^{-k(2s-1)/2} \max_{e \in \mathscr{E}}\Vert \tilde{f}_{k_{1}} H_{s} ( \tilde{g}_{k_{2}},\tilde{h}_{k_{3}}) \Vert_{L^{1}_{e} L^{2}_{e^{\perp},t}}
\\
& \lesssim  2^{-k(2s-1)/2}\max_{e \in \mathscr{E}} \Vert \tilde{f}_{k_{1}} \Vert_{L^{2}_{e} L^{\infty}_{e^{\perp},t}} \Vert H_{s}( \tilde{g}_{k_{2}}, \tilde{h}_{k_{3}} ) \Vert_{L^{2}(\R^{n+1})}
\\
& \lesssim 2^{-k(2s-1)/2} 2^{k_{1} \frac{n-1}{2}} \Vert f_{k_{1}} \Vert_{Z_{k_{1}}} \Vert H_{s}( \tilde{g}_{k_{2}}, \tilde{h}_{k_{3}} ) \Vert_{L^{2}(\R^{n+1})}
\\
& \lesssim 2^{(k_{1}-k)(2s-1)/2} \Vert f_{k_{1}} \Vert_{Z_{k_{1}}} 2^{k_{2} \frac{n+1}{2}} \Vert g_{k_{2}} \Vert_{Z_{k_{2}}} 2^{k_{3} \frac{n+1}{2}} \Vert h_{k_{3}} \Vert_{Z_{k_{3}}},
\end{split}
\end{equation}
where in the last line we used Lemma \ref{comest} and the assumption that $ 2^{k_{1}} \approx 2^{k_{2}}$. Multiplying by $ 2^{\sigma k} $ and summing over $k \leq k_{1} +C $ yields \eqref{eqmod2} in this case. If on the other hand $ 2^{\max \{ k_{2}, k_{3} \}} \ll 2^{k_{1}}$. Then in we have 
\begin{equation}\label{eqmod5}
\begin{split}
&  \Vert \Delta_{k}( \i \partial_{t } - ( - \Delta)^{s} + \i)^{-1} \tilde{f}_{k_{1}} H_{s} ( \tilde{g}_{k_{2}},\tilde{h}_{k_{3}}) \Vert_{Z_{k}}
\\
& 2^{-k_{1} (2s-1)/2} \max_{e \in \mathscr{E}}\Vert \mathcal{F}^{-1} \left( \vartheta_{e} \hat{\tilde{f}}_{k_{1}} \right) H_{s} ( \tilde{g}_{k_{2}}, \tilde{h}_{k_{3}}) \Vert_{L^{1}_{e} L^{2}_{e^{\perp},t}}
\\
& \lesssim 2^{-k_{1} (2s-1)/2} \max_{e \in \mathscr{E}} \Vert \mathcal{F}^{-1} \left( \vartheta_{e} \hat{\tilde{f}}_{k_{1}} \right) \Vert_{L^{\infty}_{e} L^{2}_{e^{\perp},t}} \Vert H_{s} ( \tilde{g}_{k_{2}}, \tilde{h}_{k_{3}} ) \Vert_{L^{1}_{e} L^{\infty}_{e^{\perp},t}}
\\
& \lesssim \Vert f_{k_{1}} \Vert_{Z_{k_{1}}} 2^{k_{2} \frac{n+1}{2}} \Vert g_{k_{2}} \Vert_{Z_{k_{2}}} 2^{k_{3} \frac{n+1}{2}} \Vert h_{k_{3}} \Vert_{Z_{k_{3}}}, 
\end{split}
\end{equation}
where in the last line we used the local smoothing estimate and Lemma \ref{comest4}. Multiplying by $ 2^{k \sigma}$ and summing ove $k \leq k_{1} +C$ yields the result in this case. This concludes case 1.
\\
\textbf{Case 2:} If $ \max \{ k_{1},k_{2},k_{3} \} \in \{ k_{2},k_{3} \}$. By symmetry we may assume $ \max \{ k_{2},k_{3} \} = k_{2}.$ Then we estimate as follows 
\begin{equation*}
\begin{split}
& \Vert \Delta_{k}( \i \partial_{t } - ( - \Delta)^{s} + \i)^{-1} \tilde{f}_{k_{1}} H_{s} ( \tilde{g}_{k_{2}},\tilde{h}_{k_{3}}) \Vert_{Z_{k}} 
\\
& \lesssim 2^{-k(2s-1)/2} \max_{e \in \mathscr{E}}\Vert \tilde{f}_{k_{1}} H_{s} ( \tilde{g}_{k_{2}},\tilde{h}_{k_{3}}) \Vert_{L^{1}_{e} L^{2}_{e^{\perp},t}}
\\
& \lesssim  2^{-k(2s-1)/2}\max_{e \in \mathscr{E}} \Vert \tilde{f}_{k_{1}} \Vert_{L^{2}_{e} L^{\infty}_{e^{\perp},t}} \Vert H_{s}( \tilde{g}_{k_{2}}, \tilde{h}_{k_{3}} ) \Vert_{L^{2}(\R^{n+1})}
\\
& \lesssim 2^{-k(2s-1)/2} 2^{k_{1} \frac{n-1}{2}} \Vert f_{k_{1}} \Vert_{Z_{k_{1}}} \Vert H_{s}( \tilde{g}_{k_{2}}, \tilde{h}_{k_{3}} ) \Vert_{L^{2}(\R^{n+1})}
\\
& \lesssim 2^{(k_{2} -k)(2s-1)/2} 2^{k_{1} \frac{n+1}{2}} \Vert f_{k_{1}} \Vert_{Z_{k_{1}}} 2^{k_{3} \frac{n+1}{2}} \Vert h_{k_{3}} \Vert_{Z_{k_{3}}} \Vert g_{k_{2}} \Vert_{Z_{k_{2}}},
\end{split}
\end{equation*}
where in the last line we used Lemma \ref{comest}. Multiplying by $ 2^{\sigma k}$ and summing over $k \in \leq k_{2} +C$ gives the result.

\end{proof}
As a consequence we obtain the estimates for the second nonlinear term
\begin{proposition}\label{secondnon2}
Let $ \sigma \geq \frac{n+1}{2}$. Then there exists $ \epsilon := \epsilon(n,s,\sigma) > 0$ so that whenever $ \Vert f \Vert_{F^{\frac{n+1}{2}}} \leq \epsilon $ then we have 
\begin{equation}\label{nonlinear2}
\Vert f H_{s} ( f , \frac{\bar{f}}{1+|f|^{2}} ) \Vert_{N^{\sigma}} \lesssim \Vert f \Vert^{2}_{F^{(n+1)/2}} \Vert f \Vert_{F^{\sigma}}.
\end{equation}
    
\end{proposition}
\begin{proof}
From Lemma \ref{TrilinearProp} we obtain 
\begin{equation}\label{eqmod7}
\begin{split}
& \Vert f H_{s} ( f , \frac{\bar{f}}{1+|f|^{2}} ) \Vert_{N^{\sigma}}
\\
& \lesssim \Vert f \Vert_{F^{\sigma}} \Vert f \Vert_{F^{\frac{n+1}{2}}} \Vert \frac{\bar{f}}{1+|f|^{2}} \Vert_{\tilde{F}^{\frac{n+1}{2}}}
\\
& + \Vert \frac{\bar{f}}{1+|f|^{2}} \Vert_{\tilde{F}^{\sigma}} \Vert f \Vert_{F^{\frac{n+1}{2}}}.
\end{split}
\end{equation}
Expand the function $ \frac{1}{1+|f|^{2}}$ using Taylor, then from the algebra properties, Corollary \ref{Algebratilde} and the smallness assumption we obtain 
\begin{equation*}
\Vert \frac{\bar{f}}{1+|f|^{2}} \Vert_{\tilde{F}^{\frac{n+1}{2}}}  \lesssim \Vert f \Vert_{F^{\frac{n+1}{2}}} \sum_{j=0}^{\infty} ( C_{\tilde{A}} \epsilon)^{j}
\end{equation*}
If $ \epsilon >0$ was taken small enough then we have 
\begin{equation}\label{eqmod8}
\Vert \frac{\bar{f}}{1+|f|^{2}} \Vert_{\tilde{F}^{\frac{n+1}{2}}}  \lesssim \Vert f \Vert_{F^{\frac{n+1}{2}}}.
\end{equation}
Similarly, for $ \sigma \geq \frac{n+1}{2}$ and $ \epsilon $ small we obtain using Corollary \ref{Algebratilde}
\begin{equation}\label{eqmod9}
\Vert \frac{\bar{f}}{1+|f|^{2}} \Vert_{\tilde{F}^{\sigma}}  \lesssim  \Vert f \Vert_{F^{\sigma}}.
\end{equation}
Plugging \eqref{eqmod8}, \eqref{eqmod9} back into \eqref{eqmod7} gives the result.
\end{proof}

\begin{proof}[Proof of Theorem \ref{NonlinearEst}]
Fix $ \sigma \geq \frac{n+1}{2}$ and take $ \epsilon >0$ so that the conclusion of Proposition \ref{firstn}, Proposition \ref{secondnon2}. Then the result follows immediately.
\end{proof}

We close this section with proving the following Lipschitz bounds, these bounds follow from \eqref{nonl1}, Lemma \ref{TrilinearProp}.

\begin{proposition}\label{LipschitzboundN}
Let $ \sigma \geq \sigma_{0} \geq \frac{n+1}{2}$.
There exists $ \epsilon := \epsilon(n,s, \sigma_{0},\sigma) > 0$ so that for any  $ f, g \in F^{\sigma} $ with $ \Vert f\Vert_{F^{\sigma_{0}}} \leq \epsilon, $ and $ \Vert g \Vert_{F^{\sigma_{0}}} \leq \epsilon$, the following Lipschitz bound holds
\begin{equation}\label{lastlip}
\begin{split}
& \Vert \mathcal{N}(f) - \mathcal{N}(g) \Vert_{N^{\sigma}}
\\
&\lesssim \epsilon \Vert f - g \Vert_{F^{\sigma_{0}}}( \Vert f \Vert_{F^{\sigma}} + \Vert g \Vert_{F^{\sigma}})
 + \epsilon^{2} \Vert f - g \Vert_{F^{\sigma}}.
\end{split}
\end{equation}
In particular, taking $ \sigma = \sigma_{0}$ then we have 
 \begin{equation}\label{LipNonlinear}
\Vert \mathcal{N} (f) - \mathcal{N}(g) \Vert_{N^{\sigma_{0}}} \lesssim \epsilon^{2} \Vert f-g \Vert_{F^{\sigma_{0}}}.
\end{equation}

 \end{proposition}
 \begin{proof}
We estimate each nonlinear term separately. Namely, we first prove
\begin{equation}\label{lip1}
\begin{split}
& \Vert H_{s} ( f, \frac{1}{1+\vert f \vert^{2}} ) - H_{s} ( g, \frac{1}{1+\vert g \vert^{2}} ) \Vert_{N^{\sigma}}
\\
& \lesssim \epsilon \Vert f -g \Vert_{F^{\sigma_{0}}}( \Vert f \Vert_{F^{\sigma}} + \Vert g \Vert_{F^{\sigma}} ) + \epsilon^{2} \Vert f -g \Vert_{F^{\sigma}}
\end{split}
\end{equation}
By expanding using Taylor, we can write 
\begin{equation*}
\begin{split}
& H_{s} ( f, \frac{1}{1+\vert f \vert^{2}} ) - H_{s} ( g, \frac{1}{1+\vert g \vert^{2}} ) 
\\
& = \sum_{j=1}^{\infty} (-1)^{j}\left( H_{s} ( f, f^{j} \bar{f}^{j}) - H_{s}(g , g^{j} \bar{g}^{j} ) \right)
\end{split}
\end{equation*}
and for each $j \geq 1$ we have
\begin{equation}\label{lk1}
\begin{split}
&  H_{s} ( f, f^{j} \bar{f}^{j}) - H_{s}(g , g^{j} \bar{g}^{j} ) 
\\
& = H_{s} ( f-g, f^{j} \bar{f}^{j} ) + H_{s} (g,f^{j} \bar{f}^{j}) - H_{s} (g, g^{j} \bar{g}^{j} )
\\
& =  H_{s} ( f-g, f^{j} \bar{f}^{j} ) + H_{s} (g, \bar{f}^{j} ( f^{j} -g^{j} ) ) + H_{s} (g,(  \bar{f}^{j}- \bar{g}^{j} ) g^{j} )
\end{split}
\end{equation}
using \eqref{nonl1} and Corollary \ref{Algebra} we estimate the $N^{\sigma}$ norm of each term on the right hand side of \eqref{lk1}. Starting with the first term we have by \eqref{nonl1}
\begin{equation*}
\begin{split}
& \Vert H_{s} (f -g , f^{j} \bar{f}^{j} ) \Vert_{N^{\sigma}}
\\
& \lesssim \Vert f- g \Vert_{F^{\sigma}} \Vert f^{j} \Vert^{2}_{F^{\frac{n+1}{2}}} + \Vert f- g \Vert_{F^{\frac{n+1}{2}}} \Vert f^{j} \Vert_{F^{\frac{n+1}{2}}}\Vert f^{j} \Vert_{F^{\sigma}} 
\end{split}
\end{equation*}
Using Corollary \ref{Algebra} we obtain 
\begin{equation*}
\begin{split}
& \Vert H_{s} (f -g , f^{j} \bar{f}^{j} ) \Vert_{N^{\sigma}}
\\
& \lesssim 
\epsilon^{2} \Vert f - g \Vert_{F^{\sigma}} ( C_{A} \epsilon)^{2(j-1)} + \Vert f- g \Vert_{F^{\sigma_{0}}} (C_{A} \epsilon)^{j-1} \epsilon  \left( j C^{j-1}_{A}\epsilon^{j-1}\right) \Vert f \Vert_{F^{\sigma}}
\end{split}
\end{equation*}
Choosing $ \epsilon \ll \frac{1}{C_{A} +100}$ and summing over $j \geq 1$ we obtain the bound 
\begin{equation}\label{vvp1}
\sum_{j \geq 1} \Vert H_{s} (f -g , f^{j} \bar{f}^{j} ) \Vert_{N^{\sigma}} \lesssim \epsilon^{2} \Vert f-g \Vert_{F^{\sigma}}+ \epsilon \Vert f - g \Vert_{F^{\sigma_{0}}} \Vert f \Vert_{F^{\sigma}}.
\end{equation}
Next, we handle the second term on the right hand side of \eqref{lk1} similarly. In particular, using \eqref{nonl1} we obtain 
\begin{equation*}
\begin{split}
& \Vert H_{s} ( g, \bar{f}^{j} ( f^{j} - g^{j} )) \Vert_{N^{\sigma}}
\\
& \lesssim \Vert g \Vert_{F^{\sigma}} \Vert f^{j} \Vert_{F^{\frac{n+1}{2}}} \Vert f^{j} - g^{j} \Vert_{F^{\frac{n+1}{2}}} 
\\
& + \Vert g \Vert_{F^{\frac{n+1}{2}}} \Vert f^{j} \Vert_{F^{\sigma}} \Vert f^{j} - g^{j} \Vert_{F^{\frac{n+1}{2}}} 
\\
&+ \Vert f^{j} -g^{j} \Vert_{F^{\sigma}} \Vert f^{j} \Vert_{F^{\frac{n+1}{2}}} \Vert g \Vert_{F^{\frac{n+1}{2}}}
\end{split}
\end{equation*}
using Corollary \ref{Algebra} we obtain
\begin{equation*}
\begin{split}
& \Vert H_{s} ( g, \bar{f}^{j} ( f^{j} - g^{j} )) \Vert_{N^{\sigma}}
\\
& \lesssim \epsilon ( C_{A} \epsilon)^{j-1} \Vert g \Vert_{F^{\sigma}} j (C_{A} \epsilon)^{j-1} \Vert f -g \Vert_{F^{\sigma_{0}}}
\\
& + \epsilon (j ( C_{A} \epsilon)^{j-1})^{2} \Vert f \Vert_{F^{\sigma}} \Vert f -g \Vert_{F^{\sigma_{0}}} 
\\
& + \epsilon (C_{A} \epsilon)^{2j}  \Vert f- g \Vert_{F^{\sigma_{0}}} j^{2} (C_{A} \epsilon)^{j-1} ( \Vert f \Vert_{F^{\sigma}} + \Vert g \Vert_{F^{\sigma}} ) + \epsilon^{2} j( C_{A} \epsilon)^{2(j-1)} \Vert f-g \Vert_{F^{\sigma}}.
\end{split}
\end{equation*}
Choosing $ \epsilon \ll \frac{1}{C_{A} +100}$ and summing over $j \geq 1$ we obtain 
\begin{equation}\label{vvp2}
\begin{split}
 & \sum_{j \geq 1} \Vert H_{s} ( g, \bar{f}^{j} ( f^{j} - g^{j} )) \Vert_{N^{\sigma}} 
 \\
 & \lesssim \epsilon \Vert f-g \Vert_{F^{\sigma_{0}}} ( \Vert f \Vert_{F^{\sigma}} + \Vert g \Vert_{F^{\sigma}} ) + \epsilon^{2} \Vert f -g \Vert_{F^{\sigma}}.
\end{split}
 \end{equation}
The same proof also give the estimate for the third term on the right hand side of \eqref{lk1}. Namely, 
\begin{equation}\label{vvp3}
\begin{split}
& \sum_{j \geq 1} \Vert H_{s} ( g, (\bar{f}^{j}- \bar{g}^{j}) g^{j} \Vert_{N^{\sigma}} 
\\
& \lesssim \epsilon \Vert f-g \Vert_{F^{\sigma_{0}}} ( \Vert f \Vert_{F^{\sigma}} + \Vert g \Vert_{F^{\sigma}} ) +  \epsilon^{2} \Vert f - g \Vert_{F^{\sigma}}.
\end{split}
\end{equation}
The estimates \eqref{vvp1}, \eqref{vvp2}, and \eqref{vvp3} yield 
\eqref{lip1}.

Next, we handle the second nonlinear term. That is, we prove

\begin{equation}\label{lip2}
\begin{split}
& \Vert f H_{s} (f, \frac{ \bar{f}}{1+|f|^{2}}) - g H_{s}( g, \frac{\bar{g}}{1+|g|^{2}}) \Vert_{N^{\sigma}}
\\
&\lesssim \epsilon \Vert f -g \Vert_{F^{\sigma_{0}}} ( \Vert f \Vert_{F^{\sigma}} + \Vert g \Vert_{F^{\sigma}} ) + \epsilon^{2} \Vert f - g \Vert_{F^{\sigma}}.
\end{split}
\end{equation}
Using the same trick, we expand by Taylor to obtain
\begin{equation*}
\begin{split}
& f H_{s} (f, \frac{ \bar{f}}{1+|f|^{2}}) - g H_{s}( g, \frac{\bar{g}}{1+|g|^{2}})
\\
& = \sum_{ j \geq 0} (-1)^{j} \left(  f H_{s} ( f, \bar{f} f^{j} \bar{f}^{j})  -  g H_{s} ( g, \bar{g} g^{j} \bar{g}^{j}) \right). 
\end{split}
\end{equation*}
And notice that for each fixed $j \geq 1$ we have, using Lemma \ref{TrilinearProp}, and Corollary \ref{Algebratilde} 
\begin{equation*}
\begin{split}
& \Vert  f H_{s} ( f, \bar{f} f^{j} \bar{f}^{j})  -  g H_{s} ( g, \bar{g} g^{j} \bar{g}^{j}) \Vert_{N^{\sigma}}
\\
& \lesssim ( (2j+1)^{2} ( C_{\tilde{A}} \epsilon)^{2j}) \left( \epsilon \Vert f-g \Vert_{F^{\sigma_{0}}} ( \Vert f \Vert_{F^{\sigma}} + \Vert g \Vert_{F^{\sigma}} ) +  \epsilon^{2} \Vert f - g \Vert_{F^{\sigma}} \right) 
\end{split}
\end{equation*}
If $ \epsilon > 0 $ was chosen small enough then we obtain 
\begin{equation*}
\begin{split}
& \Vert \sum_{ j \geq 1}  \left(  f H_{s} ( f, \bar{f} f^{j} \bar{f}^{j})  -  g H_{s} ( g, \bar{g} g^{j} \bar{g}^{j}) \right) \Vert_{N^{\sigma}} 
\\
& \lesssim \epsilon \Vert f-g \Vert_{F^{\sigma_{0}}} ( \Vert f \Vert_{F^{\sigma}} + \Vert g \Vert_{F^{\sigma}} ) +  \epsilon^{2} \Vert f - g \Vert_{F^{\sigma}}.
\end{split}
\end{equation*}
Similarly, for $j=0$ we use trilinearity of the operator $ f H_{s}(g,h)  $ to obtain 
\begin{equation*}
\begin{split}
& \Vert  \left(  f H_{s} ( f, \bar{f} )  -  g H_{s} ( g, \bar{g}) \right) \Vert_{N^{\sigma}} 
\\
& \lesssim \epsilon \Vert f-g \Vert_{F^{\sigma_{0}}} ( \Vert f \Vert_{F^{\sigma}} + \Vert g \Vert_{F^{\sigma}} ) +  \epsilon^{2} \Vert f - g \Vert_{F^{\sigma}}.
\end{split}
\end{equation*}
Combining the above readily yields \eqref{lip2}.
\end{proof}

\section{Proof of Theorem \ref{Redeq} and \Cref{MainResult}}
In this section we prove Theorem \ref{Redeq}. First, we recall the linear estimates which were already proven in \cite{model}. We will sketch the proof for completeness. Throughout this section we fix a Schwartz function
$$
\psi : \R \to [0,1]
$$
which is supported in $ [-2,2]$ and equal to $1$ on $[-1,1]$.
\begin{lemma}\label{LinearEstimates}
Let $ \sigma \geq 0 $ then the following linear estimates hold
\begin{enumerate}
    \item for any $ f_{0} \in B^{\sigma}_{2,1}$ one has
    $$
    \Vert \psi(t) e^{ -\i t (-\Delta)^s} f_{0} \Vert_{F^{\sigma}} \lesssim \Vert f_{0} \Vert_{B^{\sigma}_{2,1}}
    $$
    \item Let $ f(x,t) : = \i \psi(t) \int_{0}^{t} e^{-\i(t-r) (-\Delta)^s} F(x,r) dr $ then we have
    $$
    \Vert f \Vert_{F^{\sigma}} \lesssim \Vert F \Vert_{N^{\sigma}}
    $$
\end{enumerate}
\end{lemma}
\begin{proof}
First, we prove (1)
\begin{equation*}
\begin{split}
& \Vert \psi(t) e^{ -\i t (-\Delta)^s} f_{0} \Vert_{F^{\sigma}}
\\
& = \sum_{k=0}^{\infty} 2^{ \sigma k} \Vert \Delta_{k} ( \psi(t) e^{-\i t (-\Delta)^s } f_{0} )\Vert_{Z_{k}}
\\
& \lesssim \sum_{k=0}^{\infty} 2^{ \sigma k} \Vert \Delta_{k} ( \psi(t) e^{-\i t (-\Delta)^s } f_{0} )\Vert_{X_{k}}
\\
& \lesssim_{\psi} \sum_{k=0}^{\infty } 2^{ \sigma k } \Vert  \Delta^{k} f_{0} \Vert_{L^{2}}.
\end{split}
\end{equation*}
Next, we proceed to prove (2). Notice that we have 
\begin{equation*}
\begin{split}
& \mathcal{F}_{\R^{n+1}}( \Delta_{k} f )
\\
& = \mathcal{F}_{\R^{n+1}} (  \Delta_{k} \i \psi(t) \int_{0}^{t} e^{-\i (t-r) (-\Delta)^s} F(x,r) dr )
\\
& = \i \int_{\R} \frac{ \mathcal{F}_{\R^{n+1}} \left( (\i \partial_{t} - (-\Delta)^{s}   + \i)^{-1} \Delta_{k} F \right)(\xi,r) ( r + \vert \xi \vert^{2s} + \i) }{ r + \vert \xi \vert^{2s}} \left( \hat{\psi}(r -\tau) - \hat{\psi}(\tau + \vert \xi\vert^{2s}) \right) dr.
\end{split}
\end{equation*}
Therefore, for $h \in Z_{k}$ set 
$$
T(h) := \mathcal{F}^{-1}_{\R^{n+1}} \left(  \i \int_{\R} \frac{ \hat{h}(\xi,r) ( r + \vert \xi \vert^{2s} + \i)}{ r + \vert \xi \vert^{2s}} \left( \hat{\psi}(r -\tau) - \hat{\psi}(\tau + \vert \xi\vert^{2s}) \right) dr \right)
$$
and to prove (2) it suffices to prove 
\begin{equation}
\Vert T(h) \Vert_{Z_{k}} \lesssim \Vert h \Vert_{Z_{k}}
\end{equation}
If $ h \in X_{k}$ then the result was already proven in \cite[(4.2)]{model}. Next, assume $ h \in Y^{e}_{k,k'}$ with $ k' \in T_{k}$ and $h$ has modulation controlled by $ 2^{2k(s+q -1)-80}$. Then write 
$$
h(x,t) = \frac{ \i \partial_{t} - (-\Delta)^{s}  }{ \i \partial_{t} - (-\Delta)^{s}   + \i} h(x,t) + \frac{ \i}{ \i \partial
_{t} -(-\Delta)^s + \i} h(x,t)
$$
And notice that by Lemma \ref{Embeddings} we have
$$
\Vert ( \i \partial_{t} - (-\Delta)^{s}   + \i)^{-1} h \Vert_{X_{k}} \lesssim \Vert h \Vert_{Y^{e}_{k,k'}}
$$
Therefore, for (2) it suffices to prove
\begin{equation}\label{wlll}
\begin{split}
& \Vert \mathcal{F}^{-1}_{\R^{n+1}} \left(  \i \int_{\R} \hat{h}(\xi,r)  \hat{\psi}(r -\tau) dr \right) \Vert_{Z_{k}} +  \Vert \mathcal{F}^{-1}_{\R^{n+1}} \left(  \i \int_{\R} \hat{h}(\xi,r)  \hat{\psi}(\tau +\vert \xi \vert^{2s}) dr \right) \Vert_{X_{k}}
\\
& \lesssim \Vert h \Vert_{Z_{k}}
\end{split}
\end{equation}
The second term on the left hand side of \eqref{wlll} satisfy the estimate by using Lemma \ref{Energy} with $ t =0$. For the first term, the result follows by repeating the argument in \cite[(4.5)]{model}, which does not depend on the factor $ 2^{-k'(2s-1)/2}$.

\end{proof}
Now we are ready to prove Theorem \ref{Redeq}. This will be a standard fix point argument.
\begin{proof}[Proof of Theorem \ref{Redeq}]
Fix $ \sigma_{0} \geq \frac{n+1}{2}$ and $f_{0} : \R^{n} \to \C$ with $ \Vert f_{0} \Vert_{B^{\sigma_{0}}_{2,1}} \leq \epsilon_{0}$ for $ \epsilon_{0}$ to be chosen later. And
define 
$$
X_{\lambda} : = \{ f \in F^{\sigma_{0}} : \Vert f \Vert_{F^{\sigma_{0}}} \leq \lambda \}
$$
for $ \lambda \ll 1$ so that the conclusion of Theorem \ref{NonlinearEst} holds. Then consider the operator given by 
$$
T(f) : = \psi(t) e^{-\i t (-\Delta)^s} f_{0} - \i \psi(t) \int_{0}^{t} e^{-\i (t-r) (-\Delta)^s} \mathcal{N}(f) (x,r) dr.
$$
From Lemma \ref{LinearEstimates} and Theorem \ref{NonlinearEst} we obtain 
\begin{equation}
\begin{split}
& \Vert T(f) \Vert_{F^{\sigma_{0}}} 
\\
& \leq   C \Vert f_{0} \Vert_{B_{2,1}^{\sigma_{0}}} + C_{\sigma_{0}} \Vert f \Vert^{3}_{F^{\sigma_{0}}} 
\\
& \leq C \epsilon_{0} + C_{\sigma_{0}}\lambda^{3}
\end{split}
\end{equation}
Therefore, for $ \epsilon_{0} , \lambda > 0$ appropriately small, namely $ C \epsilon_{0} \leq \lambda/2$ and $ \lambda^{2} C_{\sigma_{0}} \leq 1/2$,  one obtains 
$$
T : X_{\lambda } \to X_{\lambda}.
$$
Next, notice that by Lemma \ref{LinearEstimates} and Proposition \ref{LipschitzboundN} we obtain that for any $ f,g \in X_{\lambda}$
\begin{equation}
\Vert T(f) - T(g) \Vert_{F^{\sigma_{0}}}
\leq C_{\sigma_{0}} \lambda^{2} \Vert f -g \Vert_{F^{\sigma_{0}}}.
\end{equation}
Hence, for appropriately small $ \lambda > 0$, say $ \lambda^{2} C_{\sigma_{0}} \ll \frac{1}{2}$, we have that $T$ is a contraction and thus, there exists a fixed point 
$$
f = T(f).
$$
Moreover, Theorem \ref{NonlinearEst} gives us the following bound; for any $ \sigma \in [\frac{n+1}{2},\sigma_{0} +100]$ we have 
\begin{equation*}
\begin{split}
& \Vert f \Vert_{F^{\sigma}}
\\
& \lesssim  \Vert f_{0} \Vert_{B^{\sigma}_{2,1}} + \lambda^{2} C_{\sigma_{0}} \Vert f \Vert_{F^{\sigma}}.  
\end{split}
\end{equation*}
Therefore, for our choice of $ \lambda \ll 1$ we obtain 

\begin{equation}\label{xzvb}
\Vert f \Vert_{F^{\sigma}}
\lesssim \Vert f_{0} \Vert_{B^{\sigma}_{2,1}}.
\end{equation}
And hence by Lemma \ref{Energy} we obtain
$$
\sup_{t \in [-1,1]} \Vert f \Vert_{B^{\sigma}_{2,1}} \lesssim \Vert f_{0} \Vert_{B^{\sigma}_{2,1}}.
$$
Moreover, if $ f, g $ are the solutions obtained above for initial data $ f_{0} ,g_{0}$ then we have 
\begin{equation*}
\begin{split}
& f - g = \psi(t) e^{-\i t (-\Delta)^s} ( f_{0} - g_{0} ) - \i \psi(t) \int_{0}^{t} e^{-\i ( t-r) (-\Delta)^s} ( \mathcal{N}(f) - \mathcal{N}(g)) dr
\end{split}
\end{equation*}
and so by Proposition \ref{LipschitzboundN} and Lemma \ref{LinearEstimates} we obtain
\begin{equation}\label{fhg}
\begin{split}
& \Vert f- g \Vert_{F^{\sigma_{0}}}
\\
& \leq C \Vert f_{0} -g_{0} \Vert_{B^{\sigma_{0}}_{2,1}} + C_{\sigma_{0}} \lambda^{2} \Vert f-g \Vert_{F^{\sigma_{0}}}.
\end{split}
\end{equation}
Hence, for our choice of $ \lambda$ this gives
\begin{equation}\label{tur}
    \Vert f- g \Vert_{F^{\sigma_{0}}} \lesssim \Vert f_{0}  - g_{0} \Vert_{B^{\sigma_{0}}_{2,1}}.
    \end{equation}
This implies, using Lemma \ref{Energy}, 
$$
\sup_{t \in [-1,1]} \Vert f-g \Vert_{B^{\sigma_{0}}_{2,1}} \lesssim \Vert f_{0} -g_{0} \Vert_{B^{\sigma_{0}}_{2,1}}.
$$
Lastly, it remains to prove the local Lipschitz bound. Let $ \sigma \in [\sigma_{0}, \infty)$ and take $ \epsilon(\sigma) > 0$ so that the hypothesis of Proposition \ref{LipschitzboundN} is satisfied, and let $ f : \R^{n+1} \to \C$ and $g : \R^{n+1} \to \C$ be the solutions obtained above for initial data $ f_{0},g_{0}$ respectively. Fix any $ \sigma' \in [\sigma_{0} , \sigma]$ and assume that 
$$
\Vert f_{0} \Vert_{B^{\sigma_{0}}_{2,1}}, \Vert g_{0} \Vert_{B^{\sigma_{0}}_{2,1}} \leq \epsilon(\sigma) \text{ , and }
\Vert f_{0} \Vert_{B^{\sigma'}_{2,1}}, \Vert g_{0} \Vert_{B^{\sigma'}_{2,1}} \leq R.
$$
Then 
$$
f - g = \psi(t) e^{-\i t ( - \Delta)^{s} }(f_{0} -g_{0} ) - \i \psi(t)  \int_{0}^{t} e^{-\i ( t-r) (-\Delta)^s} ( \mathcal{N}(f) - \mathcal{N}(g)) dr
$$
and by Lemma \ref{LinearEstimates} and Proposition \ref{LipschitzboundN} we obtain 
\begin{equation*}
\begin{split}
& \Vert f -g \Vert_{F^{\sigma'}}
\\
& \leq C  \Vert f_{0} - g_{0} \Vert_{B^{\sigma'}_{2,1}} + C \Vert \mathcal{N}(f) - \mathcal{N} (g) \Vert_{N^{\sigma'}}
\\
& \leq C\Vert f_{0} - g_{0} \Vert_{B^{\sigma'}_{2,1}} + C_{\sigma'} R \lambda \Vert f - g \Vert_{F^{\sigma_{0}}} + C_{\sigma'} \epsilon(\sigma)^{2} \Vert f -g \Vert_{F^{\sigma'}}. 
\end{split}
\end{equation*}
Hence, choosing $ \epsilon(\sigma)$ appropriately small, and using \eqref{tur} we obtain
\begin{equation*}
\Vert f -g \Vert_{F^{\sigma'}} \lesssim_{R,\sigma'} \Vert f_{0} - g_{0} \Vert_{B^{\sigma'}_{2,1}}
\end{equation*}
and we can conclude by using Lemma \ref{Energy}.
\end{proof}

\appendix 
\section{Product structures}
The following two lemmata give control of the modulation of the product of two functions. Namely, under appropriate conditions, if $ f_{j_{1}}, f_{j_{2}}$ have modulation $2^{j_{1}},2^{j_{2}} $ then the product cannot have arbitrary large/small modulation.

In this section we prove several estimates regarding the size of the modulation of a product of two functions. We start with the following 
\begin{lemma}\label{prod1}
Let $ k_{1},k_{2}  \geq 0$ with $ k_{1} \leq k_{2} +10$ and $ j_{1},j_{2} \geq 0$. 
Let $ f_{k_{1}}, f_{k_{2}} $ be two functions with Fourier transform supported in $ D_{k_{1}, \leq j_{1}}$, $ D_{k_{2}, \leq j_{2}}$ respectively. Then given $ j \geq 0$ with $j\geq  k_{1}+ (2s-1)k_{2}+ 100$ and $j \geq  \max\{ j_{1} , j_{2} \} +50 $, we have 
$$
Q_{j} \left( \Delta_{ \leq k_{2} +20} ( \tilde{f}_{k_{1}} f_{k_{2}}) \right) =0
$$
where $ \tilde{f}_{k_{1}} \in \{ f_{k_{1}}, \bar{f}_{k_{1}} \}$.
\end{lemma}
\begin{proof}
Take the Fourier transform 
\begin{equation*}
\begin{split}
& \mathcal{F} \left( Q_{j} \left( \Delta_{\leq k_{2} +20} ( \tilde{f}_{k_{1}} f_{k_{2}}) \right) \right) (\xi,\tau)
\\
& = \varphi_{j} ( \vert \xi \vert^{2s} + \tau) \varphi_{\leq k_{2} +20} ( \xi) \int_{\R^{n+1}} \hat{ \tilde{f}}_{k_{1}}( \eta,r) \hat{ f}_{k_{2}}(\xi-\eta, \tau -r ) d \eta dr,
\end{split}
\end{equation*}
and in order for this to be nonzero we must have the following 
\begin{enumerate}
    \item 
    $
    2^{j-1} \leq \vert \vert \xi \vert^{2s} + \tau \vert \leq 2^{j+1}
    $
    \item  $ \vert r  \pm \vert \eta \vert^{2s}  \vert \leq 2^{j_{2} +1} \ll 2^{j}$
    \item  $ \vert \xi - \eta \vert \approx 2^{k_{2}} \geq 2^{-30} \vert \xi \vert$.
\end{enumerate}
Here the sign of $ |\eta|^{2s}$ depend on whether we have a complex conjugate or not. 
Then notice the following
\begin{equation*}
\begin{split}
& \vert \vert \xi - \eta \vert^{2s} + \tau -r \vert
\\
& \geq \vert \vert \xi \vert^{2s} + \tau \vert - \left( \vert \pm \vert \eta \vert^{2s} + r \vert \right) 
\\
& - \vert \left( \vert \xi - \eta \vert^{2s} - \vert \xi \vert^{2s}  \pm \vert \eta \vert^{2s} \right) \vert  
\\
& \geq 2^{j-1} - 2^{j-50} - \vert \left( \vert \xi - \eta \vert^{2s} - \vert \xi \vert^{2s} \pm \vert \eta \vert^{2s} \right) \vert
\end{split}
\end{equation*}
it remains to show that 
\begin{equation}
\vert \left( \vert \xi - \eta \vert^{2s} - \vert \xi \vert^{2s} \pm \vert \eta \vert^{2s} \right) \vert \ll 2^{j}.
\end{equation}
This follows easily since $ \vert \eta \vert^{2s} \leq 2^{2s k_{1}} \ll 2^{j}$ and 
\begin{equation*}
\begin{split}
& | \vert \xi - \eta \vert^{2s} - \vert \xi \vert^{2s} |
\\
& \lesssim \max\{ \vert \xi - \eta \vert^{2s-1}, \vert \xi \vert^{2s-1} \} \vert \eta \vert
\\
&  \ll 2^{j}.
\end{split}
\end{equation*}
\end{proof}

We have a similar result for lower bounds
\begin{lemma}\label{prod2}
Let $ k_{1},k_{2} \geq 0 $ with $ k_{1} \leq k_{2} -10$, and $ j_{1} \ge  k_{1} + (2s-1)k_{2}+100$. Take any $ f_{k_{1}}, f_{k_{2}} $ in $ Z_{k_{1}}, Z_{k_{2}}$ respectively. Then there exists  a constant $ C$ so that for any $j \leq j_{1} -C $ we have
\begin{equation}\label{lft}
Q_{j} \left( Q_{\geq j_{1} +100} (f_{k_{2}}) Q_{j_{1}} (\tilde{f}_{k_{1}}) \right) =0
\end{equation}

where $ \tilde{f}_{k_{1}} \in \{ f_{k_{1}} , \bar{f}_{k_{1}} \}.$
\end{lemma}
\begin{proof}
We take the Fourier transform,
\begin{equation*}
\begin{split}
& \mathcal{F} \left( Q_{j} \left( Q_{\geq j_{1} +100} (f_{k_{2}}) Q_{j_{1}} (\tilde{f}_{k_{1}}) \right) \right)(\xi,\tau)
\\
& = \varphi_{j} ( \vert \xi \vert^{2s} + \tau) \int_{\R^{n+1}} \varphi_{\geq j_{1} + 100} ( \vert \xi - \eta \vert^{2s} + \tau -r) \hat{f}_{k_{2}} ( \xi - \eta, \tau -r) \varphi_{j}( \vert \eta \vert^{2s} + r ) \hat{ \tilde{ f}}_{k_{1}}(\eta,r) d \eta dr
\end{split}
\end{equation*}
and we notice that a necessary condition for this to be nonzero is to have 
\begin{enumerate}
    \item $ \vert \vert \xi - \eta \vert^{2s} + \tau - r \vert \geq 2^{j_{1} +100}$
    \item  $ \vert r \pm \vert \eta \vert^{2s} \vert \in [ 2^{j_{1}-1}, 2^{j_{1} +1}]$
    \item $ 2^{j_{1} -10} \geq \vert \xi \vert^{2s-1} \vert \eta \vert$
    \item  $ \vert \eta \vert \ll \vert \xi \vert$.
\end{enumerate}
Here the sign of $ |\eta|^{2s}$ depend on whether we have a complex conjugate on $ f_{k_{1}}$ or not. In either case, notice 
\begin{equation*}
\begin{split}
& \vert \vert \xi \vert^{2s} + \tau \vert
\\
& \geq \vert \vert \xi - \eta \vert^{2s} + \tau -r \vert \vert - \vert r \pm \vert \eta \vert^{2s} \vert - \vert \left(  \vert \xi - \eta \vert^{2s} - \vert \xi \vert^{2s} \pm \vert \eta \vert^{2s} \right) \vert
\\
& \geq C( 2^{j_{1} + 100} - 2^{j_{1}}) 
\\
& \geq C 2^{j_{1}}.
\end{split}
\end{equation*}
\end{proof}

Lastly, below we include the calculations for Lemma \ref{stereographic stuff}. 
\begin{lemma}
Let $ L^{-1} : \C \to \S^{2}
$ be the inverse of the stereographic projection, i.e for $ z = z_{1} + \i z_{2} $ 
$$
L^{-1}(z) = \frac{1}{1+|z|^{2}} [ 2z_{1} , 2z_{2} , 1- |z|^{2} ].
$$
 Then we have
\begin{enumerate}
    \item 
    $$
    \partial_{z_{1}} L^{-1}(z) = \frac{2}{(1+|z|^{2})^{2}} [ (1+ \vert z \vert^{2} - 2 z^{2}_{1}, -2 z_{1} z_{2} , -2z_{1}]
    $$
    \item 
    $$
    \partial_{z_{2}} L^{-1}(z) = \frac{2}{(1+ |z|^{2})^{2}} [ -2 z_{1} z_{2} , (1 + |z|^{2} ) - 2 z_{2}^{2} , -2z_{2}]
    $$
    \item 
    $$
    | \partial_{z_{1}} L^{-1}(z) | = \frac{2}{1+|z|^{2}} = | \partial_{z_{2}} L^{-1}(z) |
    $$
    \item 
    $$
    \langle \partial_{z_{1}}L^{-1}(z), \partial_{z_{2}} L^{-1}(z) \rangle = 0
    $$
    \item 
    $$
    L^{-1}(z) \times \partial_{z_1} L^{-1} (z) = \partial_{z_{2}} L^{-1}(z) 
    $$
    \item 
    $$
    L^{-1}(z) \times \partial_{z_{2}} L^{-1}(z) = - \partial_{z_{1}} L^{-1}(z)
    $$
\end{enumerate}
\end{lemma}
\begin{proof}
This is direct computations.
The derivative of the first entry is 
\begin{equation*}
\begin{split}
& \partial_{z_{1}} \big[ \frac{1}{1+ | z|^{2}} 2z_{1} \big]
\\
& =  \frac{2}{1+ |z|^{2}} - \frac{ -4 z_{1}^{2}}{(1+ |z|^{2})^{2}}
\\
& = \frac{2}{ ( 1+ |z|^{2})^{2} } \big[ ( 1 + \vert z \vert^{2} ) - 2 z_{1}^{2} ].
\end{split}
\end{equation*}
For the second entry 
\begin{equation*}
\begin{split}
& \partial_{z_{1}} \big[ \frac{2 z_{2}}{1+ |z|^{2}}]
\\
& = \frac{ - 4 z_{1} z_{2}}{(1+ \vert z \vert^{2})^{2}}
\\
& = \frac{2}{ (1 + | z|^{2})^{2}} \big[ -2 z_{1} z_{2} \big]
\end{split}
\end{equation*}
for the last entry 
\begin{equation*}
\begin{split}
& \partial_{z_{1}} \big[  \frac{ 1 - |z|^{2}}{1+|z|^{2}} \big]
\\
& = \frac{ -2z_{1} }{1 + |z|^{2}} + \frac{ - 2 z_{1} ( 1- |z|^{2})}{(1+|z|^{2})^{2}}
\\
& = \frac{2}{(1+ |z|^{2})^{2}} \big[ -z_{1} ( 1 + |z|^{2} + (1 - |z|^{2})]
\\
& = \frac{2}{(1+ |z|^{2})^{2}} ( -2 z_{1}) 
\end{split}
\end{equation*}
this proves (1). Next, we prove (2) 
\begin{equation*}
\begin{split}
& \partial_{z_{2}} \big[ \frac{2z_{1}}{ 1+ |z|^{2}}   \big]
\\
& = \frac{ -4 z_{1} z_{2}}{(1+|z|^{2})^{2}}
\\
& = \frac{2}{(1+|z|^{2})^{2}} \big[ -2 z_{1} z_{2} \big]
\end{split}
\end{equation*}
for the second entry we obtain 
\begin{equation*}
\begin{split}
& \partial_{z_{2}} \big[ \frac{2z_{2}}{1+ |z|^{2}} \big]
\\
& = \frac{2}{1+ |z|^{2}} - \frac{-4 z_{2}^{2}}{(1+ | z|^{2})^{2}}
\\
& = \frac{2}{(1+ | z|^{2})^{2}} \big[ (1+ |z|^{2}) - 2 z_{2}^{2} \big]
\end{split}
\end{equation*}
lastly, for the last entry 
\begin{equation*}
\begin{split}
& \partial_{z_{2}} \big[  \frac{1 - |z|^{2}}{1+|z|^{2}} \big]
\\
& = \frac{-2 z_{2}}{ 1+ |z|^{2}} + \frac{ -2z_{2} ( 1 - | z|^{2})}{( 1+ |z|^{2})^{2}}
\\
& = \frac{2}{(1+ |z|^{2})^{2}} ( -z_{2} ( 1+ |z|^{2} + 1 - |z|^{2}))
\\
& = \frac{2}{ ( 1+ |z|^{2})^{2}} ( -2z_{2})
\end{split}
\end{equation*}
this proves (2). We proceed to prove (3) 
\begin{equation*}
\begin{split}
& | \partial_{z_{1}} L^{-1} (z) |^{2}
\\
& = \frac{4}{ (1 + |z|^{2})^{4}} \big[ (1+ |z|^{2})^{2} + 4 z_{1}^{4} - 4 z_{1}^{2} ( 1+ |z|^{2}) + 4 z_{1}^{2} z_{2}^{2} + 4 z_{1}^{2}]
\\
& = \frac{4}{(1+ |z|^{2})^{4} } ( 1 + |z|)^{2}
\\
& = \frac{4}{(1+ |z|^{2})^{2}}
\end{split}
\end{equation*}
and similarly, we obtain $ | \partial_{z_{2}} L^{-1}(z) | = \frac{2}{1+ |z|^{2}}$. This proves (3). Lastly, for (4) we have
\begin{equation*}
\begin{split}
& \langle \partial_{z_{1}} L^{-1}(z) , \partial_{z_{2}} L^{-1}(z) \rangle
\\
& = \frac{4}{(1+ |z|^{2})} \big[-2 z_{1} z_{2} (1 + |z|^{2} ) + 4 z_{1}^{2} z_{1} z_{2} - 2z_{1} z_{2} (1+ |z|^{2} ) + 4 z_{2}^{2} z_{1} z_{2} + 4 z_{1} z_{2}    \big]
\\
& = \frac{4}{(1+ |z|^{2})} \big[ -4z_{1} z_{2}( - |z|^{2} + z_{2}^{2} + z_{1}^{2})    \big]
\\
& = 0 
\end{split}
\end{equation*}
this proves (4). Next, we proceed to (5)
\begin{equation*}
\begin{split}
& L^{-1}(z) \times \partial_{z_{1}} L^{-1}(z)
\\
& = \frac{2}{(1 + |z|^{2})^{3}} \big[ -4 z_{1}z_{2} + 2z_{1} z_{2} ( 1- |z|^{2}) , 4z_{1}^{2} |z|^{2} 
+ (1 -|z|^{2})(1+|z|^{2})
\\
&, -4 z_{1}^{2} z_{2} - 2z_{2} - 2 z_{2} |z|^{2} + 4 z_{2} z_{1}^{2}      \big]
\end{split}
\end{equation*}
We simplify each entry, for the first entry 
\begin{equation*}
\begin{split}
& \frac{ 2 }{ (1 + |z|^{2})^{3}} ( -4 z_{1} z_{2} + 2z_{1} z_{2} ( 1 - |z|^{2}))
\\
& = \frac{ 2 }{ (1 + |z|^{2})^{3}} ( -2 z_{1} z_{2} ( 1  + |z|^{2}))
\\
& = \frac{2}{(1+ |z|^{2})^{2}} ( -2 z_{1} z_{2})
\end{split}
\end{equation*}
which agrees with the first entry of $ \partial_{z_{2}}L^{-1}(z)$.
Next, we do the second entry 
\begin{equation*}
\begin{split}
& \frac{2}{(1+ |z|^{2})^{3}} \big[ 4 z_{1}^{2} + ( 1 - | z|^{2})(1+|z|^{2} - 2 z_{1}^{2}) \big]
\\
& = \frac{2}{(1+ |z|^{2})^{3}} \big[ 2z_{1}^{2}(1+ | z|^{2}) + ( 1 -|z|^{2} )( 1+ |z|^{2} ) \big]
\\
& = \frac{2}{(1+ |z|^{2})^{2}} ( 2 z_{1}^{2} + (1 - |z|^{2}) )
\\
& = \frac{2}{(1+ |z|^{2})^{2}}( 1+ |z|^{2} - 2z_{2}^{2})
\end{split}
\end{equation*}
which agrees with the second entry of $ \partial_{z_{2}} L^{-1}(z)$. Lastly, we do the third entry 
\begin{equation*}
\begin{split}
& \frac{2}{(1+|z|^{2})^{3}} ( -4 z_{1}^{2} z_{2} - 2z_{2} - 2 z_{2} |z|^{2} + 4 z_{2}z_{1}^{2} )
\\
& = \frac{2}{(1+|z|^{2})^{3}} ( - 2z_{2} (1+|z|^{2}))
\\
& \frac{2}{(1+|z|^{2})^{2}} ( - 2z_{2}) 
\end{split}
\end{equation*}
which agrees with the third entry of $\partial_{z_{2}}L^{-1}(z)$. This proves (5). The calculations for (6) are almost identical. 
\end{proof}

\bibliographystyle{abbrv}
\bibliography{bib}

\end{document}